\RequirePackage{rotating}
\documentclass[12,ptoneside]{amsart}

\usepackage{amsmath,amstext,amsthm,amssymb,bbm,amsbsy}
\usepackage{mathtools}
\usepackage[numbers]{natbib}
\usepackage{enumerate}
\usepackage{bm}
\usepackage[ruled,linesnumbered]{algorithm2e}
\usepackage{multirow}
\usepackage{tabularx}
\usepackage{makecell}
\usepackage{booktabs}
\usepackage{xcolor}

\usepackage[letterpaper, margin=1in]{geometry}

\usepackage{subfig}
\usepackage{enumitem}
\usepackage{rotating}

\usepackage{amsfonts}
\usepackage{graphicx}
\usepackage{epstopdf}
\usepackage{algorithmic}
\ifpdf
  \DeclareGraphicsExtensions{.eps,.pdf,.png,.jpg}
\else
  \DeclareGraphicsExtensions{.eps}
\fi

\usepackage{cleveref}

\newcommand{\update}[1]{{\color{black} #1}}
\newcommand{\updatenew}[1]{{\color{black} #1}}

\setlist[enumerate]{leftmargin=.25in}
\setlist[itemize]{leftmargin=.25in}

\newtheorem{prop}{Proposition}[section]
\newtheorem{lem}{Lemma}[section]
\newtheorem{remark}{Remark}[section]
\newtheorem{cor}{Corollary}[section]

\DeclareMathOperator*{\argmin}{arg\,min}
\DeclareMathOperator*{\argmax}{arg\,max}

\newcommand{\normiii}[1]{{\left\vert\kern-0.25ex\left\vert\kern-0.25ex\left\vert #1 
    \right\vert\kern-0.25ex\right\vert\kern-0.25ex\right\vert}}

\def\Rn{\mathbb{R}^n}
\def\RN{\mathbb{R}^N}

\def\R{\mathbb{R}}
\def\gmRn{\Gamma_0(\Rn)}
\def\conv{\square}
\def\alplimj{\alpha_{j,\infty}}
\def\vlimj{v_{j,\infty}}
\def\N{\mathbb{N}}
\def\dom{\mathrm{dom}~}
\def\ri{\mathrm{ri\ }}
\def\epi{\mathrm{epi~}}

\begin{document}

\title{On Decomposition Models in Imaging Sciences and Multi-time Hamilton-Jacobi Partial Differential Equations}

\author{J\'er\^ome Darbon}
\address{Department of Applied Mathematics, Brown University, Providence, RI}
\email{jerome\_darbon@brown.edu}
\author{Tingwei Meng}
\address{Department of Applied Mathematics, Brown University, Providence, RI}
\email{tingwei\_meng@brown.edu}

\thanks{The authors are listed in alphabetical order. This work was funded by NSF 1820821}

\maketitle

\begin{abstract}
This paper provides new theoretical connections between multi-time Hamilton-Jacobi partial differential equations and variational image decomposition models in imaging sciences. We show that the minimal values of these optimization problems are governed by multi-time Hamilton-Jacobi partial differential equations. The minimizers of these optimization problems can be represented using the momentum in the corresponding Hamilton-Jacobi partial differential equation.  Moreover, variational behaviors of both the minimizers and the momentum are investigated as the regularization parameters approach zero. In addition, we provide a new perspective from convex analysis to prove the uniqueness of convex solutions to Hamilton-Jacobi equations. Finally, we consider image decomposition models that do not have unique minimizers and we propose a regularization approach to perform the analysis using multi-time Hamilton-Jacobi partial differential equations.
\end{abstract}

\section{Introduction}
In the late 20th century, the Hamilton-Jacobi (HJ) equation was widely studied in the field of partial differential equations (PDEs). To be specific, the solution $S(x,t)$ defined for $x\in\Rn$, $t\geq 0$ satisfies the following Cauchy problem
\begin{equation*}
\begin{cases}
\frac{\partial S(x,t)}{\partial t} + H(x,t,S(x,t),\nabla_x S(x,t)) = 0, &x\in \Rn, t>0;\\
S(x,0) = J(x), &x\in \Rn,
\end{cases}
\end{equation*}
where $H$ is the Hamiltonian and $J$ is the initial data. When the Hamiltonian only depends on the spatial gradient $\nabla_x S(x,t)$, under some regularity and convexity assumptions, the solution is given by the Hopf formula or Lax formula \cite{bardi1984hopf,hopf1965generalized}
\begin{align*}
S(x,t) &= \sup_{p\in \R^n} \langle p, x\rangle - J^*(p) - tH(p) & \text{(Hopf formula)}\\
    &= \inf_{u\in\R^n} J(u) + t H^*\left(\frac{x - u}{t}\right) & \text{(Lax formula)}
\end{align*}
where $J^*$ and $H^*$ are the Legendre transform of the functions $J$ and $H$, respectively. From the physics point of view, HJ PDE describes the movement of a particle in a physics model whose energy function is given by the Hamiltonian $H$. To be specific, the variables $x$ and $t$ are the current position and time of the particle. The characteristic line of the PDE gives the trajectory of the particle. The momentum is given by the spatial gradient $\nabla_x S(x,t)$ which coincides with the maximizer in the Hopf formula. The velocity is given by $\frac{x-u}{t}$ where $u$ is the minimizer in the Lax formula.

We refer the readers to the review paper \cite{crandall1992user} for thorough details and \cite{Darbon1,imbert:2001:jnca} for connections between convex analysis and HJ equations.
An extension of this PDE is to consider the time variable $t$ in a higher dimensional space $\R^N$, in which case the PDE system is called the multi-time Hamilton-Jacobi equation, first discussed by Rochet from an economic point of view \cite{rochet1985taxation}. Later, Lions and Rochet \cite{lions1986hopf} considered the multi-time HJ equations when the Hamiltonians are convex functions which only depend on the momentum. They proposed the generalized Hopf formula by writing it as the composition of several semigroups of the corresponding single-time HJ operators.
Following their work, several existence and uniqueness results \cite{barles2001commutation,cardin2008commuting,motta2006nonsmooth,plaskacz2002oleinik, tho2005hopf} were provided in more general cases, for example, when the Hamiltonians have spatial or time dependence.

It is well known that the HJ equation has a deep relationship with optimal control \cite{bressan2007introduction} and differential games \cite{evans1983differential,souganidis1985max}. 
Later, Darbon \cite{Darbon1} provided a representation formula for the minimizers of a specific kind of optimization problem, which relates the minimizers to the spatial gradients of the solutions to the HJ equations. As we will see below, many models in imaging sciences can be viewed from a perspective of HJ PDEs.
Following that work, we generalize the results to multi-time HJ equations and a larger set of optimization problems, including the decomposition models in image processing.

\bigbreak
In the past few decades, many decomposition models have been proposed in image processing. These models are applied to different practical problems, such as inpainting \cite{Bertalmio2003,ELAD2005340}, image classification \cite{AUJOL20061004}, and road detection \cite{Gilles2010}. Here, we give a brief overview of convex variational models in this area. There are many models that cannot be fully listed here, for which we refer the readers to \cite{chan2005image,gilles2009image}.

The basic idea of image decomposition is to regard an image $x$ as a summation of several components $\{u_j\}$, and solve the following minimization problem:
\begin{equation}\label{eqt:imagemodel}
\argmin_{u_0 + \cdots + u_N = x} f_0(u_0) + \sum_{j=1}^N \lambda_j f_j(u_j).
\end{equation}
Here, each function $f_j$ is designed to characterize the corresponding component $u_j$. One may tune the parameters $\{\lambda_j\}$ to put emphasis on different components. There are many celebrated decomposition models in the literature of imaging sciences.
In the introduction we mention the continuous versions of the models, while later in the main part of this paper we will work with their discrete versions.
The first widely used decomposition model is the Rudin-Osher-Fatemi (ROF) model, proposed in \cite{rudin1992nonlinear}, which applies the total variation (TV) semi-norm and $\|\cdot\|_{L^2}^2$ to recognize the geometry and noise in an image, respectively. 
In the continuous setting, for any function $u\in L^1(\Omega)$ and $\Omega \subset \R^2$, the TV semi-norm of $u$ is defined by
\begin{equation*}
    \|u\|_{TV} := \sup\left\{\int_{\Omega}u(x){\rm div} \phi(x) dx \colon \phi \in C_c^1(\Omega, \R^2), \|\phi\|_{L^\infty} \leq 1\right\}.
\end{equation*}
Here and after in the introduction, the derivatives and divergence are in the distribution sense. The space $BV(\Omega)$ is the space containing all functions of bounded variation, defined by
\begin{equation*}
    BV(\Omega) = \{u\in L^1(\Omega)\colon \|u\|_{TV} < +\infty\}.
\end{equation*}
Under these settings, the ROF model solves the following problem
\begin{equation*}
\argmin_{u\in BV(\Omega)} \|u\|_{TV} + \frac{1}{2\lambda}\|x-u\|_{L^2}^2.
\end{equation*}
The mathematical analysis for the ROF model is provided in \cite{Acar1994Analysis, Allard2008TotalI, Allard2008TotalII, Allard2009TotalIII, aubert.02.book, Brezis2019Remarks, Burger2004Convergence, Casas1999Regularization, Caselles2007Discont, Caselles2011Regularity, Caselles2015Total, Chambolle2010Introduction, Chambolle2016Total, Chambolle1997, Chavent1997Regularization, Choksi2014Remarks, Dobson1996Analysis, Hintermuller2018Function, HINTERMULLER2017Analytical, Nikolova2004Weakly, Ring2000Structural, Valkonen2015Jump, Vogel2002Computational}.
Later, Meyer \cite{meyer2001oscillating} pointed out the disadvantage of $\|\cdot\|_{L^2}^2$ in capturing oscillating patterns. In order to overcome this disadvantage, he suggested using the norm in either of the three spaces $E,F,G$ to replace it, \update{where these three spaces are defined as follows. We use the notations of Meyer to describe these spaces \cite{meyer2001oscillating}. First, define the space of functions of bounded mean oscillation ($BMO$) by
\begin{equation*}
    \begin{split}
        BMO:= \left\{f\in L_{loc}^1(\R^n):\ \sup\left\{\frac{1}{|Q|}\int_Q \left|f(x) - f_Q\right|dx:\ Q \text{ is any ball in }\R^n\right\} <+\infty\right\}\\
        \text{ where the symbol }f_Q\text{ is defined by } f_Q:= \frac{1}{|Q|}\int_Qf(x) dx,
    \end{split}
\end{equation*}
and the homogeneous Besov space $\dot{B}_1^{1,1}$ by
\begin{equation*}
\begin{split}
    \dot{B}_1^{1,1} := \Big\{f\in L^{\frac{n}{n-1}}(\R^n):\ &\sum_{j\in \mathbb{Z}}\sum_{k\in \mathbb{Z}^n} |c(j,k)| 2^{j(1-n/2)} < +\infty, \\
    &\text{ where } \{c(j,k)\} \text{ are the wavelet coefficients of } f \Big\}.
    \end{split}
\end{equation*}
Let $\dot{B}_\infty^{-1,\infty}$ be the dual space of $\dot{B}_1^{1,1}$.}
Then, define $E, F, G$ by $E := \dot{B}_\infty^{-1,\infty}$, $F := div(BMO)$ and $G:= div(L^\infty)$.
To be specific, the space $G$ and $G-$norm are defined as follows
\begin{equation*}
\begin{split}
    &G:= \{f = \partial_1 g_1 + \partial_2 g_2:\ g_1,g_2\in L^\infty(\R^2)\},\\
    &\|f\|_G: = \inf\{\|(g_1^2+g_2^2)^{1/2}\|_{L^\infty}:\ f = \partial_1 g_1 + \partial_2 g_2\}.
\end{split}
\end{equation*}
The space $F$ is similarly defined by replacing the space $L^\infty$ in the above definition with the $BMO$ space. 
\update{The corresponding models proposed by Meyer are stated as follows}
\begin{equation}\label{eqt:intro_meyer}
    \begin{split}
        &\update{\argmin_{u\in BV(\Omega)} \|u\|_{TV} + \lambda \|x-u\|_X, \quad \text{where the space $X$ can be $E$, $F$ or $G$.}}
    \end{split}
\end{equation}
For mathematical analysis of these models, we refer the readers to \cite{GARNETT200725,Gilles2010,Le2005BMO}. In \cite{GARNETT200725}, the space $E$ is also generalized to any homogeneous Besov space $\dot{B}_p^{\alpha-2,q}$, where $p,q\in[1,+\infty]$ and $\alpha \in (0,2)$.
However, Meyer's models are hard to solve numerically. There are mainly two approaches to numerically solve the model with $G-$norm. The first approach is approximating $L^\infty$ in the definition of $G$ by $L^p$ \cite{Vese2004}.
Osher et al. \cite{OSV} proposed an equivalent formulation called OSV when $p=2$. In a word, OSV uses the square of $H^{-1}-$norm instead of $G-$norm. 
\update{To be specific, the OSV model solves}
\begin{equation*}
\update{
    \argmin_{u\in BV(\Omega)} \|u\|_{TV} + \lambda \int_\Omega \left|\nabla (\Delta)^{-1}(x-u)\right|^2 dxdy.}
\end{equation*}
The other approach called $A^2BC$ model is proposed by Aujol et al. \cite{aujol2003image,aujol2005image}, replacing the $G-$norm with the indicator function of balls in the space $G$. \update{In other words, it solves the following problem
\begin{equation}\label{eqt:intro_a2bc}
    \argmin_{u\in BV(\Omega)} \|u\|_{TV} + I\{ \|x-u\|_G\leq \mu\},
\end{equation}
where $I\{\cdot\}$ denotes the indicator function whose definition will be given in \cref{sec:bkgd}.
It is shown that this $A^2BC$ model gives the solution to Meyer's model \cref{eqt:intro_meyer} with $X = G$ when the parameter $\mu$ is appropriately chosen. 
In practice, they use a Moreau-Yosida type approximation and solve the following problem instead
\begin{equation}\label{eqt:intro_a2bc2}
    \argmin_{u\in BV(\Omega),\ v\in G} \|u\|_{TV} + I\{ \|v\|_G\leq \mu\} + \frac{1}{2\lambda} \|x-u-v\|_{L^2}^2, 
\end{equation}
This regularized model converges to \cref{eqt:intro_a2bc} as the parameter $\lambda$ approaches zero.}
Moreover, it is easy to implement using Chambolle's projection method \cite{Chambolle2004}. 
Similarly, in \cite{aujol2005dual}, the indicator function of the $E-$ball is used to replace the $E-$norm, \update{which provides a similar numerical implementation approach to the Meyer's model \cref{eqt:intro_meyer} with $X = E$}.

In the above models, an image is decomposed into a geometrical part and an oscillating part. However, for a noisy image, the oscillating part may contain both the texture in the original image and the noise. To split these two parts, a $u+v+w$ model is proposed in \cite{aujol2005dual}, which constrains the $G-$norm of the texture part and the $E-$norm of the noisy part. Later, Gilles \cite{Gilles2007} modified the $u+v+w$ model with a coefficient assigned to each pixel to smoothly indicate whether it is in texture or noise. He also modified the $A^2BC$ model by requiring the $G-$norm of the noise to be much smaller than the $G-$norm of the texture.
In \cite{AUJOL2006916,Duval2009,Duval2010}, the authors extended some of the abovementioned models, which are originally proposed for gray-scale images, to color images.
Besides, there are many other functions used in image decomposition. For example, the $L^1-$norm \cite{Alliney1997,aujol2006structure,Chan2005Aspects,Nikolova2004} is used to promote sparsity or remove salt and pepper noise. In \cite{Aujol2006Constrained,aujol2006structure}, the quadratic form $\langle \cdot, K \cdot\rangle$, where $K$ is a linear symmetric positive operator, is used for adaptive kernel selection of the texture component. Note that this quadratic form generalizes the $L^2$ term in ROF and the $H^{-1}$ term in OSV.

The previous work \cite{Darbon1} clarifies the relationship between single-time HJ equations and decomposition models with two terms (i.e. $N=1$ in \cref{eqt:imagemodel}), such as the ROF model, Meyer's models and some of their variations.
However, as mentioned above, there are many other models handling three or more components. Also, in practice, one may modify a model by adding a quadratic term for numerical consideration, \update{such as in \cref{eqt:intro_a2bc2}}. This kind of modification is applied to most of the above models. As a result, the objective function in the numerical implementation actually contains three or more terms. 
On the other hand, new models can be constructed by regarding the functions mentioned above as building blocks and combining them together. For instance, the morphological component analysis \cite{Fadili:2010:ieee,Starck2004,Starck2005} combines ROF model and $L^1$ minimization for the coefficients with respect to two sets of dictionaries chosen for the representation of texture and geometry. Another example is \cite{CHAN2007464}, which adds a higher order term $\alpha \|\Delta v\|_{L^2}^2$ to the models introduced above, in order to reduce the staircase effect.
\update{Actually, the higher order terms in image processing are widely studied in the literature. Two important models are the TV-TV$^2$ infimal convolution model \cite{Chambolle1997} and the Total Generalized Variation (TGV) model \cite{Bredies2010Total}. In fact, after discretization, the higher order linear operators are discretized using some matrices. In other words, the results in this paper can be applied to the discrete models with higher order terms by regarding them as matrix multiplication. In conclusion, }it is valuable to generalize the previous work \cite{Darbon1} and provide a framework to analyze the models involving more than two components. \update{Also, our proposed framework is suitable for a large class of discrete decomposition models in imaging sciences, even including some models containing higher order terms.}

\bigbreak

Now, we briefly introduce the intuition and the basic setup for our framework and demonstrate the idea using some experimental results of the \update{discrete} $A^2BC$ model. In general, for a discrete decomposition model \cref{eqt:imagemodel}, an image is regarded as a vector $x\in \Rn$, where $n$ is the number of pixels.
If we can relate each $f_j$, $j\geq 1$, to a Hamiltonian and $f_0$ to an initial function, then the minimal value, regarded as a function of the input data $x$ and the parameters $\{\lambda_j\}$, relates to the solution of the corresponding multi-time HJ equation. Here, the parameters $\{\lambda_j\}$ are regarded as time variables. 

\begin{figure}[htbp]
\centering
\subfloat[]{\label{fig:stripa}\includegraphics[width = 0.3\textwidth]{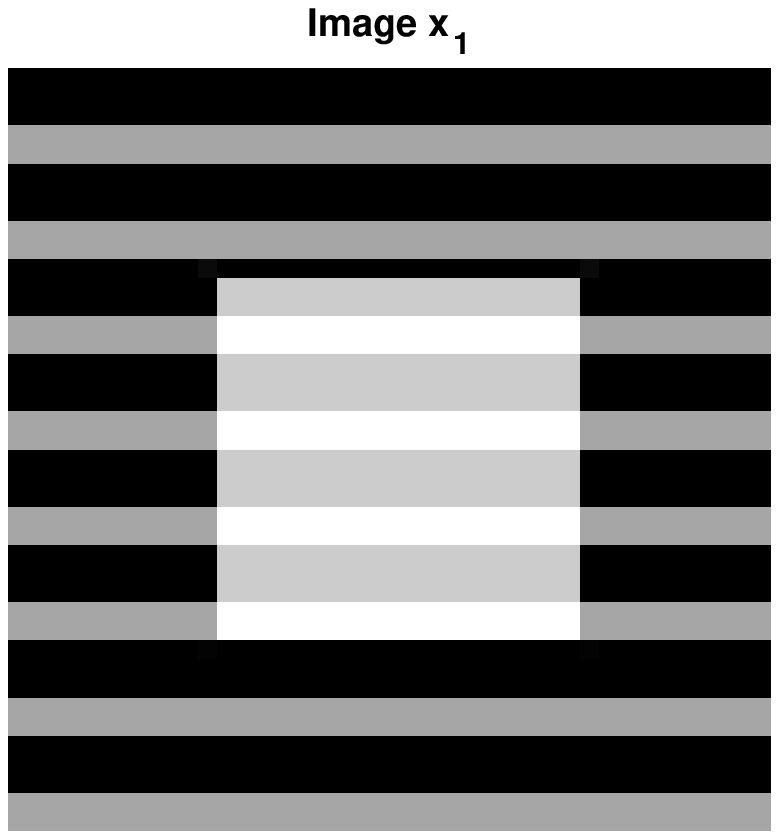}}
\subfloat[]{\label{fig:stripb}\includegraphics[width = 0.3\textwidth]{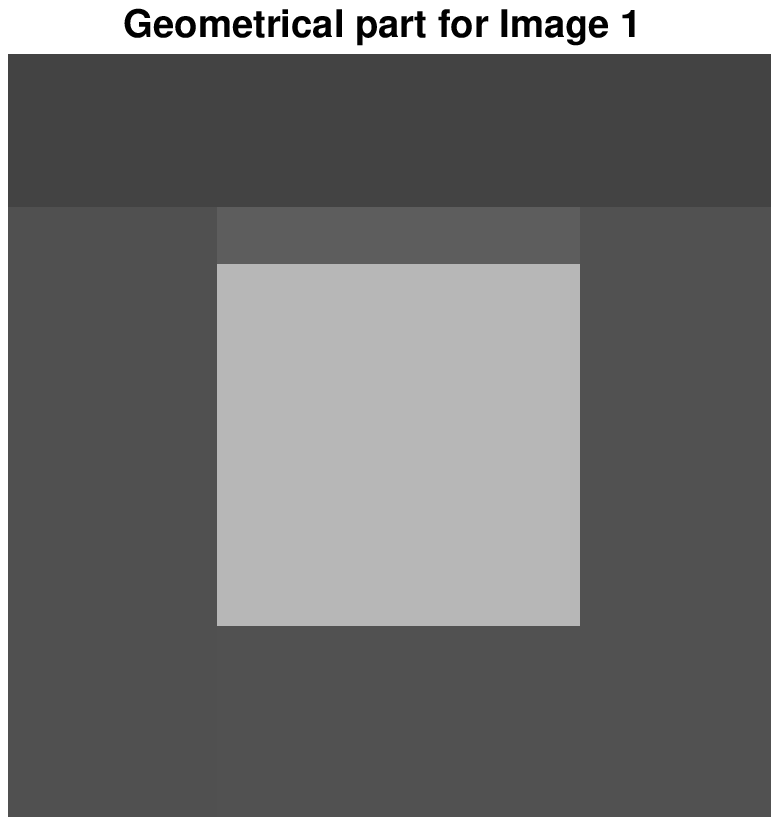}}
\subfloat[]{\label{fig:stripc}\includegraphics[width = 0.3\textwidth]{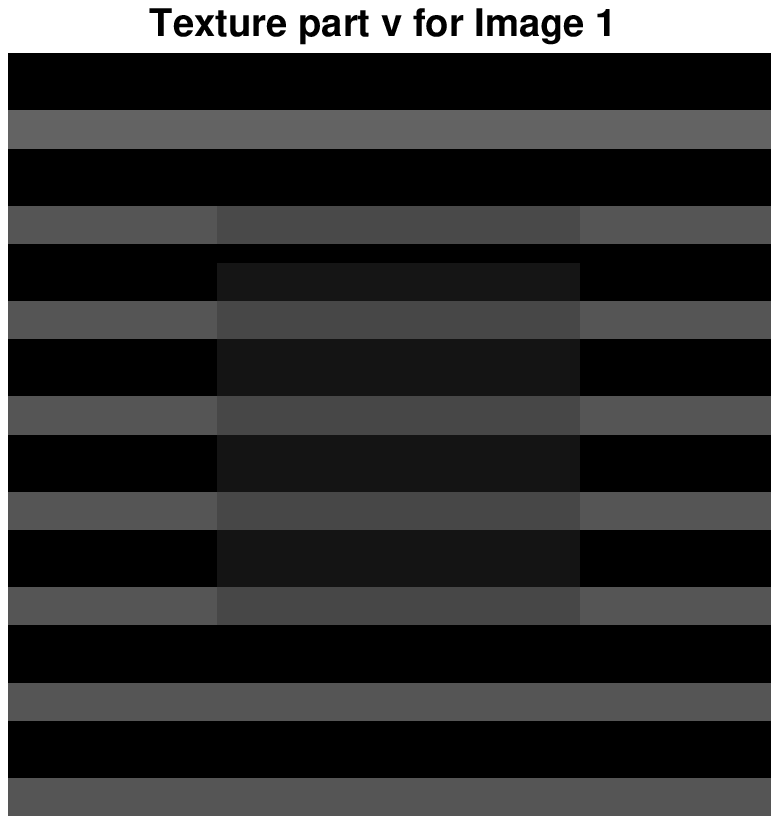}}\\
\subfloat[]{\label{fig:stripd}\includegraphics[width = 0.3\textwidth]{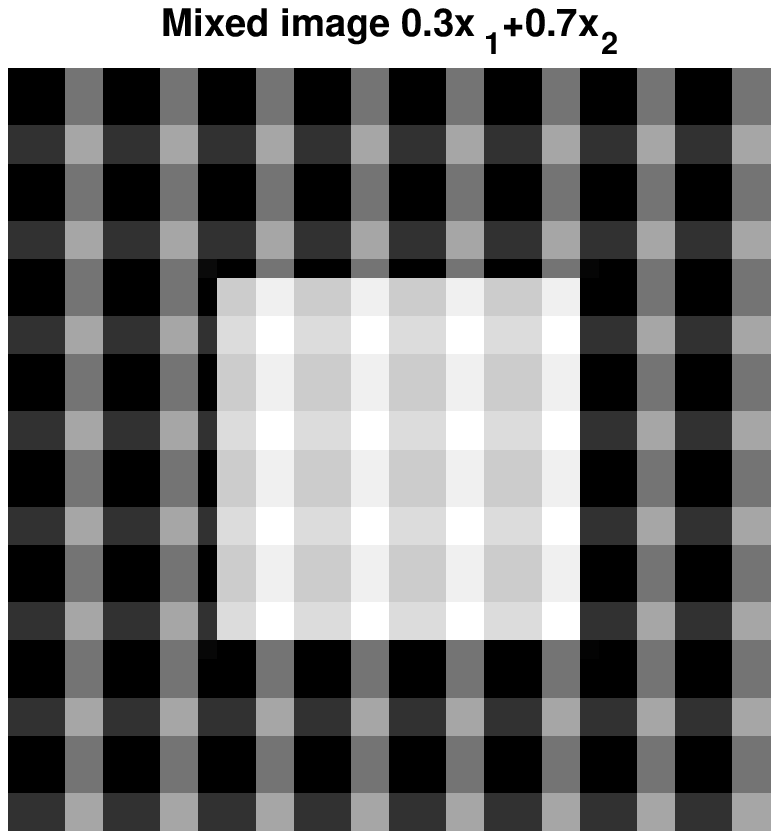}}
\subfloat[]{\label{fig:stripe}\includegraphics[width = 0.3\textwidth]{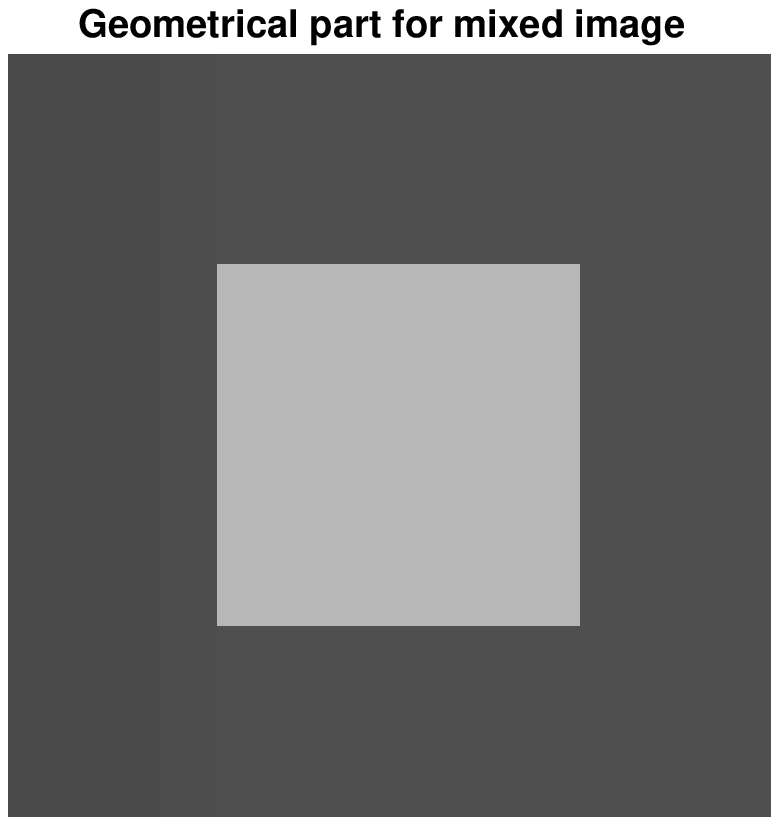}}
\subfloat[]{\label{fig:stripf}\includegraphics[width = 0.3\textwidth]{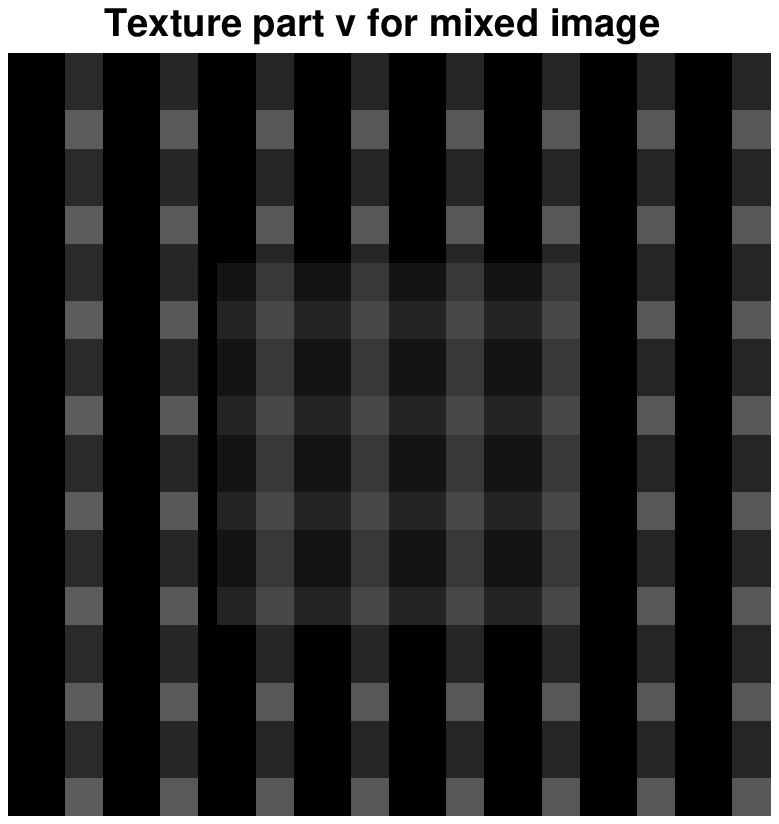}}
\caption{The $A^2BC$ model is applied to an artificial image. The original image $x_1$ and the corresponding minimizers $u,v$ are shown in (a)-(c). The convex combination $0.3 x_1 + 0.7 x_2$ of $x_1$ and its rotation $x_2$ is shown in (d), whose minimizers are shown in (e)-(f).}
\label{fig:strip}
\end{figure}

For example, the \update{discrete} $A^2BC$ model solves the following optimization problem:
\begin{equation}\label{eqt:AABFC}
S(x, \mu, \lambda) := \min_{u,v\in \Rn} J(u) + J^*\left(\frac{v}{\mu}\right) + \frac{1}{2\lambda} \|x-u-v\|_2^2.
\end{equation}
The desired quantities are the minimizers, denoted as $u(x, \mu, \lambda)$ and $v(x, \mu, \lambda)$.
\update{Here, the discrete total variation semi-norm $J:\ \R^{m_1\times m_2}\to \R$ is defined as follows
\begin{equation}\label{eqt:discreteTV}
    J(u):= \sum_{i=1}^{m_1-1}\sum_{j=1}^{m_2-1} |u_{i+1,j} - u_{i,j}| + |u_{i,j+1} - u_{i,j}|.
\end{equation}
\updatenew{In this paper, we identify the space $\R^{m_1\times m_2}$ containing all matrices with $m_1$ rows and $m_2$ columns with the Euclidean space $\R^n$ where $n=m_1m_2$.
The discrete total variation $J$ defined above is the anisotropic version, which will be used in this paper.}
Its Legendre transform $J^*$ is the indicator function of the unit ball in the dual space. \updatenew{To be specific, let $\|\cdot\|_G$ be the dual norm of $J$, which is given by 
\begin{equation*}
\begin{split}
    \|v\|_G = \inf\Biggl\{&\sup_{\substack{1\leq i\leq m_1\\ 1\leq j\leq m_2}} \sqrt{(g_{i,j})^2 + (h_{i,j})^2} \colon \quad v_{i,j} = g_{i,j} - g_{i-1,j} + h_{i,j} - h_{i, j-1}, \quad \\
    &\quad \quad g_{0,j} = g_{m_1,j} = h_{i,0} = h_{i,m_2}=0, \  g_{i,j}, h_{i,j}\in \R \quad \forall \ 1\leq i\leq m_1, 1\leq j\leq m_2\Biggr\}.
\end{split}
\end{equation*}
Then, we have $J^*(p) = I\{\|p\|_G\leq 1\}$ for any $p\in \R^n$ where $I\{\cdot\}$ denotes the indicator function.} Notice that any indicator function is invariant under multiplication with a positive constant, then we have $\mu J^* = J^*$.}
Hence, the above optimization problem is equivalent to 
\begin{equation*}
S(x, \mu, \lambda) = \min_{u,v\in \Rn} J(u) + \mu J^*\left(\frac{v}{\mu}\right) + \frac{\lambda}{2} \left\|\frac{x-u-v}{\lambda}\right\|_2^2.
\end{equation*}
We shall see that such a representation for $S$ will allow us to show that $S$ satisfies the following multi-time HJ equation
\begin{equation*}
    \begin{cases}
    \frac{\partial S(x, \mu, \lambda)}{\partial \mu} + J(\nabla_x S(x, \mu, \lambda)) = 0, &x\in\Rn, \mu>0, \lambda>0;\\
    \frac{\partial S(x, \mu, \lambda)}{\partial \lambda} + \frac{1}{2}\|\nabla_x S(x, \mu, \lambda)\|_2^2 = 0, &x\in\Rn, \mu>0, \lambda>0;\\
    S(x,0,0) = J(x), &x\in \Rn.
    \end{cases}
\end{equation*}

\begin{figure}[tbp!]
    \centering
    \subfloat[]{\label{fig:stripg}\includegraphics[width = 0.5\textwidth]{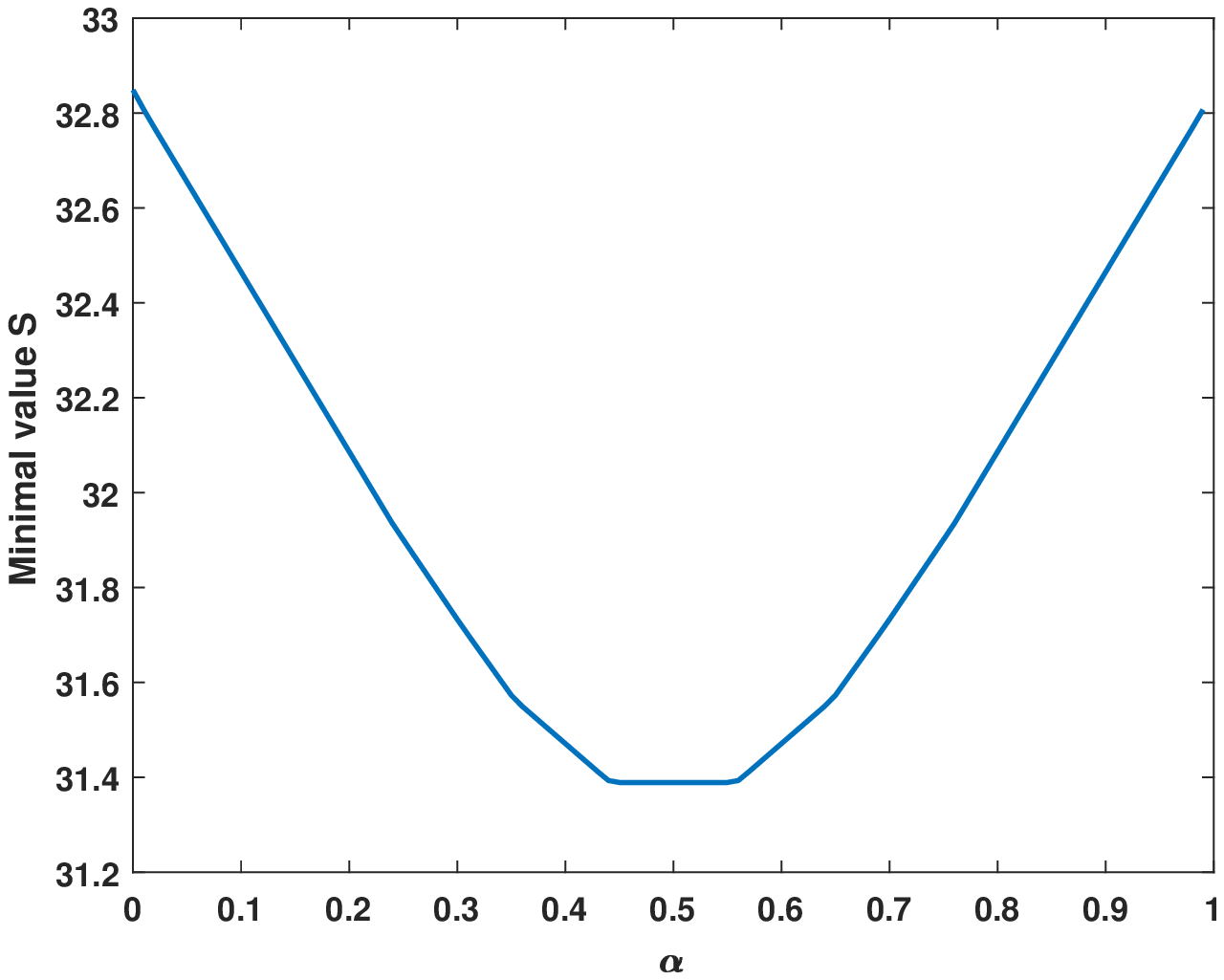}}
\subfloat[]{\label{fig:striph}\includegraphics[width = 0.5\textwidth]{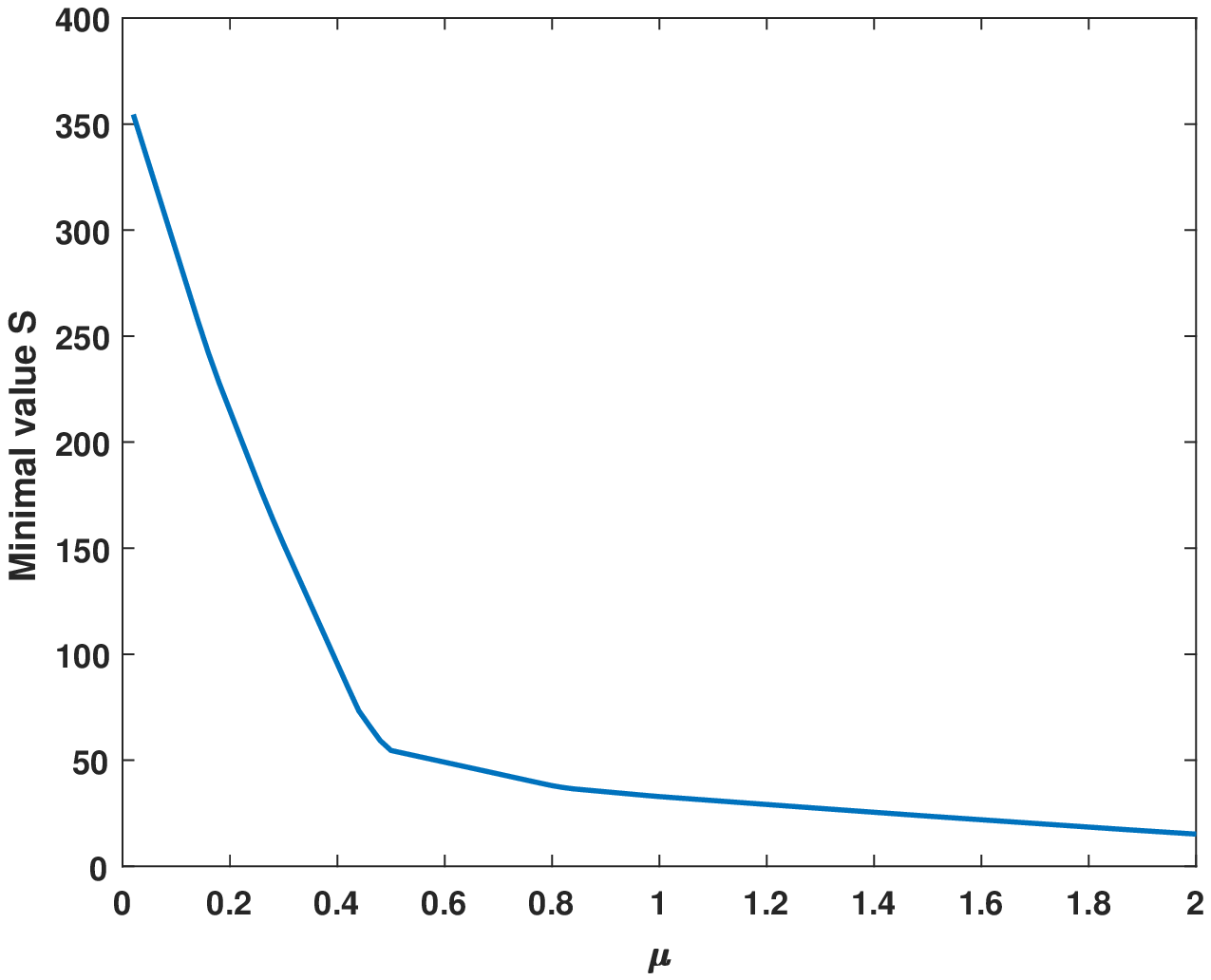}}\\
\subfloat[]{\label{fig:stripi}\includegraphics[width = 0.5\textwidth]{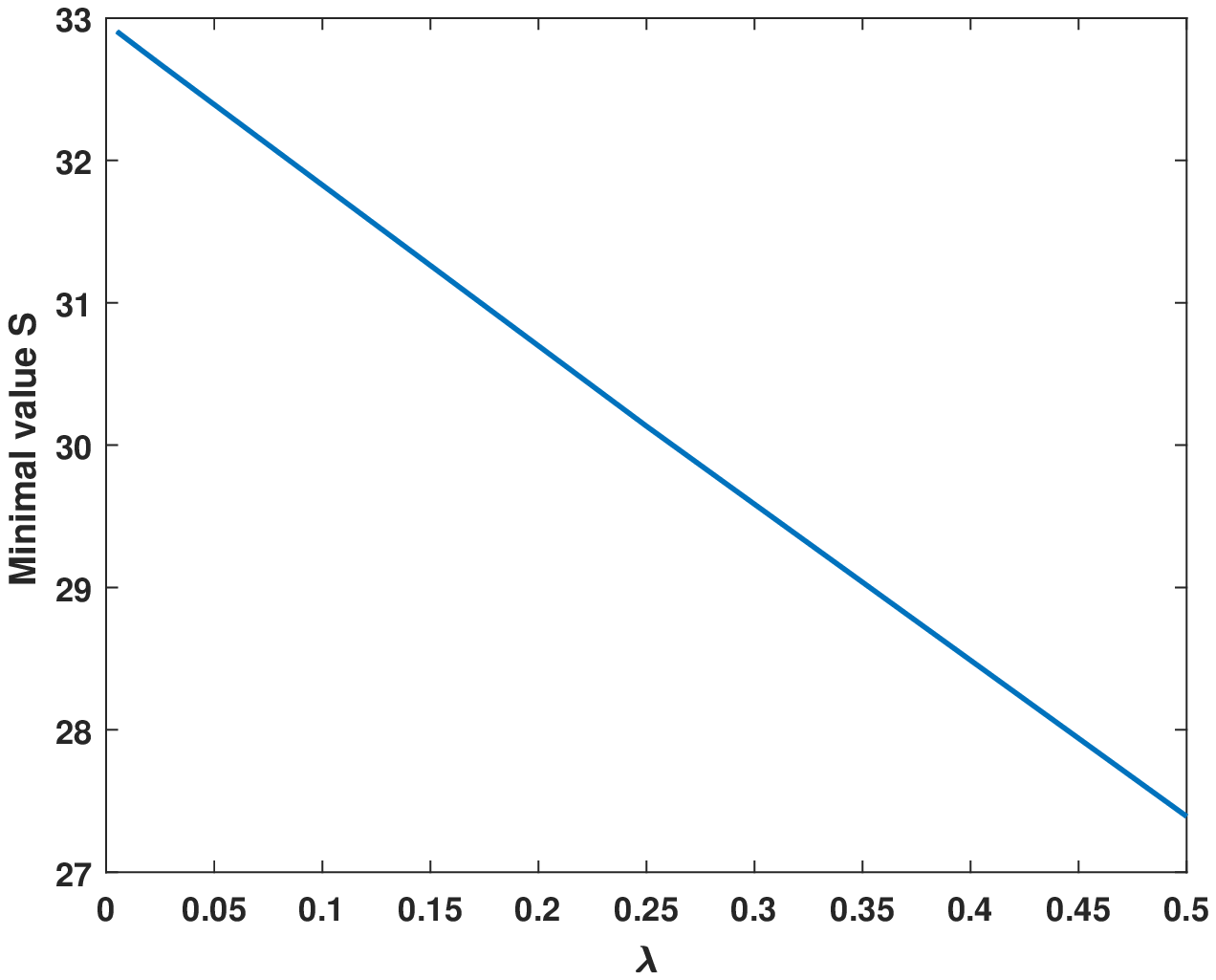}}
    \caption{The graphs of the minimal values $S$ with respect to the variables $\alpha$, $\mu$ and $\lambda$ in the first example are shown in (a)-(c), respectively. \updatenew{To be specific, (a) shows the function $\alpha \mapsto S(\alpha x_1 + (1-\alpha) x_2, \alpha \mu_1 + (1-\alpha) \mu_2, \alpha \lambda_1 + (1-\alpha) \lambda_2)$, 
    (b) shows the function $\mu\mapsto S(x_1, \mu, \lambda_1)$, and (c) shows the function $\lambda \mapsto S(x_1, \mu_1, \lambda)$.}}
    \label{fig:strip2}
\end{figure}

\begin{figure}[htbp!]
\centering
\subfloat[]{\label{fig:noisy1}\includegraphics[width = 0.3\textwidth]{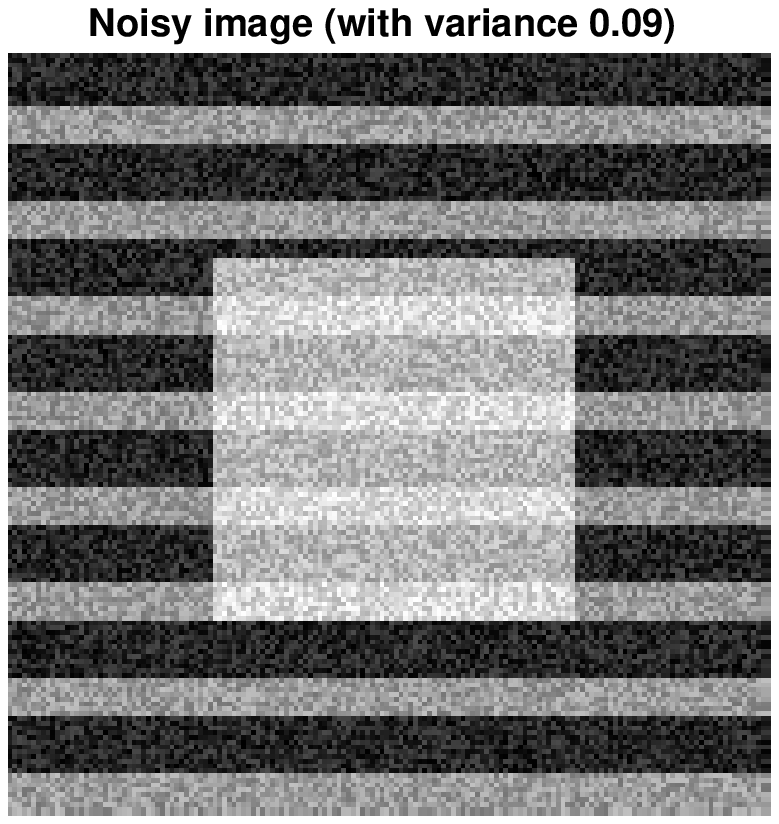}}
\subfloat[]{\label{fig:noisy2}\includegraphics[width = 0.3\textwidth]{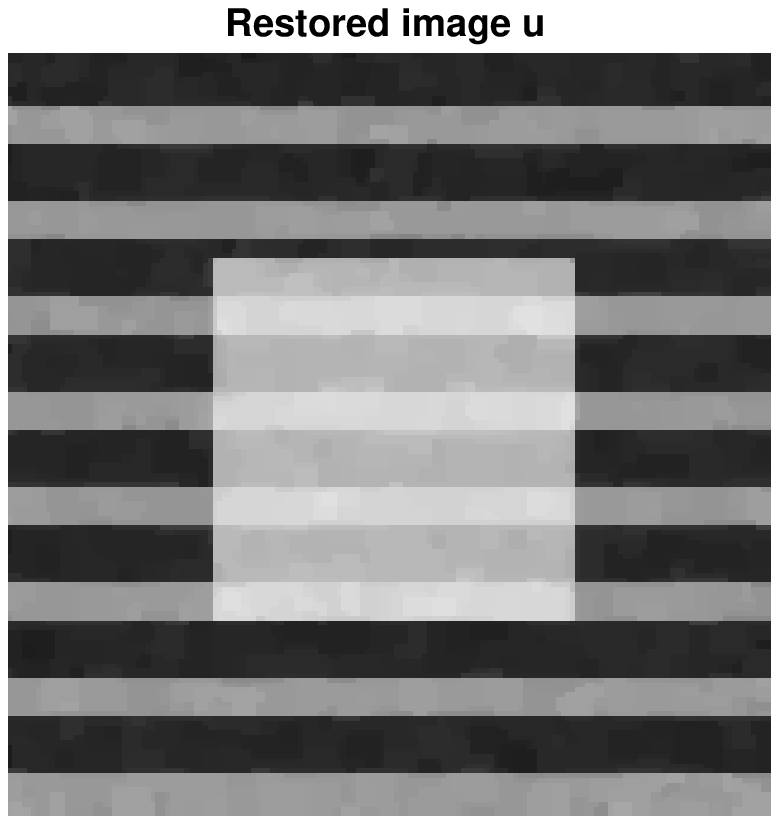}}
\subfloat[]{\label{fig:noisy3}\includegraphics[width = 0.3\textwidth]{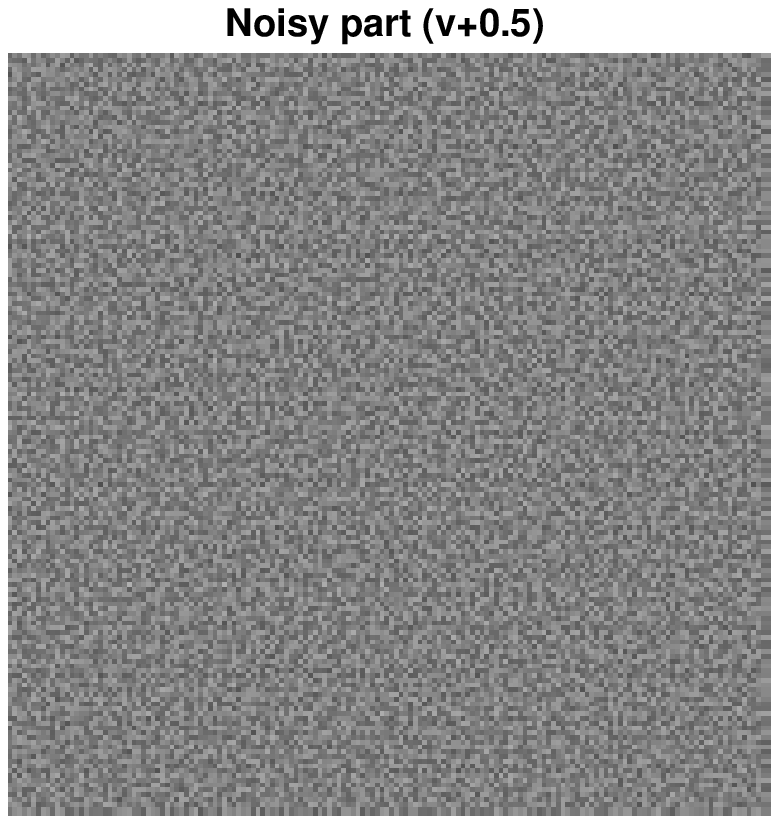}}
\caption{The $A^2BC$ model is applied to the noisy test image shown in (a). The corresponding minimizers $u$ and $v$ are shown in (b) and (c), respectively.}
\label{fig:noisy_strip}
\end{figure}

\begin{figure}[htbp!]
\centering
\subfloat[]{\label{fig:barbara_full1}\includegraphics[width = 0.3\textwidth]{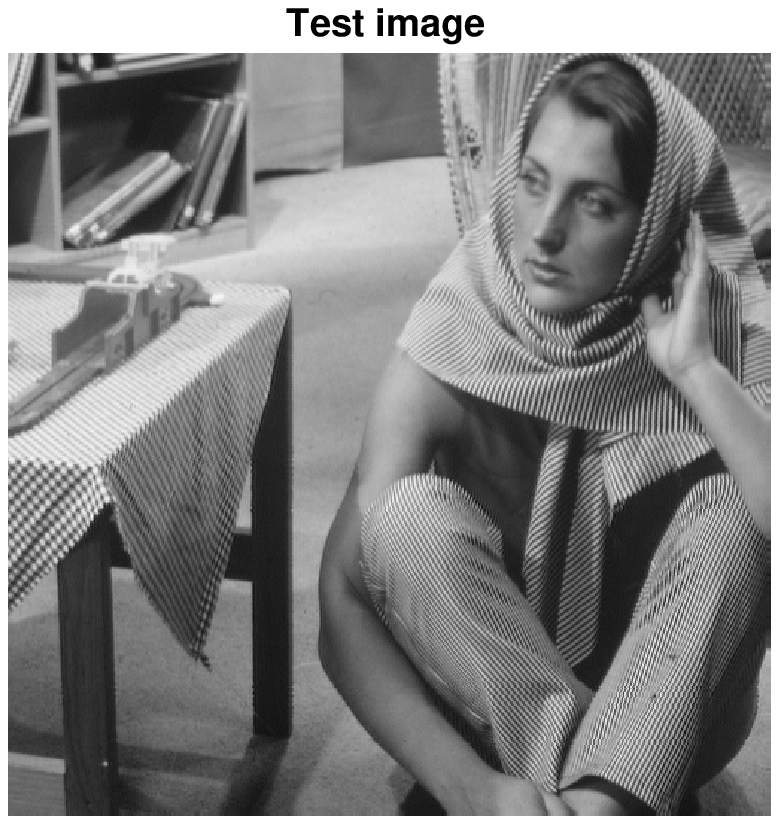}}
\subfloat[]{\label{fig:barbara_full2}\includegraphics[width = 0.3\textwidth]{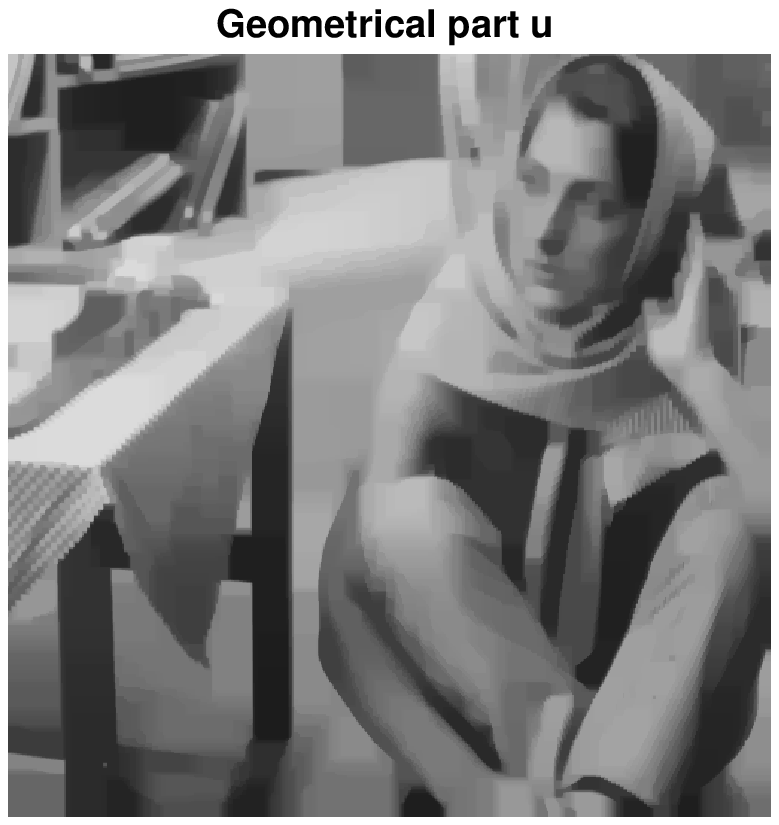}}
\subfloat[]{\label{fig:barbara_full3}\includegraphics[width = 0.3\textwidth]{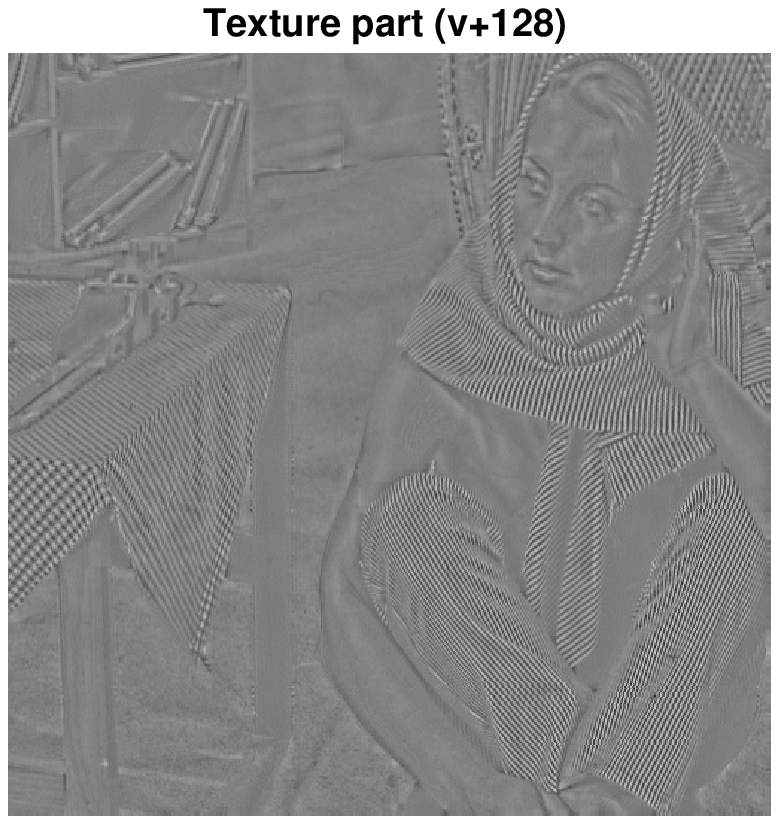}}
\caption{The $A^2BC$ model is applied to the test image ``Barbara". The original image is shown in (a). The corresponding minimizers $u$ and $v$ are shown in (b) and (c), respectively.}
\label{fig:barbara_full}
\end{figure}

In \cref{fig:strip,fig:strip2,fig:noisy_strip,fig:barbara_full,fig:barbara,fig:barbara2}, the minimizers $u,v$ and the minimal values $S$ for the corresponding input images are shown. 
To compute the minimizers, we apply a splitting algorithm to convert the optimization problem (\ref{eqt:AABFC}) to two subproblems involving computing the proximal point of $\lambda J$ and computing the projection to a $\mu-$ball of Meyer's norm. The second subproblem is the dual problem to the first one. As a result, for both subproblems, we can apply the algorithm in \cite{chambolle.09.ijcv,darbon2006image,hochbaum.01.jacm} to obtain the exact minimizers.

\begin{figure}[htbp!]
\centering
\subfloat[]{\label{fig:barbaraa}\includegraphics[width = 0.3\textwidth]{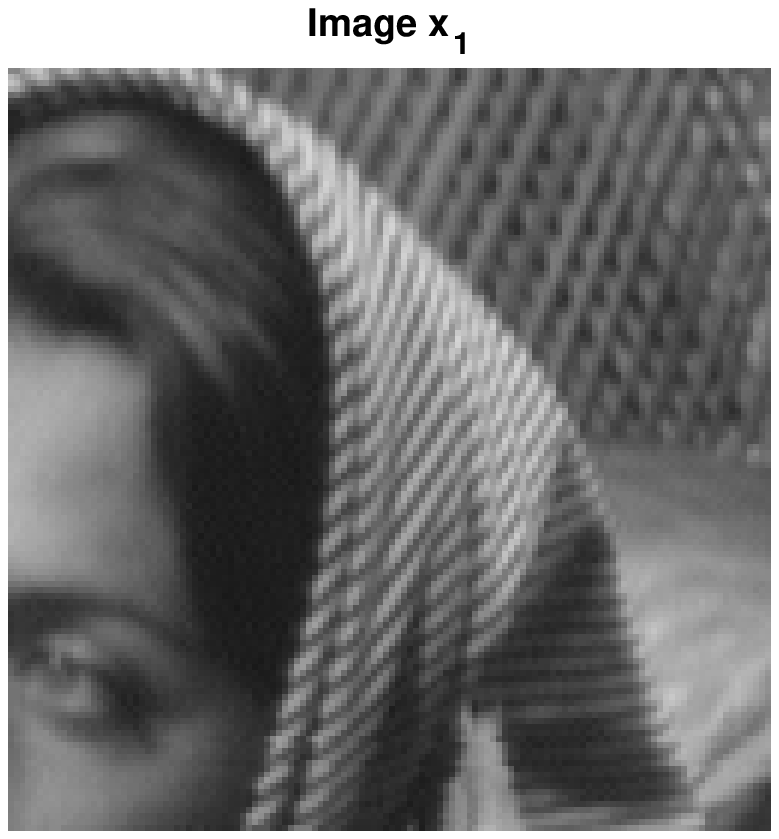}}
\subfloat[]{\label{fig:barbarab}\includegraphics[width = 0.3\textwidth]{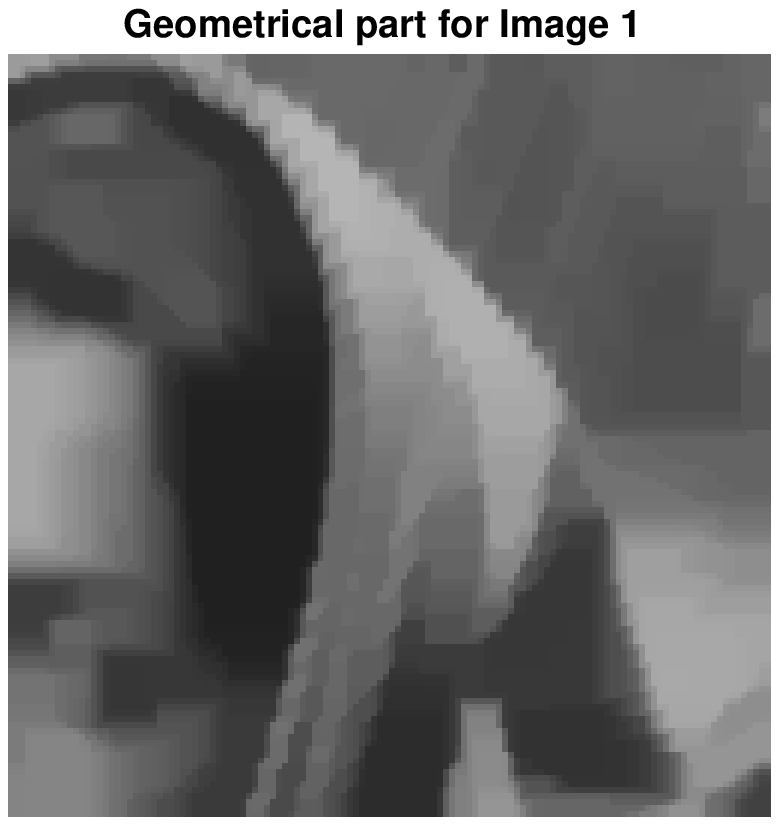}}
\subfloat[]{\label{fig:barbarac}\includegraphics[width = 0.3\textwidth]{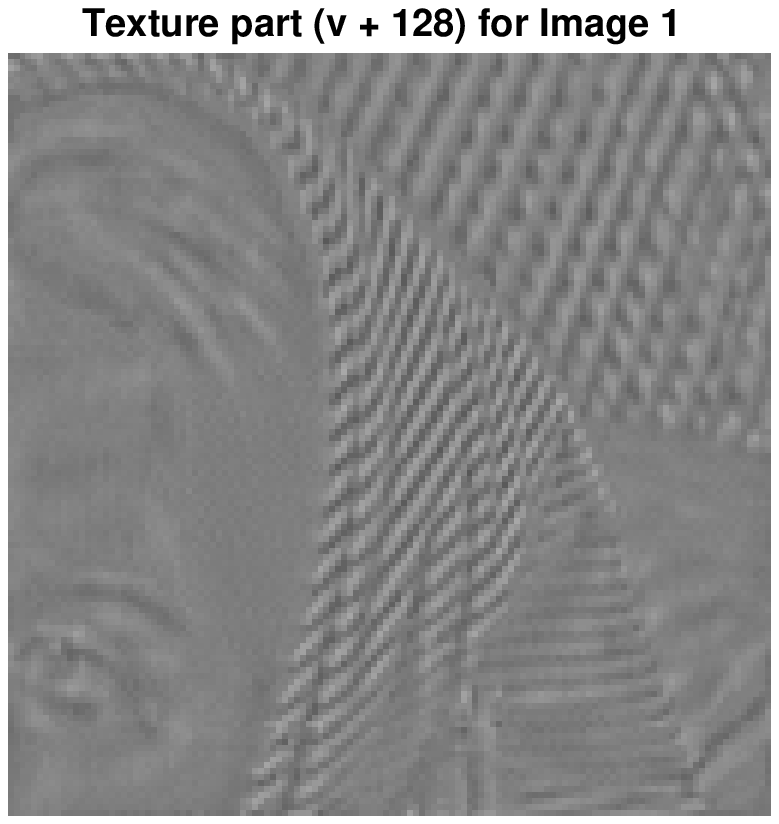}}\\
\subfloat[]{\label{fig:barbarad}\includegraphics[width = 0.3\textwidth]{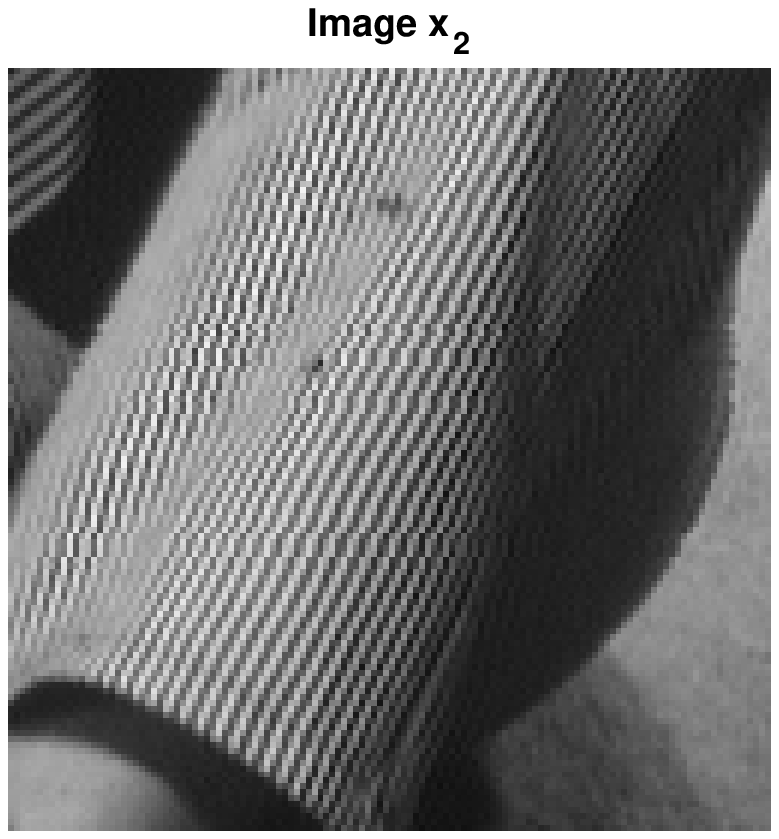}}
\subfloat[]{\label{fig:barbarae}\includegraphics[width = 0.3\textwidth]{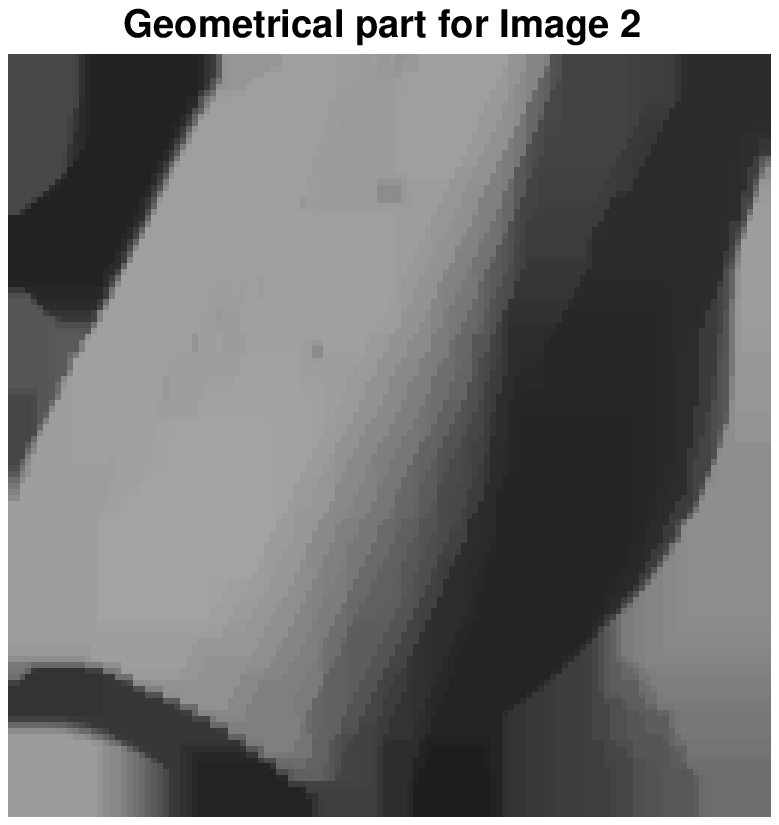}}
\subfloat[]{\label{fig:barbaraf}\includegraphics[width = 0.3\textwidth]{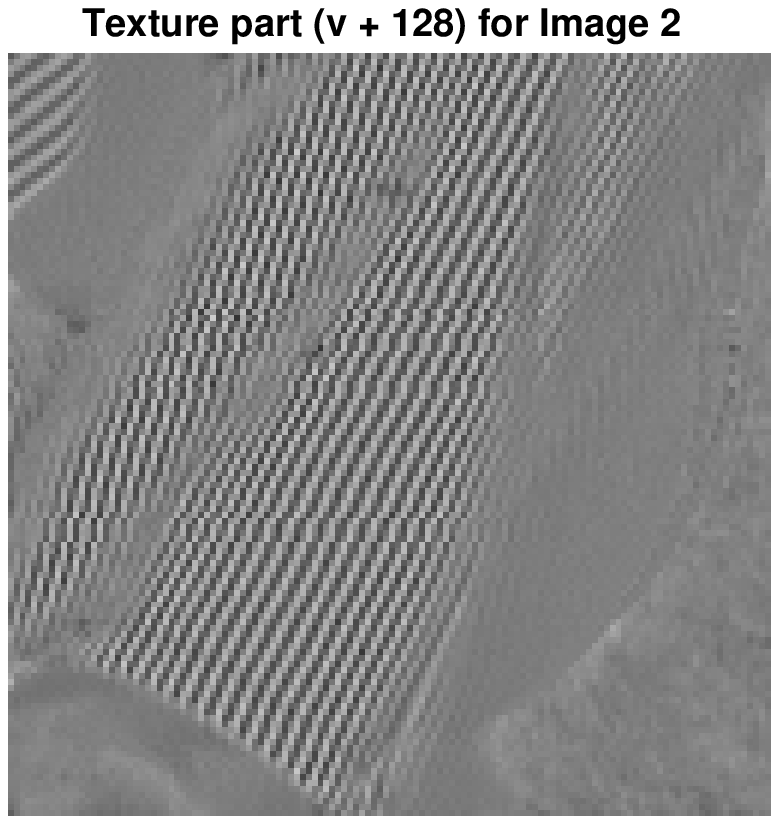}}\\
\subfloat[]{\label{fig:barbarag}\includegraphics[width = 0.3\textwidth]{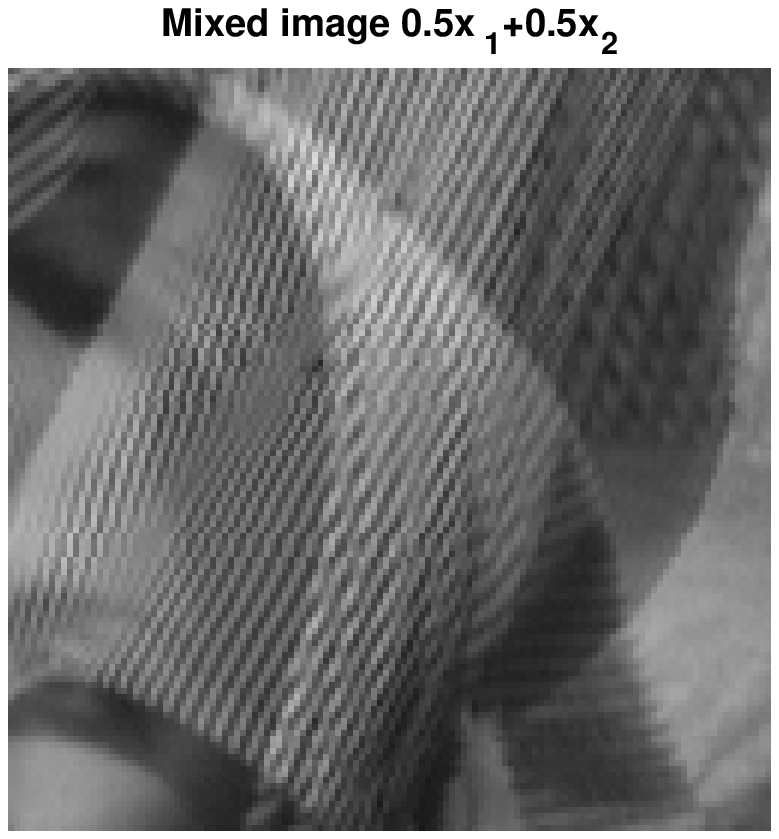}}
\subfloat[]{\label{fig:barbarah}\includegraphics[width = 0.3\textwidth]{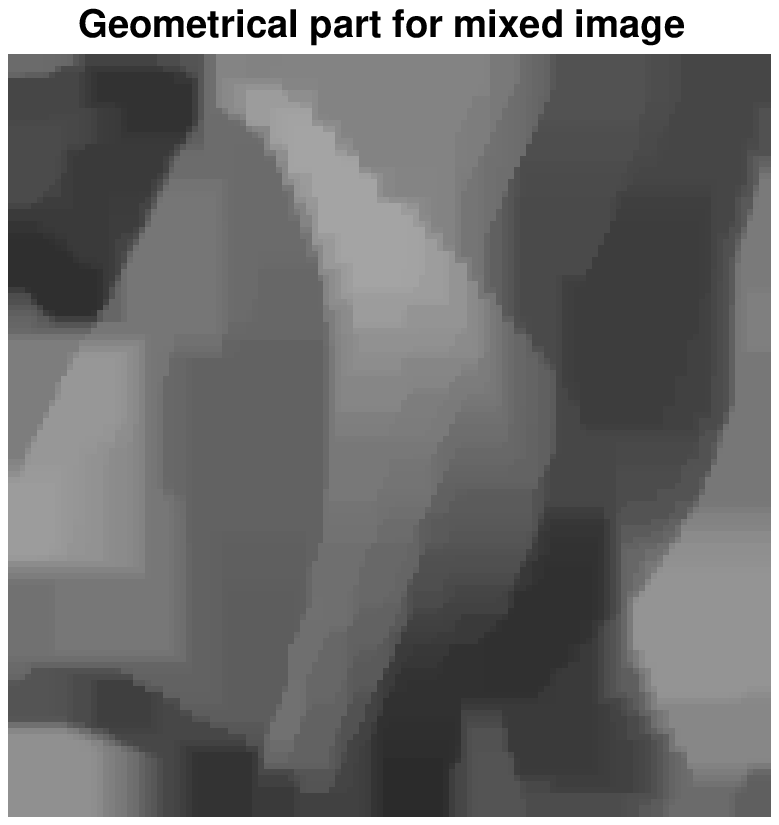}}
\subfloat[]{\label{fig:barbarai}\includegraphics[width = 0.3\textwidth]{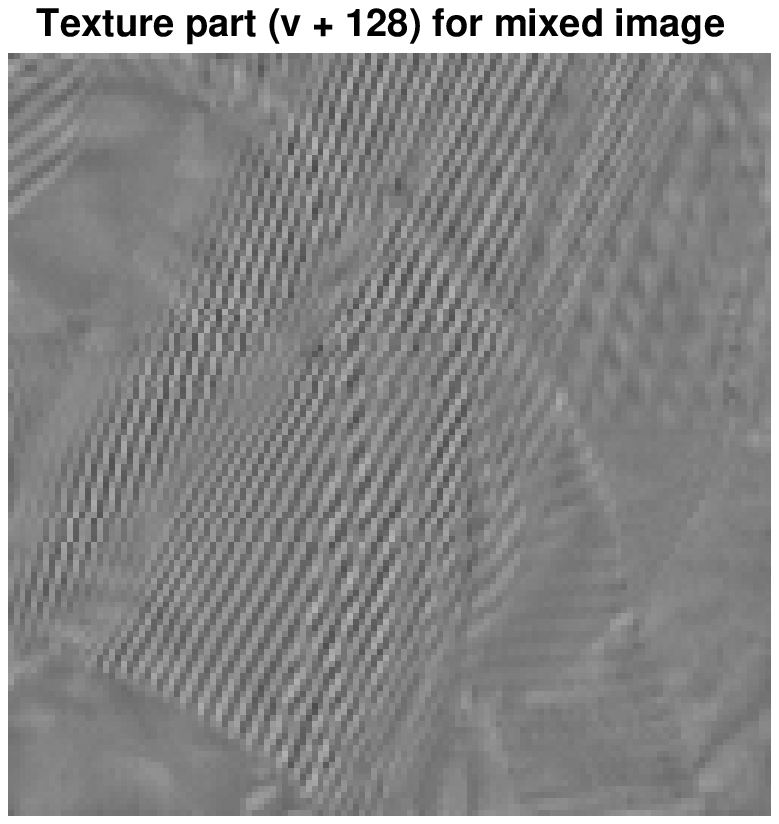}}
\caption{The $A^2BC$ model is applied to two parts of the image ``Barbara". The original image $x_1$, $x_2$ and corresponding minimizers $u,v$ are shown in (a)-(f). The convex combination $0.5 x_1 + 0.5 x_2$ and its minimizers are shown in (g)-(i).}
\label{fig:barbara}
\end{figure}

In the first example, the test image $x_1$ is shown in \cref{fig:stripa}. We consider the following parameters $\mu_1 = 1, \lambda_1 = 0.01$. The corresponding minimizers $u$ and $v$ are shown in \cref{fig:stripb,fig:stripc}.
When $x=x_1$, $\lambda = \lambda_1$ are fixed, the minimal values $S(x_1, \mu, \lambda_1)$ can be regarded as a function of $\mu$, whose graph is plotted in \cref{fig:striph}. Similarly, the graph of $S(x_1, \mu_1, \lambda)$ is plotted in \cref{fig:stripi}. To illustrate the variation of $S$ with respect to $x$, we choose another image $x_2$ with corresponding suitable parameters $\mu_2$, $\lambda_2$, and plot the function values $f: \alpha \mapsto S(\alpha x_1 + (1-\alpha) x_2, \alpha \mu_1 + (1-\alpha) \mu_2, \alpha \lambda_1 + (1-\alpha) \lambda_2)$ with $\alpha\in [0,1]$. In this example, $x_2$ is chosen to be a rotation of $x_1$, and the parameters remain the same: $\mu_2=\mu_1$, $\lambda_2=\lambda_1$.  The graph of $f$  is plotted in \cref{fig:stripg}. We also show an example of the mixed image $x = \alpha x_1 + (1-\alpha) x_2$ for $\alpha = 0.3$ and the corresponding minimizers $u,v$ in \cref{fig:stripd,fig:stripe,fig:stripf}. In addition, the $A^2BC$ model (with parameters $\mu = 0.06, \lambda = 0.01$) is applied to a noisy image shown in \cref{fig:noisy1}, whose minimizers are shown in \cref{fig:noisy2,fig:noisy3}.

\begin{figure}[tbp!]
    \centering
    \subfloat[]{\label{fig:barbaraj}\includegraphics[width = 0.5\textwidth]{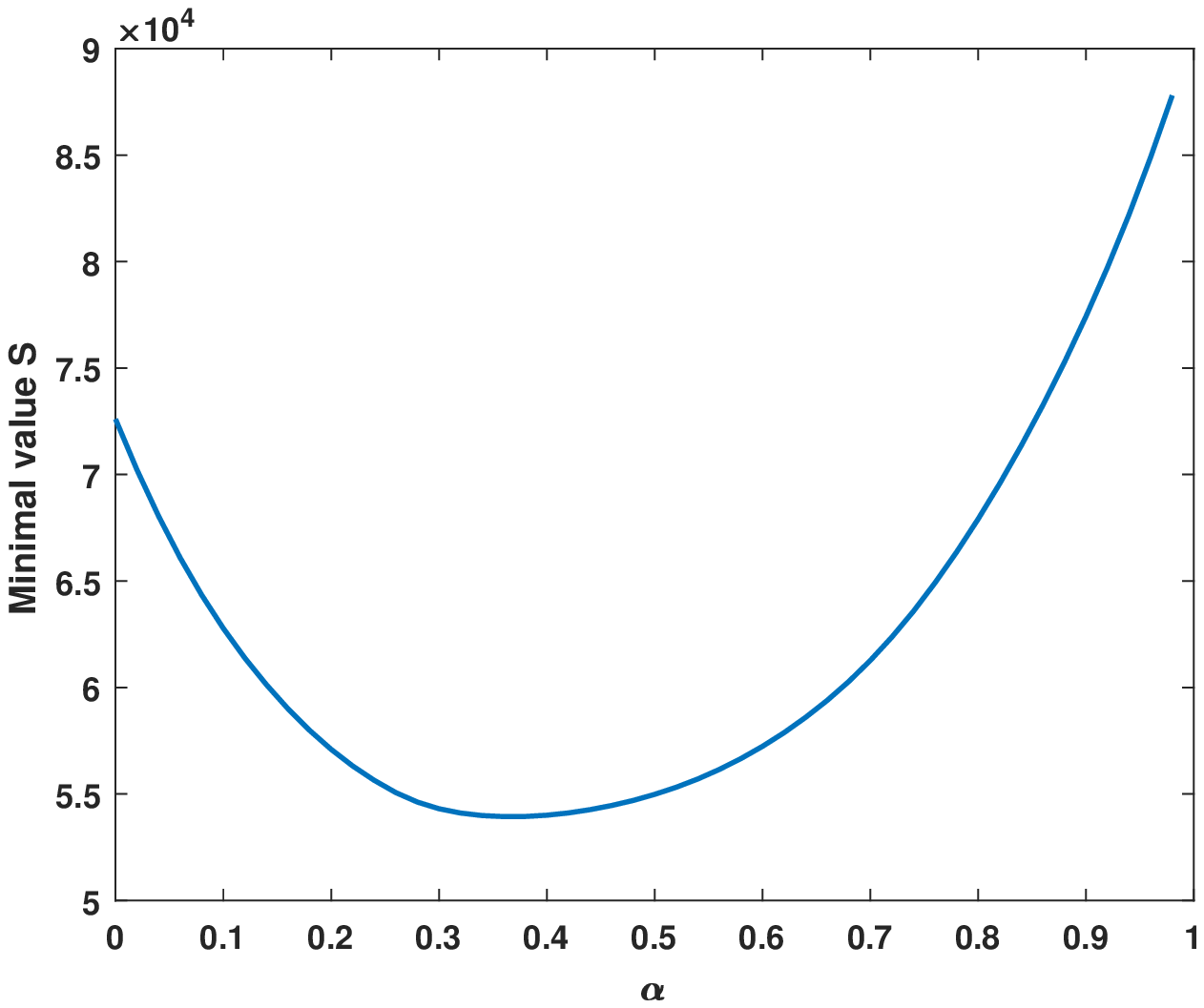}}
\subfloat[]{\label{fig:barbarak}\includegraphics[width = 0.5\textwidth]{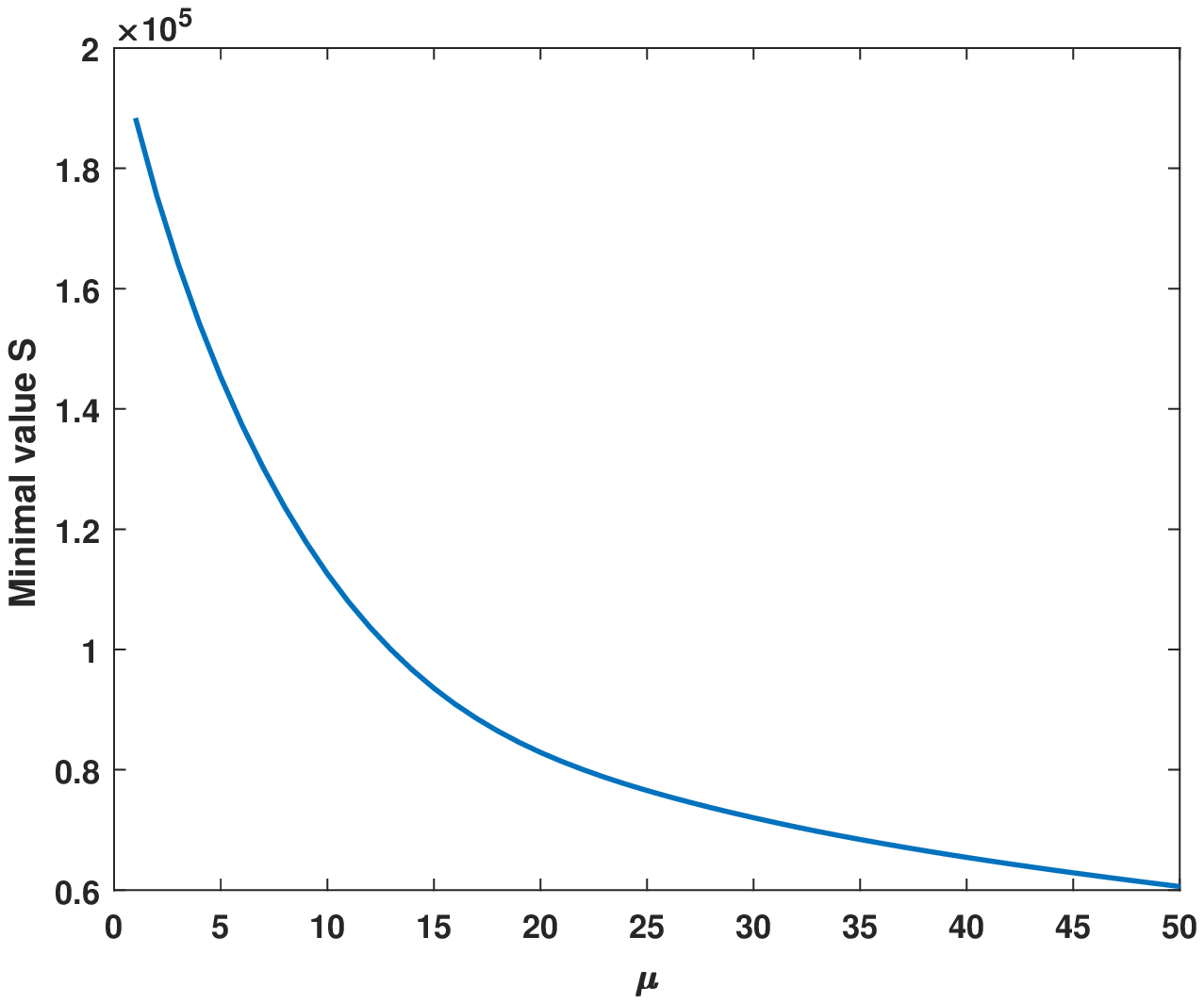}}\\
\subfloat[]{\label{fig:barbaral}\includegraphics[width = 0.5\textwidth]{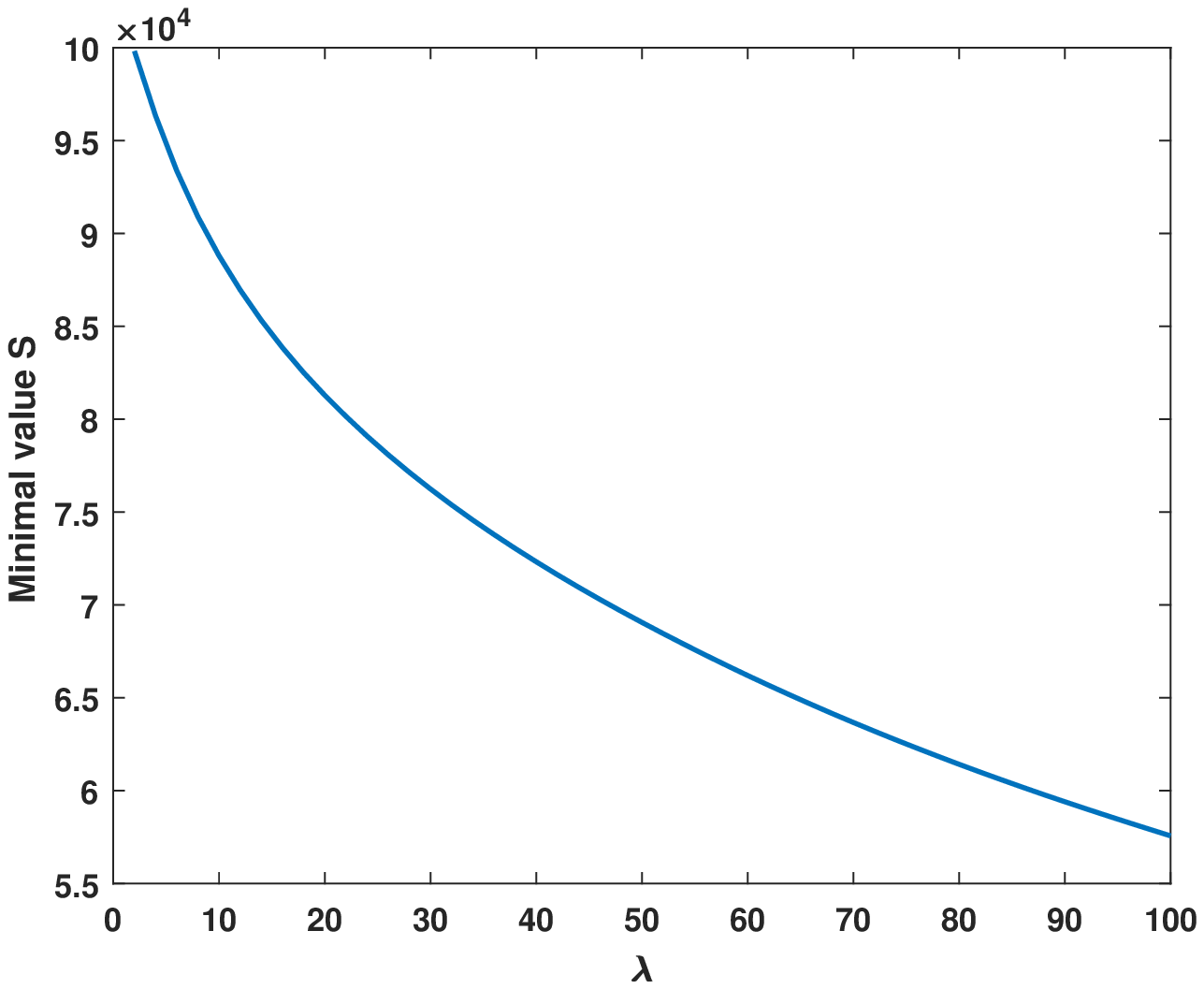}}
    \caption{The graphs of the minimal values $S$ with respect to the variables $\alpha$, $\mu$ and $\lambda$ in the second example are plotted in (a)-(c), respectively.}
    \label{fig:barbara2}
\end{figure}

The test image ``Barbara" is used in the second example. The original image and the corresponding minimizers $u,v$ in the $A^2BC$ model with parameters $\mu = 30, \lambda =8$ are shown in \cref{fig:barbara_full}. To demonstrate the variations of the minimal values, we choose two parts $x_1,x_2$ of the image, shown in \cref{fig:barbaraa,fig:barbarad}, and repeat the experiment in the first example. Setting $\mu_1 = 16$, $\mu_2 = 24$, $\lambda_1 = 8$, and $\lambda_2 = 12$, the corresponding minimizers $u, v$ are shown in \cref{fig:barbarab,fig:barbarac,fig:barbarae,fig:barbaraf}. The mixed image ($\alpha = 0.5$) and minimizers are shown in \cref{fig:barbarag,fig:barbarah,fig:barbarai}, and the dependence of $S$ on $x,\mu, \lambda$ is shown in \cref{fig:barbaraj,fig:barbarak,fig:barbaral}.

It can be seen from \cref{fig:strip2,fig:barbara2} that $S$ is a convex function with respect to the input image $x$ and the parameters. \updatenew{This can be proved with a similar argument as in the proof of \cref{prop:convex}.} In this paper, more properties about $S$ and the minimizers $u, v$ are revealed. 

\bigbreak

\update{\bf{Our contribution.}} The contribution of this paper is the theoretical results connecting the multi-time HJ equation and some optimization models such as decomposition models in imaging sciences. 
\update{There are three parts in this paper. In the first part, we consider the decomposition models and the corresponding dual problems, and investigate the properties of their optimizers and optimal values.}
To be specific, for some optimization problems, the minimal value coincides with the solution $S(x,t_1,\cdots,t_N)$ to a corresponding multi-time HJ equation. \update{This relationship in the case of single-time HJ equations has been studied in \cite{Darbon1}. 
We generalize the representation formula for the minimizer $u_j$ and the variational analysis results of $S$ and $\nabla_x S$ in \cite{Darbon1} to the case of multi-time HJ equations. Moreover, we present a new variational analysis of the scaled minimizer $\frac{u_j}{t_j}$. In the variational analysis, we consider a sequence $\{(x_k, t_{1,k},\cdots, t_{N,k})\}_k$, whose elements are perturbed variables near the point $(x,0,\cdots,0)$ and the perturbation becomes smaller when $k$ is larger.
We show that the limits of the corresponding spatial gradients $\nabla_x S$ and the scaled minimizers $\frac{u_j}{t_j}$ solve two optimization problems which are dual to each other.
In the second part, we prove the uniqueness of the convex solution to the multi-time HJ equation under some specific assumptions. In the field of PDEs, the uniqueness of the viscosity solution has been widely studied, for which we refer the readers to \cite{crandall1992user} and the references listed there. Here, our contribution is to provide a new perspective from convex analysis and use the duality technique to prove the uniqueness of the convex solution.}
At last, we propose a regularization method for the decomposition problems which may have non-unique minimizers or non-differentiable minimal values. 
\update{The regularization method is used to select a unique minimizer $u_{\lambda,\mu}$ and a unique gradient $p_{\lambda,\mu}$ of the minimal function where $\lambda$ and $\mu$ are some positive parameters. In fact, the gradient $p_{\lambda,\mu}$ coincides with the maximizer in the corresponding dual problem. This regularization method can be regarded as a generalization of the Moreau-Yosida approximation, which is introduced, for example, in \cite{Aubin:1984:DIS:577434, brezis1973operateurs}. Instead of only considering the primal problem as in the Moreau-Yosida approximation, our contribution here is to consider both the primal problem and the dual problem at the same time. Then, we apply the variational analysis result in the first part to prove the convergence of $u_{\lambda,\mu}$ and $p_{\lambda,\mu}$. We show that they converge to the $l^2$-projection of zero onto the corresponding sets of the original problems, when the regularization parameters $\lambda$ and $\mu$ approach zero in a comparable rate.}

\update{\bf{Organization of the paper. }}The paper is organized as follows. \Cref{sec:bkgd} gives a brief review of the convex optimization theorems which are used in the later proofs. The main results are stated in \cref{sec:prop,sec:uniq,sec:degenerate}.
In \cref{sec:prop}, the connection between some decomposition models and the multi-time HJ equation is shown.
\Cref{thm:u_formula} provides the representation formula for the minimizers $u_j$ of some decomposition models. Also, we investigate the variational behaviors of the minimal value $S$, the momentum $\nabla_x S$ and the velocities $\frac{u_j}{t_j}$ in \cref{thm:conv_grad}. \Cref{sec:uniq} is devoted to the proof of the uniqueness of the convex solution to the multi-time HJ equation. 
In \cref{sec:degenerate}, we present a regularization method for the degenerate cases which do not satisfy the assumptions in \cref{sec:prop}. The method is demonstrated using a specific example but the analysis can be easily applied to other models.
Finally, some conclusions are drawn in \cref{sec:conclusion}.

\section{Mathematical Background} \label{sec:bkgd}

\newcolumntype{s}{>{\hsize=.6\hsize}X}
\begin{table}[tbp]
\centering
\update{
 \caption{Notations used in this paper. Here, we use $C$ to denote a set, $f$ to denote a function and $x,d$ to denote vectors in $\Rn$.}
 \label{tab:notations}
\begin{tabularx}{\textwidth}{ l|s|X }
\hline
 Notation & Meaning & Definition\\
 \hline
 $\dom f$ & domain of $f$ & $\{x\in \Rn:\ f(x)\in\R\}$
 \\
 $\ri C$ & relative interior of $C$ & the interior of $C$ with respect to the minimal hyperplane containing $C$ in $\Rn$
 \\
 $N_C(x)$ & normal cone of $C$ at $x$ & $\{q\in\Rn:\ \langle q, y-x\rangle \leq 0 \text{ for any } y\in C\}$
 \\
 $C_\infty(x)$ & asymptotic cone of $C$ & $\{d\in \Rn:\ x+td\in C \text{ for all } t>0\}$
 \\
 $\epi f$ & epigraph of $f$ & $\{(x,t)\in\Rn\times\R:\ x\in \dom f,\ t\geq f(x)\}$
 \\
 $\gmRn$ & a useful and standard class of convex functions & the set containing all proper, convex, l.s.c. functions from $\Rn$ to $\R\cup\{+\infty\}$
 \\
 $f'(x,d)$ & directional derivative of $f$ at $x$ along the direction $d$ & $\lim_{h \to 0^+} \frac{1}{h}(f(x+hd) - f(x))$
 \\
 $\partial f(x)$ & subdifferential of $f$ at $x$ & $\{p\in\Rn:\ f(y)\geq f(x) + \langle p, y-x\rangle \ \forall y\in\Rn\}$
 \\
 $I_C$ & the indicator function of $C$ & If $x\in C$, then define $I_C(x):=0$. Otherwise, define $I_C(x):= +\infty$.
 \\
 $f^*$ & Legendre transform of $f$ & $f^*(p) := \sup_{x\in \Rn} \langle p, x\rangle - f(x)$
 \\
 $f\conv g$ & inf-convolution of $f$ and $g$ & $(f\conv g)(x) := \inf_{u\in \Rn} f(u) + g(x-u)$
 \\
\hline
 \end{tabularx}
 }
\end{table}

In this section, several basic definitions and theorems in convex analysis are reviewed. All the results and notations can be found in \cite{convex_analy_book1,convex_analy_book2}. We also refer the readers to \cite{Bauschke:2011:CAM:2028633,Borwein,RockafellarConvexAnalysis}.

First, a set $C$ in $\Rn$ is convex if $\alpha x+ (1-\alpha)y\in C$ whenever $x,y\in C$ and $\alpha\in [0,1]$. The relative interior of $C$, denoted as $\ri C$, is the interior of $C$ with respect to the minimal hyperplane containing $C$ in $\Rn$. For any convex set $C$, the normal cone of $C$ at $x\in C$, denoted by $N_C(x)$, can be characterized by
\begin{equation}\label{eqt:def_normal_cone}
    q\in N_C(x) \text{ if and only if } \langle q, y-x\rangle \leq 0 \text{ for any } y\in C.
\end{equation}
\update{Here, we use the angle bracket $\langle \cdot, \cdot \rangle$ to denote the inner product operator in any Euclidean space $\R^n$.}
For any closed convex set $C$ and any point $x\in C$, one can define the asymptotic cone of $C$, denoted as $C_\infty(x)$, by
\begin{equation}\label{eqt:defasym}
    C_\infty(x) = \{d\in \Rn:\ x+td\in C \text{ for all } t>0\}.
\end{equation}
In fact, the asymptotic cone is independent of $x$, as stated in the following result.
\begin{prop}\cite[Prop.III.2.2.1]{convex_analy_book1}
\label{prop:asym}
Let $C$ be a closed convex set and $x,y\in C$. Then $C_\infty(x) = C_\infty(y)$. In other words, for any $d\in C_\infty(x)$, $y+td\in C$ for any $t>0$.
\end{prop}

A function $f: \Rn \to \R\cup\{+\infty\}$ is said to be convex if for any $\alpha \in (0,1)$ and any $x,y\in \Rn$,
\begin{equation*}
f(\alpha x + (1-\alpha)y)\leq \alpha f(x) + (1-\alpha) f(y).
\end{equation*}
The function $f$ is called proper if it is not identically equal to $+\infty$.
The domain of $f$, denoted by $\dom f$, is defined to be the set where $f$ does not take the value $+\infty$. The epigraph of $f$, denoted as $\epi f$, is defined by:
\begin{equation*}
\epi f := \{(x,t):\ x\in \dom f,\ t\geq f(x)\}.
\end{equation*}

Then, $f$ is convex (proper, or lower semi-continuous, respectively) if and only if $\epi f$ is convex (non-empty, or closed, respectively). 
We denote $\gmRn$ to be the set of proper, convex and lower semi-continuous (l.s.c) functions from $\Rn$ to $\R\cup\{+\infty\}$. In this section, we only consider the functions in $\gmRn$. These functions have good continuity properties, which are stated below.
\begin{prop} \cite[Lem.IV.3.1.1 and Chap.I.3.1 - 3.2]{convex_analy_book1}
\label{prop1}
Let $f\in \gmRn$. If $x\in \ri \dom f$, then $f$ is continuous at $x$ in $\dom f$. If $x\in \dom f \setminus \ri \dom f$, then for any $y\in \ri \dom f$,
\begin{equation*}
f(x) = \lim_{t\to 0^+} f(x + t(y-x)).
\end{equation*}
\end{prop}

For any $f\in \gmRn$ and $x\in \dom f$, the directional derivative at $x$ along any direction $d$, denoted as $f'(x,d)$, is well-defined in $\R\cup \{\pm\infty\}$. When $f$ is differentiable at $x$, $f'(x,\cdot) = \langle \nabla f(x), \cdot\rangle$ is a linear function. In general, when $f$ is not differentiable, $f'(x,\cdot)$ is only sublinear, in which case we can consider the linear functions dominated by it. Each normal vector of such linear functions gives a subgradient of $f$ at $x$, whose formal definition is given below. Also, the rigorous statement about the relation we described above between the directional derivatives and subgradients is given in \cref{prop5}. 

A vector $p$ is called a subgradient of $f$ at $x$ if it satisfies
\begin{equation*}
f(y)\geq f(x) + \langle p, y-x\rangle, \text{ for any } y\in\Rn.
\end{equation*}
The collection of all such subgradients is called the subdifferential of $f$ at $x$, denoted as $\partial f(x)$. It is easy to check that $0\in \partial f(x)$ if and only if $x$ is a minimizer of $f$. As a result, one can check whether $x$ is a minimizer by computing the subdifferential.

As is well known, the subdifferential operator is a (maximal) monotone operator. To be specific, 
\begin{equation}\label{eqt:subdiff_monotone}
\langle p-q, x-y\rangle \geq 0 \text{ for any }p\in \partial f(x)\text{ and }q\in \partial f(y).
\end{equation}
Moreover, in most cases, the subdifferential operator commutes with summation.
\begin{prop} \cite[Cor.XI.3.1.2]{convex_analy_book2}
\label{prop7}
Let $f,g\in \gmRn$. Assume $\ri \dom f \cap \ri \dom g \neq \emptyset$. Then $\partial (f+g)(x) = \partial f(x) + \partial g(x)$ for any $x\in \dom f \cap \dom g$.
\end{prop}

Here, we give one simple example. For any convex set $C$, the indicator function $I_C$ is defined by
\begin{equation*}
    I_C(x) := \begin{cases}
    0, &x\in C;\\
    +\infty, &x\not\in C.
    \end{cases}
\end{equation*}
In this paper, we also use the notation $I\{\cdot\}$ to denote the indicator function if the set $C$ is given in the form of some constraints. \update{By definition, the indicator function $I_C$ remains the same after multiplying by a positive constant, i.e. we have $\alpha I_C = I_C$ for any $\alpha >0$.} One can compute the subdifferential of the indicator function and obtain
\begin{equation} \label{eqt:subgrad_indicator}
    \partial I_C(x) = N_C(x).
\end{equation}

\bigbreak
Next, we introduce one important transform in convex analysis called Legendre transform. For any function $f\in \gmRn$, the Legendre transform of $f$, denoted as $f^*$, is defined by
\begin{equation}\label{eqt:Legendre}
f^*(p) := \sup_{x\in \Rn} \langle p, x\rangle - f(x).
\end{equation}
Legendre transform gives a duality relationship between $f$ and $f^*$. In other words, if $f\in \gmRn$, then $f^*\in \gmRn$ and $f^{**} = f$.
Similarly, along with this duality relationship, some properties are dual to others, as stated in the following proposition. (Here and after, a function $g$ is called 1-coercive if $\lim_{\|x\|\to +\infty} g(x)/\|x\| = +\infty$.)
\begin{prop} \cite[Chap.X.4.1]{convex_analy_book2}
\label{prop6}
Let $f\in \gmRn$. Then $f$ is finite-valued if and only if $f^*$ is 1-coercive. Also, $f$ is differentiable if and only if $f^*$ is strictly convex.
\end{prop}
In particular, the subgradients can be characterized by the maximizers in \cref{eqt:Legendre}.
\begin{prop} \cite[Cor.X.1.4.4]{convex_analy_book2}
\label{prop3}
Let $f\in \gmRn$ and $p,x \in \Rn$. Then $p\in \partial f(x)$ if and only if $x\in \partial f^*(p)$, if and only if $f(x) + f^*(p) = \langle p, x\rangle$.
\end{prop}

The concepts we introduced above, including directional derivatives, subgradients and Legendre transform, can be linked all together by the following proposition. 
\begin{prop} \cite[Example X.2.4.3]{convex_analy_book2}
\label{prop5}
Let $f\in \gmRn$ and $x\in \dom f$ such that $\partial f(x)$ is nonempty, then $(f'(x,\cdot))^* = I_{\partial f(x)}$. Moreover, if $x\in \ri \dom f$, then $f'(x,\cdot)\in \gmRn$, hence $f'(x,\cdot) = I_{\partial f(x)}^*$.
\end{prop}

Except from Legendre transform, there is another operator to construct convex functions called inf-convolution. Given two functions $f,g\in \gmRn$, assume there exists an affine function $l$ such that $f(x)\geq l(x)$ and $g(x)\geq l(x)$ for any $x\in\Rn$. Then, the inf-convolution between $f$ and $g$, denoted as $f\conv g$, is a convex function taking values in $\R\cup \{+\infty\}$. The definition of the inf-convolution $f\conv g$ is given by
\begin{equation} \label{eqt:infconv}
(f\conv g)(x) := \inf_{u\in \Rn} f(u) + g(x-u).
\end{equation}
In the following proposition, the relation between Legendre transform and inf-convolution is stated. Actually, the Hopf formula and Lax formula introduced in the next section are formulated using Legendre transform and inf-convolution operator, respectively.
As a result, these two operators play a significant role in our analysis in this paper.
\begin{prop} \cite[Thm.X.2.3.2 and Thm.XI.3.4.1]{convex_analy_book2}
\label{prop4}
Let $f,g \in \gmRn$. Assume the intersection of $\ri \dom f^*$ and $\ri \dom g^*$ is non-empty. Then $f\conv g\in \gmRn$ and $f\conv g = (f^*+g^*)^*$. Moreover, for any $x\in \dom f\conv g$, the optimization problem \cref{eqt:infconv} has at least one minimizer, and $\partial (f\conv g)(x) = \partial f(u) \cap \partial g(x-u)$ for any minimizer $u$.
\end{prop}

\section{Properties of the Solutions to the Multi-time Hamilton-Jacobi Equations} \label{sec:prop}
In this section, we provide a representation formula for the minimizers in the Lax formula and highlight the relation of the minimizers and the momentum in the multi-time HJ equation. 
Also, we  investigate the variational behaviors of both the solution to the multi-time HJ equation and the corresponding momentum when time variables approach zero. Moreover, we also present a new result stating the variational behaviors of the velocities, which has not been developed before, even for the single-time case. Similar to the duality relation of the Hopf and Lax formulas, the cluster points of the minimizers and momentum solve two optimization problems, which are also dual to each other. An illustration is given in \update{the upper part of} \cref{fig:illustrationprop}.

\begin{sidewaysfigure}
\centering
\includegraphics[width=1\textheight,height=0.8\textwidth,keepaspectratio]{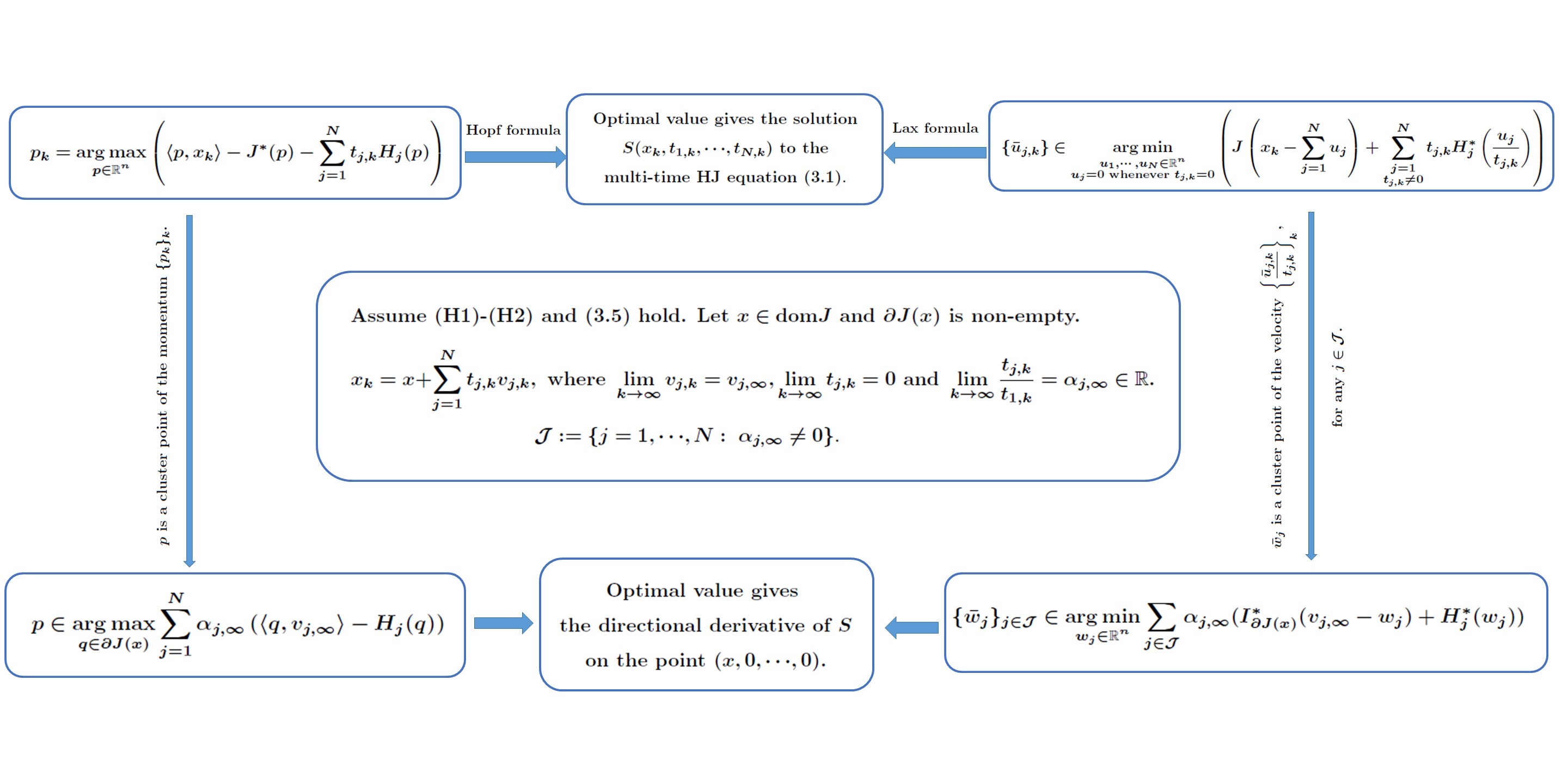}
\caption{This is an illustration for \cref{thm:conv_grad}. It shows the relation of four optimization problems and the multi-time HJ equation. Here, $\bar{u}_{j,k}$ denotes the minimizer $u_j(x_k, t_{1,k}. \cdots, t_{N,k}).$}
\label{fig:illustrationprop}
\end{sidewaysfigure}

We consider the solution $S(x,t_1,\cdots,t_N)$ to the following multi-time HJ equation
\begin{equation} \label{eqt:H-J}
\begin{cases}
\frac{\partial S}{\partial t_j} + H_j(\nabla_x S) = 0 \text{ for any }j\in\{1,\cdots, N\}, &x\in\Rn, t_1,\cdots,t_N>0; \\
S(x,0,\cdots,0) = J(x), &x\in \Rn.
\end{cases}
\end{equation}
Here, we only consider the multi-time HJ equations whose Hamiltonians only depend on the momentum $\nabla_x S$.
Several conditions are imposed on the Hamiltonians $\{H_j\}$ and the initial data $J$ in this section. 
To be specific, we assume
\setlist[enumerate]{leftmargin=.5in}
\setlist[itemize]{leftmargin=.5in}
\begin{itemize}
\item[(H1)] $H_j: \Rn\to \R$, is convex and 1-coercive for any $j=1,\cdots,N$. Moreover, at least one of them is strictly convex;
\item[(H2)] $J\in \gmRn$.
\end{itemize}
\setlist[enumerate]{leftmargin=.25in}
\setlist[itemize]{leftmargin=.25in}
From the assumption (H1), by \cref{prop6}, it is known that $H_j^*$ is also finite-valued, convex and 1-coercive for any $j=1,\cdots,N$. Moreover, at least one $H_j^*$ is differentiable.

\bigbreak
It is well known that in this case the unique classical solution is given by the Hopf formula \update{\cite{lions1986hopf, tho2005hopf} stated as follows}
\begin{equation} \label{eqt:hopf}
S_H(x,t_1,\cdots,t_N) := \left(J^* + \sum_{j=1}^N t_jH_j\right)^*(x) = \sup_{p\in\Rn} \left(\langle p,x\rangle - J^*(p) - \sum_{j=1}^N t_jH_j(p)\right),
\end{equation}
and the Lax formula \update{\cite{tho2005hopf} stated as follows}
\begin{equation}\label{eqt:lax}
\begin{split}
S_L(x,t_1,\cdots,t_N) :=& \left(J\conv (t_1H_1)^* \conv \cdots \conv (t_NH_N)^*\right)(x) \\
=& \inf_{\substack{u_1,\cdots,u_N\in\Rn\\u_j=0\text{ whenever } t_j=0}} \left( J\left(x-\sum_{j=1}^Nu_j\right) + \sum_{\substack{j=1\\t_j \neq 0}}^N t_jH_j^*\left(\frac{u_j}{t_j}\right)\right),
\end{split}
\end{equation}
for any $x\in\Rn$ and $t_1,\cdots, t_N\geq 0$. We extend $S_H$ and $S_L$ to the whole domain by simply setting the function values to $+\infty$ whenever the function value is not defined. \update{There are some physical interpretations of the HJ PDEs and the optimizers in the above two formulas. Given suitable Hamiltonians $\{H_j\}$ and a suitable initial condition $J$, the HJ PDE \cref{eqt:H-J} describes the movement of a particle. Roughly speaking, in a time interval with length $t_j$, a particle moves along the characteristic line of the $j-$th equation in the PDE system. The velocity in this time interval equals $\frac{u_j}{ t_j}$ where $(u_1,\cdots, u_N)$ denotes the minimizer in the Lax formula \cref{eqt:lax}. On the other hand, the maximizer in the Hopf formula \cref{eqt:hopf} gives the momentum of the particle, which coincides with the spatial gradient $\nabla_x S(x,t_1,\cdots, t_N)$. We refer the reader to \cite{Basdevant.07.book} for details about HJ PDEs and variational principles in physics.}

Under the assumptions (H1) and (H2), $S_H = S_L$, and the value is finite if there exists some $t_j>0$. In addition, the minimizers in the Lax formula \cref{eqt:lax} exist whenever the minimal value is finite. This result can be proved using \cref{prop4}. Also, by \cref{prop3}, it is not hard to check $S_H\in C^1(\Rn\times (0,+\infty)^N)$ and satisfies HJ equation \cref{eqt:H-J}. Moreover, the spatial gradient is the unique maximizer in the Hopf formula \cref{eqt:hopf}. To conclude, the Hopf and Lax formulas express the classical solution to the multi-time HJ equation as two optimization problems.
The Hopf formula provides a physical interpretation and has the momentum $\nabla_x S$ as the maximizer, while its dual problem in the Lax formula is in the same form as some decomposition models in imaging sciences.

The following proposition states that the solution is actually a convex function, hence the techniques in convex analysis can be applied to analyze the solution. The results hold even under weaker assumptions. Actually, a part of the proposition can be further generalized to the case when $J, H_j \in \gmRn$ and $\dom J^* \subseteq \dom H_j$ for any $j$.

\begin{prop}\label{prop:convex}
\update{Let $J, H_j\in \Gamma_0(\Rn)$ and $\dom H_j = \Rn$ for any $j$. Then, $S_H\in \Gamma_0(\mathbb{R}^{n+N})$, whose Legendre transform is given by
\begin{equation*}
S_H^*(p,E^-) = J^*(p) + \sum_{j=1}^N I\{E_j^- + H_j(p)\leq 0\},
\end{equation*}
for any $p\in\Rn$ and $E^- = (E_1^-,\cdots, E_N^-)\in\mathbb{R}^N$.
Here, $I\{\cdot\}$ denotes the indicator function.
Moreover, if the assumptions (H1)-(H2) are satisfied, then $S_H(x,t_1,\cdots, t_N)$ is finite for any $x\in\Rn$ and $t_1,\cdots,t_N\geq 0$ which are not all zero.}
\end{prop}

\begin{proof}
First, we prove that $S_H$ is the Legendre transform of $F$, where $F$ is defined by
\begin{equation*}
F(p,E^-) := J^*(p) + \sum_{j=1}^N I\{E_j^-+H_j(p)\leq 0\},
\end{equation*}
for any $p\in\Rn$ and any $E^-=(E_1^-,\cdots, E_N^-)\in \RN$. It is easy to check $F\in \Gamma_0(\R^{n+N})$.

By definition, for any $x\in\Rn$ and $t=(t_1,\cdots, t_N)\in\RN$,
\begin{equation} \label{eqt:formula_conjugate}
\begin{split}
F^*(x, t) = \sup_{p\in\Rn, E^-\in\RN} \left( \langle x,p\rangle + \sum_{j=1}^N t_jE_j^- - J^*(p) - \sum_{j=1}^N I\{E_j^-+H_j(p)\leq 0\}\right).
\end{split}
\end{equation}

First, we consider the case when there exists $k$ such that $t_k<0$. Take $p\in \dom J^*$. For any $j\neq k$, take $E_j^- = -H_j(p)$, which is a finite value. From the above equation, 
\begin{equation*}
F^*(x,t) \geq  \langle x, p \rangle + \sum_{j\neq k} t_jE_j^- - J^*(p) + \limsup_{\substack{E_k^-\leq -H_k(p)\\E_k^- \to -\infty}} t_kE_k^-  = +\infty.
\end{equation*}
Hence $F^*(x,t) = +\infty = S_H(x,t)$ if $t_k<0$ for some $k$.

Then, consider the case when $t_1,\cdots, t_N\geq 0$. Let $x\in\Rn$, from \cref{eqt:formula_conjugate}, we obtain
\begin{equation}
\begin{split}
F^*(x, t) &= \sup_{\substack{p\in\Rn\\ E_j^-\leq -H_j(p)\  \forall j}} \left( \langle x,p\rangle + \sum_{j:\ t_j>0} t_jE_j^- - J^*(p)\right)\\
&= \sup_{p\in\Rn}\left[ \sup_{E_j^-\leq -H_j(p)\  \forall j} \left( \langle x,p\rangle + \sum_{j:\ t_j>0} t_jE_j^- - J^*(p)\right)\right]\\
&= \sup_{p\in\Rn} \left(\langle x,p\rangle - \sum_{j:\ t_j>0} t_jH_j(p) - J^*(p)\right) = S_H(x,t).
\end{split}
\end{equation}

Therefore, $S_H = F^*$, which implies $S_H$ is a convex lower semi-continuous function and $F = S_H^*$. \update{Moreover, if there exists some $k$ such that $t_k>0$ and $t_j\geq 0$ for any $j\neq k$, then, by assumption (H1), we deduce that $J^* +\sum_jt_jH_j$ is 1-coercive, which, by \cref{prop6}, implies its Legendre transform $S_H(\cdot,t_1,\cdots, t_N)$ (with respect to x) is finite-valued.}
\end{proof}

By investigating $S_H$ on the boundary of the domain, the solution to a lower time dimensional equation is embedded in the solution to the higher time dimensional equation, in the sense that the restriction of $S_H$ on the subspace $\{(x,t_1,\cdots,t_N):\ t_j=0 \ \forall j\in J\}$ for any index set $J\subset \{1,\cdots,N\}$ is the solution to the corresponding lower time dimensional HJ equation with Hamiltonians $\{H_j\}_{j\not\in J}$. 

\bigbreak
The following proposition states a representation formula for the minimizers in the Lax formula. \update{In the decomposition model \cref{eqt:lax}, a given image $x$ is decomposed into different components including $u_1,\cdots, u_N$ and the residual $x-\sum_{j=1}^N u_j$. However, sometimes the primal minimization problem is difficult to solve, then the following proposition can be applied to compute $(u_1,\cdots, u_N)$ using the momentum $\nabla_x S_L(x,t_1,\cdots, t_N)$. In fact, the momentum is the maximizer of the dual problem in the Hopf formula \cref{eqt:hopf}. In other words, the following proposition gives the relation of the optimizers in the primal decomposition problem and the dual problem.}

\begin{prop} \label{thm:u_formula}
Suppose the assumptions (H1)-(H2) hold. Let $x\in\Rn, t_1,\cdots, t_N\geq 0$ and assume the time variables $\{t_j\}$ are not all zero. 
Denote $(u_1,\cdots, u_N)$ to be any minimizer of the minimization problem in \cref{eqt:lax} with parameters $x$ and $t_1,\cdots, t_N$. Here, each $u_j$ can be regarded as a function of $(x,t_1,\cdots, t_N)$. Then, for any $j$,
\begin{equation}
u_j(x,t_1,\cdots,t_N) \in t_j\partial H_j\left(\nabla_x S_L(x,t_1,\cdots,t_N)\right).
\end{equation}

Specifically, if a stronger assumption is imposed, say, all the Hamiltonians are differentiable, then the minimizer $(u_1,\cdots, u_N)$ is unique and satisfies
\begin{equation*}
 u_j(x,t_1,\cdots,t_N) = t_j\nabla H_j\left(\nabla_x S_L(x,t_1,\cdots,t_N)\right) \text{ for any }j.
\end{equation*}
\end{prop}

\begin{proof}
Since $\dom H_j = \Rn$ for each $j$, by \cref{prop4} and induction, the minimizers $u_j$ exist if $S_L(x,t_1,\cdots, t_N)<+\infty$, and 
\begin{equation} \label{eqt:lax_diff}
\begin{split}
\partial_x S_L(x, t_1,\cdots, t_N) 
&= \partial J\left(x - \sum_{j=1}^N u_j\right) \bigcap \left(\bigcap\limits^N_{j=1} \partial \left(t_jH_j^*\left(\frac{\cdot}{t_j}\right)\right)(u_j)\right)\\
&= \partial J\left(x - \sum_{j=1}^N u_j\right) \bigcap \left(\bigcap\limits^N_{j=1} \partial H_j^*\left(\frac{u_j}{t_j}\right)\right).
\end{split}
\end{equation}
From the assumption (H1), there exists some $j$ such that $H_j^*$ is differentiable, hence the intersection above contains at most one element. On the other hand, $\partial_x S_L$ is non-empty in the interior of the domain of $S_L(\cdot, t_1,\cdots, t_N)$, which is the whole space $\Rn$ because $S_L=S_H$ is finite-valued when the time variables are not all zero. 
Therefore, the above intersection contains exactly one element. In other words, $S_L$ is differentiable with respect to $x$ for any $t_1, \cdots, t_N\geq 0$ which are not all zero and $x\in \Rn$. Moreover, by \cref{eqt:lax_diff}, $\nabla_x S_L\in \partial H_j^*(u_j/t_j)$, which implies $u_j/t_j \in \partial H_j(\nabla_x S_L(x, t_1, \cdots, t_N))$ for any $j$.      
\end{proof}                                                                                                                      

\bigbreak

In the remaining part of this section, we investigate the multi-time HJ equation \cref{eqt:H-J} and the minimization problem \cref{eqt:lax} in a variational point of view. \update{To be specific, let $v_{j,k}\in \Rn$ and $t_{j,k}>0$ for any $j\in \{1,\cdots, N\}$ and $k\in \N$
such that they satisfy
$\lim_{k\to+\infty} t_{j,k} = 0$ and $\lim_{k\to+\infty}v_{j,k} = \vlimj$ for any $j$.
Let $x\in\Rn$ and $x_k = x + \sum_{j=1}^N t_{j,k}v_{j,k}$ for any $k$.}
We are interested in the convergence behavior of the momentum $\nabla_x S_H$ and the minimizers $u_j$ evaluated at $(x_k, t_{1,k},\cdots, t_{N,k})$. We will demonstrate one application in \cref{sec:degenerate}. 

Among all the sequences $\{t_{j,k}\}_k,j=1,\cdots,N$, by taking subsequences, we can assume there is a sequence with the lowest convergence rate. According to the symmetry of the time variables, without loss of generality, we can assume $\{t_{1,k}\}_k$ \update{is the slowest sequence converging to zero compared to $\{t_{j,k}\}_k$ for any $j>1$, i.e., we assume that $\left\{\frac{t_{j,k}}{t_{1,k}}\right\}_k$ has a finite limit denoted as $\alplimj\in\R$ for any $j$.}
In summary, the following notations and assumptions are adopted:
\begin{equation}\label{eqt:seq_notation}
\left\{
\begin{aligned}
&x_k = x + \sum_{j=1}^N t_{j,k}v_{j,k}, \text{ where }t_{j,k}>0,\ x, v_{j,k}\in\Rn \text{ for any } j\in\{1,\cdots,N\} \text{ and } k\in \N;\\
&\lim_{k\to+\infty} t_{j,k} = 0\text{ and } \lim_{k\to+\infty}v_{j,k} = \vlimj;\\
&\lim_{k\to+\infty} \frac{t_{j,k}}{t_{1,k}} = \alplimj\in\R.
\end{aligned}
\right.
\end{equation}
\update{In the decomposition models, $\{x_k\}$ is given by a sequence of observed images. In each $x_k$ there is a constant component denoted by $x$ and several other components denoted by $t_{j,k}v_{j,k}$ for $j=1,\cdots, N$. In the remaining part of this section, we investigate the behavior of the minimizers of the decomposition model in \cref{eqt:lax} when the components $t_{j,k}v_{j,k}$ converge to zero and the parameters $t_{j,k}$ in the model vanish.}

First, we show the convergence of $u_j$ to zero\update{, which is stated in (i) in the following proposition}. \update{In other words, the decomposition model recovers the constant component $x$ when the other components $t_{j,k}v_{j,k}$ and the parameters $t_{j,k}$ in the model converge to zero. Then, (ii) and (iii) in the following proposition are technical results about the convergence rate, which will be used in later proofs.}

\begin{prop}\label{thm:u_conv}
Assume (H1)-(H2) and \cref{eqt:seq_notation} hold. Let $(u_1,\cdots, u_N)$ be any minimizer of the minimization problem in \cref{eqt:lax}. Let $x\in \dom J$. Then,
\begin{itemize}
    \item[(i)] For any $j=1,\cdots,N$,
    \begin{equation}
\lim_{k\to+\infty} u_j(x_k, t_{1,k},\cdots, t_{N,k}) = 0.
\end{equation}
\item[(ii)] If $\partial J(x)\neq \emptyset$ and $\alplimj= 0$, then
\begin{equation*}
    \lim_{k\to+\infty}\frac{1}{t_{1,k}}u_j(x_k, t_{1,k},\cdots,t_{N,k}) = 0.
\end{equation*}
\item[(iii)] If $\partial J(x)\neq \emptyset$ and $\alplimj\neq 0$, then the sequence $\left\{\frac{1}{t_{j,k}}u_j(x_k, t_{1,k},\cdots,t_{N,k})\right\}_k$ is bounded.
\end{itemize}
\end{prop}
\def\yb{\bar{u}}
\begin{proof}
Denote $\yb_{j,k} := u_j(x_k,t_{1,k},\cdots, t_{N,k})$ for any $j =1,\cdots, N$, and $\yb_{0,k} := x_k - \sum_{j=1}^N \yb_{j,k}$. Define $I := \{j: \{\|\yb_{j,k}\|/t_{j,k}\}_k \text{ is not bounded}\}$. \update{Recall that for each $j=1,\cdots, N$, $\{v_{j,k}\}_k\subset \R^n$ and $\{t_{j,k}\}_k\subset (0,+\infty)$ are two sequences satisfying $\lim_{k\to +\infty} v_{j,k} = v_{j,\infty}$ and $\lim_{k\to +\infty} t_{j,k} = 0$, respectively. And the $k-th$ spatial variable $x_k$ is defined to be $x + \sum_{j=1}^N t_{j,k}v_{j,k}$.}

Proof of (i):
By Lax formula \cref{eqt:lax}, 
\begin{equation} \label{eqt:15}
\begin{split}
J(\yb_{0,k}) + \sum_{j=1}^N t_{j,k}H_j^*\left(\frac{\yb_{j,k}}{t_{j,k}}\right)
&\leq 
J\left(x_k-\sum_{j=1}^N t_{j,k}v_{j,k}\right) + \sum_{j=1}^N t_{j,k}H_j^*(v_{j,k}) \\
&= J(x) + \sum_{j=1}^N t_{j,k}H_j^*(v_{j,k}).
\end{split}
\end{equation}
Since $J$ is a convex function, there exists $z\in \dom J$ such that $\partial J(z)\neq \emptyset$. Let $q\in \partial J(z)$. Then, using the convexity of $J$ and Cauchy-Schwarz inequality, we get
\begin{equation}\label{eqt:16}
J(\yb_{0,k}) \geq J(z)+\langle q, \ \yb_{0,k} - z\rangle \geq J(z) - \|q\|\sum_{j=1}^N \|\yb_{j,k}\| - \|q\| \|x_k-z\|.
\end{equation}
Combining \cref{eqt:15} and \cref{eqt:16}, we get
\begin{equation}\label{eqt:proof_u_conv_1}
\sum_{j=1}^N t_{j,k} H_j^*\left(\frac{\yb_{j,k}}{t_{j,k}}\right)
\leq J(x)-J(z) + \sum_{j=1}^N t_{j,k}H_j^*(v_{j,k}) + \|q\|\sum_{j=1}^N \|\yb_{j,k}\| + \|q\| \|x_k-z\|.
\end{equation}

For any $j\in I$, since $\|\yb_{j,k}\|/t_{j,k}$ is not bounded, without loss of generality, by taking subsequences, we can assume $\|\yb_{j,k}\|/t_{j,k}$ increases to infinity.
Since $H_j^*$ is 1-coercive, for any $M>0$, there exists $K$ such that for any $k>K$, $H_j^*(\yb_{j,k}/t_{j,k}) \geq M\|\yb_{j,k}\|/t_{j,k}$. Together with \cref{eqt:proof_u_conv_1}, we get
\begin{equation}\label{eqt:2}
\begin{split}
&\ \sum_{j\in I} (M-\|q\|)\|\yb_{j,k}\| \leq \sum_{j\in I} \left(t_{j,k} H_j^*\left(\frac{\yb_{j,k}}{t_{j,k}}\right)-\|q\|\|\yb_{j,k}\|\right)\\
\leq &\ J(x)-J(z) +\|q\| \|x_k-z\|+ \sum_{j=1}^N t_{j,k}H_j^*(v_{j,k})
 + \sum_{j\not\in I} \left(\|q\|\|\yb_{j,k}\|  - t_{j,k} H_j^*\left(\frac{\yb_{j,k}}{t_{j,k}}\right)\right).
\end{split}
\end{equation}
Since $\{t_{j,k}\}_k$ and $\{v_{j,k}\}_k$ are bounded, and $H_j^*$ is continuous in $\Rn$ for any $j$, then the right hand side is bounded. However, $M$ can be arbitrarily large, then the boundedness of left hand side (deduced by the boundedness of the right hand side) implies $\|\yb_{j,k}\|\to 0$ for any $j\in I$. If $j\not\in I$, then $\|\yb_{j,k}\|/t_{j,k}$ is bounded by the definition of $I$, hence $\yb_{j,k}$ also converges to zero.

\bigbreak

Proof of (ii): We can apply the same argument as above and set $z=x$, because $\partial J(x) \neq \emptyset$. From \cref{eqt:2}, using the definition of $x_k$ in \cref{eqt:seq_notation} and triangle inequality, we have
\begin{equation*}
\begin{split}
\sum_{j\in I} (M-\|q\|)\|\yb_{j,k}\| &\leq \|q\| \|x_k-x\|+ \sum_{j=1}^N t_{j,k}H_j^*(v_{j,k}) + \sum_{j\not\in I} \left(\|q\|\|\yb_{j,k}\|  - t_{j,k} H_j^*\left(\frac{\yb_{j,k}}{t_{j,k}}\right)\right)\\
&\leq 
\sum_{j=1}^N t_{j,k}\left(H_j^*(v_{j,k}) + \|q\|\|v_{j,k}\|\right) + \sum_{j\not\in I} t_{j,k}\left(\|q\|\left\|\frac{\yb_{j,k}}{t_{j,k}}\right\|  -  H_j^*\left(\frac{\yb_{j,k}}{t_{j,k}}\right)\right).
\end{split}
\end{equation*}
Dividing both sides by $t_{1,k}$, we can obtain
\begin{equation*}
\begin{split}
(M-\|q\|)\sum_{j\in I} \frac{\|\yb_{j,k}\|}{t_{1,k}} \leq 
&\sum_{j=1}^N \frac{t_{j,k}}{t_{1,k}}\left(H_j^*(v_{j,k}) + \|q\|\|v_{j,k}\|\right) \\
&+ \sum_{j\not\in I} \frac{t_{j,k}}{t_{1,k}}\left(\|q\|\left\|\frac{\yb_{j,k}}{t_{j,k}}\right\|  -  H_j^*\left(\frac{\yb_{j,k}}{t_{j,k}}\right)\right).
\end{split}
\end{equation*}
With the same argument as in the proof of (i), we deduce that the right hand side is bounded, while $M$ can be arbitrarily large. Therefore, $\|\yb_{j,k}\| / t_{1,k}$ converges to zero for any $j\in I$. If $j\not\in I$ and $\alplimj =0$, then $\|\yb_{j,k}\| / t_{j,k}$ is bounded by the definition of $I$ and $t_{j,k}/t_{1,k}$ converges to zero by the definition of $\alplimj$, hence $\|\yb_{j,k}\| / t_{1,k}$ also converges to zero.

\bigbreak
Proof of (iii):
It suffices to prove the contrapositive statement. To be specific, let $j\in I$, i.e. $\|\yb_{j,k}\| / t_{j,k}$ is unbounded, it suffices to prove $\alplimj = 0$.
In the proof of (ii), we know that $\|\yb_{j,k}\| / t_{1,k}$ converges to zero if $j\in I$. Then, the unboundedness of $\{\yb_{j,k}/t_{j,k}\}_k$ implies that $t_{j,k}/t_{1,k}$ converges to $0$, hence $\alplimj = 0$ and (iii) is proved.
\end{proof}
\bigbreak

\update{Similarly, we also consider the maximizers $\nabla_x S_H$ in the dual problem \cref{eqt:hopf} with the observed data $x_k$ and the parameters $\{t_{j,k}\}_{j=1}^N$. The following lemma states the boundedness of the maximizers $\{\nabla_x S_H(x_k,t_{1,k},\cdots,t_{N,k})\}_k$ which will be used in the later proofs.}

\begin{lem}\label{lem:bdd}
Under the assumptions (H1)-(H2) and \cref{eqt:seq_notation}, for any $x\in \dom J$ such that $\partial J(x)\neq \emptyset$,
the sequence $\{\nabla_x S_H(x_k,t_{1,k},\cdots,t_{N,k})\}_k$ is bounded and any cluster point $p$ is in $\partial J(x)$.
\end{lem}
\begin{proof}
\updatenew{Recall that for each $j\in \{1,\cdots, N\}$, $\{v_{j,k}\}_k\subset \R^n$ and $\{t_{j,k}\}_k\subset (0,+\infty)$ are two sequences satisfying the assumptions in \cref{eqt:seq_notation}.}
Denote $p_k := \nabla_x S_H(x_k, t_{1,k},\cdots, t_{N,k})$. Then, $p_k$ is a maximizer of the maximization problem in \cref{eqt:hopf}. Hence, for any $q$ in $\partial J(x)$,
\begin{equation*}
\langle x_k, p_k\rangle - J^*(p_k) - \sum_{j=1}^N t_{j,k} H_j(p_k)
\geq \langle x_k, q\rangle - J^*(q) - \sum_{j=1}^N t_{j,k} H_j(q).
\end{equation*}
Since $q\in\partial J(x)$, we have $x\in \partial J^*(q)$, hence $J^*(p_k)\geq J^*(q) + \langle x, p_k-q\rangle $. Combining this inequality and the above one we can obtain
\begin{equation*}
\sum_{j=1}^N t_{j,k} H_j(p_k) - \sum_{j=1}^N t_{j,k} H_j(q) \leq \langle x_k-x, p_k-q\rangle \leq \sum_{j=1}^N t_{j,k} \|v_{j,k}\|(\|p_k\| + \|q\|).
\end{equation*}
Here, for the second inequality above, we used the definition of $x_k$ in \cref{eqt:seq_notation} and Cauchy-Schwarz inequality. Then, rearranging the terms and dividing by $t_{1,k}$, we get
\begin{equation}\label{eqt:proof_bdd_1}
\sum_{j=1}^N \frac{t_{j,k}}{t_{1,k}} (H_j(p_k) - \|v_{j,k}\| \|p_k\|) \leq \sum_{j=1}^N \frac{t_{j,k}}{t_{1,k}} (H_j(q) + \|v_{j,k}\|\|q\|).
\end{equation}
If $\{p_k\}_k$ is not bounded, without loss of generality, we can assume $\|p_k\|$ increases to infinity.
Since $H_j$ is 1-coercive for all $j$, then for any $M>0$, there exists $K$ such that $H_j(p_k)\geq M\|p_k\|$ for any $k>K$ and any $j=1,\cdots, N$. Then, from \cref{eqt:proof_bdd_1}, for any $k>K$, we obtain
\begin{equation*}
\sum_{j=1}^N \frac{t_{j,k}}{t_{1,k}} (M - \|v_{j,k}\|) \|p_k\| \leq \sum_{j=1}^N \frac{t_{j,k}}{t_{1,k}} (H_j(q) + \|v_{j,k}\|\|q\|).
\end{equation*}
The right hand side is bounded. However, since $\|p_k\|$ goes to infinity, the term for $j=1$ on the left hand side is unbounded, while the terms for $j>1$ is non-negative. As a result, the left hand side can be arbitrarily large, which leads to a contradiction.
Therefore, we can conclude that $\{p_k\}_k$ is bounded.

For the remaining part, let $p$ be a cluster point, then there exists a subsequence converging to $p$, still denoted as $p_k$. Since $S_H$ solves the multi-time HJ equation \cref{eqt:H-J} and $H_j$ is continuous for any $j$, then we have
\begin{equation*}
\begin{split}
    \lim_{k\to+\infty} \nabla S_H(x_k, t_{1,k},\cdots, t_{N,k})
    &=\lim_{k\to+\infty} (p_k, -H_1(p_k), \cdots, -H_N(p_k))\\
    &=(p,-H_1(p), \cdots, -H_N(p)).
\end{split}
\end{equation*}
By the continuity property \cite[Prop.XI.4.1.1]{convex_analy_book2} of the subdifferential operator $\partial S_H$ of the convex lower semi-continuous function $S_H$, we can conclude that 
\begin{equation*}
    (p,-H_1(p), \cdots, -H_N(p)) \in \partial S_H(x,0,\cdots, 0),
\end{equation*}
which implies $p\in \partial J(x)$.
\end{proof}

The variational behaviors of the momentum $\nabla_x S$ and the velocities $u_j/t_j$ are presented in the following proposition. To be specific, the cluster points of the momenta and the velocities solve two optimization problems, respectively, and the two problems are dual to each other. An illustration of this result is given in \cref{fig:illustrationprop}.

\begin{prop} \label{thm:conv_grad}
Assume (H1)-(H2) and \cref{eqt:seq_notation} hold. Let $x\in \dom J$ and $\partial J(x)\neq \emptyset$. Then,

\begin{itemize}
\item[(i)] the directional derivative of $S_H$ corresponds to a maximization problem:
\begin{equation}\label{eqt:dirS}
\lim_{k\to+\infty} \frac{S_H(x_k,t_{1,k},\cdots, t_{N,k}) - S_H(x,0,\cdots,0)}{t_{1,k}} = \max_{q\in \partial J(x)} \sum_{j=1}^N \alplimj\left(\langle q, \vlimj\rangle - H_j(q)\right).
\end{equation}
Moreover, let $p$ be any cluster point of $\{\nabla_x S_H(x_k,t_{1,k}, \cdots, t_{N,k})\}_k$, then,
\begin{equation}\label{eqt:dirS1}
p\in \argmax_{q\in \partial J(x)}\sum_{j=1}^N\alplimj\left(\langle q,\vlimj\rangle - H_j(q)\right).
\end{equation}
\item[(ii)] the directional derivative of $S_L$ corresponds to the dual minimization problem:
\begin{equation}\label{eqt:Sdirpf0}
\begin{split}
&\lim_{k\to+\infty} \frac{S_L(x_k,t_{1,k},\cdots, t_{N,k}) - S_L(x,0,\cdots,0)}{t_{1,k}} \\
= &\min_{w_j\in\Rn} \sum_{j=1}^N \alplimj (I_{\partial J(x)}^*(\vlimj - w_j) + H_j^*(w_j)).
\end{split}
\end{equation}
Moreover, if $\bar{w}_j$ is a cluster point of $\{u_j(x_k,t_{1,k},\cdots, t_{N,k})/t_{j,k}\}_k$ for any $j$ satisfying $\alplimj \neq 0$, then 
\begin{equation}\label{eqt:dirS2}
\bar{w}_j \in \argmin_{w_j\in\Rn} \left(I_{\partial J(x)}^*(\vlimj - w_j) + H_j^*(w_j)\right).
\end{equation}
\end{itemize}
\end{prop}

Specially, if $H_j$ is strictly convex and $\alplimj \neq 0$ for some $j$, then the maximizer in \cref{eqt:dirS1} is unique, which implies the convergence of $\nabla_x S_H(x_k, t_{1,k},\cdots, t_{N,k})$ to the unique maximizer. Similarly, for any $j$ such that $H_j$ is differentiable and $\alplimj \neq 0$, we can conclude that $u_j(x_k, t_{1,k},\cdots, t_{N,k})/t_{j,k}$ converges to the unique minimizer in \cref{eqt:dirS2}.

\begin{remark}
It is straightforward to obtain $\lim_{k\to+\infty} \frac{S_H(x_k,t_{1,k},\cdots, t_{N,k}) - S_H(x,0,\cdots,0)}{\|(t_{1,k}, \cdots, t_{N,k})\|_2}$ using the following computation
\begin{equation*}
\begin{split}
    &\lim_{k\to+\infty} \frac{S_H(x_k,t_{1,k},\cdots,t_{N,k}) - S_H(x,0,\cdots,0)}{\|(t_{1,k}, \dots, t_{N,k})\|_2}\\
    =\ & \lim_{k\to+\infty} \frac{S_H(x_k,t_{1,k},\cdots,t_{N,k}) - S_H(x,0,\cdots,0)}{t_{1,k}} \cdot \frac{t_{1,k}}{\|(t_{1,k}, \dots, t_{N,k})\|_2}\\
    =\ & \lim_{k\to+\infty} \frac{S_H(x_k,t_{1,k},\cdots,t_{N,k}) - S_H(x,0,\cdots,0)}{t_{1,k}} \cdot \frac{1}{\|(\alpha_{1,\infty}, \dots, \alpha_{N,\infty})\|_2},
\end{split}
\end{equation*}
where the last equality follows from the assumption that $\alpha_{j,\infty} = \lim_{k\to+\infty} t_{j,k}/ t_{1,k}$ for any $j=1,\cdots, N$.
\end{remark}

\begin{proof}
\updatenew{Recall that the $k-th$ spatial variable $x_k$ is defined to be $x + \sum_{j=1}^N t_{j,k}v_{j,k}$, where $\{v_{j,k}\}_k\subset \R^n$ and $\{t_{j,k}\}_k\subset (0,+\infty)$ are two sequences satisfying the assumptions in \cref{eqt:seq_notation}.}
Denote $\Delta S_k := S_H(x_k,t_{1,k},\cdots, t_{N,k}) - S_H(x,0,\cdots,0)$. 

Proof of (i): For any $q\in \partial J(x)$, by Hopf formula \cref{eqt:hopf}, we obtain
\begin{equation*}
\begin{split}
\Delta S_k &
= \left(J^* + \sum_{j=1}^N t_{j,k} H_j\right)^*(x_k) - J(x)
 \geq  \langle q,x_k\rangle - J^*(q) - \sum_{j=1}^N t_{j,k} H_j(q) - J(x).
\end{split}
\end{equation*}
Since $q\in\partial J(x)$, we have $J^*(q)+J(x)=\langle q,x\rangle$. Hence, together with the definition of $x_k$ in \cref{eqt:seq_notation}, we get
\begin{equation*}
\Delta S_k \geq \langle q,x_k-x\rangle - \sum_{j=1}^N t_{j,k} H_j(q) = \sum_{j=1}^N t_{j,k}(\langle q, v_{j,k}\rangle -H_j(q)).
\end{equation*}
Therefore, we have
\begin{equation*}
\liminf_{k\to+\infty} \frac{\Delta S_k}{t_{1,k}} \geq \liminf_{k\to+\infty} \sum_{j=1}^N \frac{t_{j,k}}{t_{1,k}}(\langle q, v_{j,k}\rangle -H_j(q)) = \sum_{j=1}^N \alplimj (\langle q, \vlimj \rangle -H_j(q)),
\end{equation*}
where we recall that $\lim_{k\to+\infty}v_{j,k} = \vlimj$ and $\lim_{k\to+\infty} t_{j,k}/t_{1,k} = \alplimj$ by \cref{eqt:seq_notation}.
Here, $q$ is an arbitrary element in $\partial J(x)$, hence we obtain
\begin{equation}\label{eqt:Sdir_pf1}
\liminf_{k\to+\infty} \frac{\Delta S_k}{t_{1,k}} \geq \sup_{q\in\partial J(x)} \sum_{j=1}^N \alplimj (\langle q, \vlimj \rangle -H_j(q)).
\end{equation}

On the other hand, for any $k$, consider the function $\phi_k: [0,+\infty)\to\R$ defined by $\phi_k(t) := S_H\left(x + \sum_{j=1}^N t\alpha_{j,k} v_{j,k},\ \alpha_{1,k}t, \cdots,\ \alpha_{N,k}t\right)$, where $\alpha_{j,k}:= t_{j,k}/t_{1,k}$. Since $S_H$ is a convex function and $\phi_k$ is its restriction on a line, then $\phi_k\in\Gamma_0(\R)$ with $\dom \phi_k = [0,+\infty)$. Also, $\phi_k$ is differentiable in $(0,+\infty)$ since $S_H$ is differentiable. The derivative of $\phi_k$ at $t_{1,k}$ is given by the chain rule:
\begin{equation*}
\phi_k'(t_{1,k}) = \sum_{j=1}^N \alpha_{j,k} \left(\langle \nabla_x S_H(x_k,t_{1,k},\cdots, t_{N,k}),\  v_{j,k}\rangle + \frac{\partial S_H}{\partial t_j}(x_k,t_{1,k},\cdots, t_{N,k})\right).
\end{equation*}
Since $S_H$ satisfies the multi-time HJ equation \cref{eqt:H-J}, we obtain
\begin{equation*}
\phi_k'(t_{1,k}) = \sum_{j=1}^N \alpha_{j,k} \left(\langle \nabla_x S_H(x_k,t_{1,k},\cdots, t_{N,k}), \ v_{j,k}\rangle - H_j(\nabla_x S_H(x_k,t_{1,k},\cdots, t_{N,k}))\right).
\end{equation*}
From straightforward computation and the convexity of $\phi_k$, we get
\begin{equation}
\frac{\Delta S_k}{t_{1,k}}
= \frac{\phi_k(t_{1,k}) - \phi_k(0)}{t_{1,k}}
\leq \phi_k'(t_{1,k}) = \sum_{j=1}^N \alpha_{j,k} \left(\langle p_k, \ v_{j,k}\rangle - H_j(p_k)\right),
\end{equation}
where $p_k := \nabla_x S_H(x_k,t_{1,k},\cdots, t_{N,k})$.

Let $p$ be a cluster point of $\{p_k\}$. Take a subsequence converging to $p$ and still denote it as $\{p_k\}$. Since $p\in \partial J(x)$ by \cref{lem:bdd} and $H_j$ is continuous for any $j$, we have
\begin{equation*}
\limsup_{k\to+\infty} \frac{\Delta S_k}{t_{1,k}} \leq \sum_{j=1}^N \alplimj \left(\langle p, \vlimj\rangle - H_j(p)\right)\leq \sup_{q\in\partial J(x)} \sum_{j=1}^N \alplimj (\langle q, \vlimj \rangle -H_j(q)).
\end{equation*}
Together with \cref{eqt:Sdir_pf1}, the equation \cref{eqt:dirS} is proved. Moreover, any cluster point $p$ is a maximizer.

\bigbreak
Proof of (ii):
Here, we adopt the notations $\bar{u}_{j,k}$ and $\bar{u}_{0,k}$ defined in the proof of \cref{thm:u_conv} to represent the minimizers in the Lax formula.
According to the Lax formula \cref{eqt:lax} evaluated at the point $(x_k,t_{1,k}, \cdots, t_{N,k})$ and by the convexity of $J$ we deduce that
\begin{equation*}
S_L = J(\bar{u}_{0,k}) + \sum_{j=1}^N t_{j,k} H_j^*\left(\frac{\bar{u}_{j,k}}{t_{j,k}}\right) \geq J(x) + \langle q, \bar{u}_{0,k} - x\rangle + \sum_{j=1}^N t_{j,k} H_j^*\left(\frac{\bar{u}_{j,k}}{t_{j,k}}\right),
\end{equation*}
for any $q \in \partial J(x)$.
Since $S_L = S_H$, we have $S_L(x_k,t_{1,k},\cdots,t_{N,k}) - S_L(x,0,\cdots,0) =\Delta S_k$. By the definition of $x_k$ and $\bar{u}_{0,k}$, we can compute $\bar{u}_{0,k} - x = x_k - x -\sum_j \bar{u}_{j,k} = \sum_j (t_{j,k}v_{j,k} - \bar{u}_{j,k})$, hence we have
\begin{equation*} 
\frac{\Delta S_k}{t_{1,k}} \geq \sum_{j=1}^N \left(\alpha_{j,k}\langle q, v_{j,k}\rangle - \left\langle q, \frac{\bar{u}_{j,k}}{t_{1,k}}\right\rangle + \alpha_{j,k}H_j^*\left(\frac{\bar{u}_{j,k}}{t_{j,k}}\right) \right),
\end{equation*}
where $\alpha_{j,k}:=t_{j,k}/t_{1,k}$.
According to \cref{thm:u_formula} we have $\bar{u}_{j,k}/t_{j,k}\in \partial H_j(p_k)$. Therefore we get
\begin{equation*}
    \alpha_{j,k}H_j^*\left(\frac{\bar{u}_{j,k}}{t_{j,k}}\right)
    = \alpha_{j,k} \left(\left\langle \frac{\bar{u}_{j,k}}{t_{j,k}}, p_k \right\rangle - H_j(p_k)\right)
    = \left\langle \frac{\bar{u}_{j,k}}{t_{1,k}}, p_k \right\rangle - \alpha_{j,k}H_j(p_k).
\end{equation*}
Combining the above two equations we obtain
\begin{equation} \label{eqt:Sdirpf2}
    \begin{split}
    \frac{\Delta S_k}{t_{1,k}} \geq &\sum_{\substack{j=1\\ \alplimj = 0}}^N \left(\alpha_{j,k}\langle q, v_{j,k}\rangle + \left\langle p_k - q, \frac{\bar{u}_{j,k}}{t_{1,k}}\right\rangle - \alpha_{j,k}H_j(p_k) \right)\\
    &+ \sum_{\substack{j=1\\ \alplimj \neq 0}}^N \left(\alpha_{j,k}\left\langle q, v_{j,k} - \frac{\bar{u}_{j,k}}{t_{j,k}}\right\rangle + \alpha_{j,k}H_j^*\left(\frac{\bar{u}_{j,k}}{t_{j,k}}\right) \right).
    \end{split}
\end{equation}
From \cref{thm:u_conv} (ii), $\|\bar{u}_{j,k}\| / t_{1,k}$ converges to zero if $\alplimj = 0$. Also, $p_k$ are bounded by \cref{lem:bdd}, hence the first sum in the right hand side of \cref{eqt:Sdirpf2} converges to zero as $k$ approaches infinity.
On the other hand, for $j$ such that $\alplimj \neq 0$, $\bar{u}_{j,k}/t_{j,k}$ is bounded by \cref{thm:u_conv} (iii).
Taking a subsequence, we can assume that $\bar{u}_{j,k}/t_{j,k}$ converges to some vector, denoted as $\bar{w}_j$. 
In conclusion, as $k$ approaches infinity in \cref{eqt:Sdirpf2}, we have
\begin{equation}\label{eqt:Sdirpf3}
\lim_{k\to+\infty}\frac{\Delta S_k}{t_{1,k}} \geq \sum_{\substack{j=1\\ \alplimj \neq 0}}^N \alplimj\left(\left\langle q, \vlimj - \bar{w}_j\right\rangle + H_j^*\left(\bar{w}_j\right) \right) \geq \sum_{j=1}^N \alplimj (\langle q, \vlimj \rangle -H_j(q)),
\end{equation}
where the second inequality holds by the definition of Legendre transform \cref{eqt:Legendre}.
From \cref{eqt:dirS}, for any maximizer $p$ in \cref{eqt:dirS1},
\begin{equation}\label{eqt:Sdirpf4}
\lim_{k\to+\infty}\frac{\Delta S_k}{t_{1,k}} = \sum_{j=1}^N \alplimj (\langle p, \vlimj \rangle -H_j(p)).
\end{equation}

Taking $q=p$ in \cref{eqt:Sdirpf3} and comparing it with \cref{eqt:Sdirpf4}, we can conclude that the inequalities in \cref{eqt:Sdirpf3} become equalities when $q=p$. As a result, when $\alplimj \neq 0$ we have $\langle p, \bar{w}_j\rangle = H_j^*(\bar{w}_j) + H_j(p)$, which implies that $p\in \partial H_j^*(\bar{w}_j)$. Then, we deduce that
\begin{equation}\label{eqt:Sdirpf5}
 \lim_{k\to+\infty}\frac{\Delta S_k}{t_{1,k}} =
    \sum_{\substack{j=1\\ \alplimj \neq 0}}^N \alplimj\left(\left\langle p, \vlimj - \bar{w}_j\right\rangle + H_j^*\left(\bar{w}_j\right) \right).
\end{equation}

On the other hand, for an arbitrary $q\in \partial J(x)$, by \cref{eqt:Sdirpf3} and \cref{eqt:Sdirpf5}, we have
\begin{equation*}
\begin{split}
\sum_{\substack{j=1\\ \alplimj \neq 0}}^N \alplimj\left(\left\langle p, \vlimj - \bar{w}_j\right\rangle + H_j^*\left(\bar{w}_j\right) \right)
&= \lim_{k\to+\infty}\frac{\Delta S_k}{t_{1,k}} \\
&\geq \sum_{\substack{j=1\\ \alplimj \neq 0}}^N \alplimj\left(\left\langle q, \vlimj - \bar{w}_j\right\rangle + H_j^*\left(\bar{w}_j\right) \right),
\end{split}
\end{equation*}
which implies that $\langle p-q, \vlimj - \bar{w}_j\rangle \geq 0$ for any $q\in \partial J(x)$, when $\alplimj \neq 0$. By \cref{eqt:def_normal_cone} and \cref{eqt:subgrad_indicator}, we can deduce that $\vlimj - \bar{w}_j \in N_{\partial J(x)}(p) = \partial I_{\partial J(x)}(p)$. \cref{prop3} gives the equality $\langle p, \vlimj - \bar{w}_j\rangle = I_{\partial J(x)}^*(\vlimj - \bar{w}_j)$. Then, \cref{eqt:Sdirpf0} follows from this equality and \cref{eqt:Sdirpf5}.

It remains to prove \cref{eqt:dirS2}. Consider any $j$ such that $\alplimj \neq 0$. Define $f: \Rn \to \R$ by $f(w):= I_{\partial J(x)}^*(\vlimj - w) + H_j^*(w)$.
Then it suffices to prove $0\in \partial f(\bar{w}_j)$.
So far, we have proved $p\in \partial H_j^*(\bar{w}_j)$ and $\vlimj - \bar{w}_j\in \partial I_{\partial J(x)}(p)$, which implies $p\in \partial I_{\partial J(x)}^*(\vlimj - \bar{w}_j)$. By straightforward computation and \cref{prop7},
\begin{equation*}
    \partial f(\bar{w}_j) = -\partial I_{\partial J(x)}^*(\vlimj - \bar{w}_j) + H_j^*(\bar{w}_j) \ni -p+p = 0.
\end{equation*}
Therefore, $\bar{w}_j$ is a minimizer of $f$, which concludes the proof.
\end{proof}

\update{The above proposition provides the explicit formulas for the variations of $S$, $\nabla_x S$ and $\frac{u_j}{t_j}$ where $u_j$ denotes the $j$-th component of the minimizer of the decomposition model in the form of \cref{eqt:lax}. Specifically, the limits of these quantities are related to the two optimization problems given by \cref{eqt:dirS1,eqt:dirS2}. 
From the perspective of image processing, given an observed image $x_k$ which is a summation of a constant component $x$ and other components $t_{j,k}v_{j,k}$, the decomposition model \cref{eqt:lax} gives $N+1$ components. In these $N+1$ components, one component converges to the constant component $x$ and the other components $u_j$ vanish as the parameters $t_{j,k}$ approach zero, by \cref{thm:u_conv}. Then, \cref{thm:conv_grad}(ii) states that the component $u_j$ converges to $0$ from a direction  $\bar{w}_j$ \cite[p. 197]{RockWets98}.  On the other hand, \cref{thm:conv_grad}(i) provides a representation formula for the cluster point of the maximizers of the dual problem in the form of \cref{eqt:hopf}.}
\section{Uniqueness of the Convex Solutions to the Multi-time Hamilton-Jacobi Equations} \label{sec:uniq}
In the previous section, we have discussed the relation of the optimization problems in the Hopf formula and Lax formula with the classical solution of the multi-time HJ equation. In fact, some results can be generalized to weaker assumptions in which case the solution provided by Hopf and Lax formulas is not classical. In this section, we prove that the only convex solution is given by the two formulas. 

In the field of PDEs, a type of solution called viscosity solution is considered for solving the HJ equation when no classical solution exists. The uniqueness of the viscosity solution has been widely studied under different assumptions \update{\cite{Bardi1997, barles1994solutions}}. However, the functions in convex analysis and optimization may take the value $+\infty$, which is an unusual condition in the field of PDEs. Therefore, to maintain the connection of the HJ equations and convex optimization problems, we consider the convex solution which may be infinity in some area and prove the uniqueness using the techniques in convex analysis.

We start with the proof for the classical convex solution, in order to demonstrate the idea of utilizing the convexity assumptions. After that, we state the uniqueness of nonsmooth convex solution under more general assumptions in \cref{cor:uniq}. 
When proving the uniqueness of the classical convex solution, we assume the properties (H1) and (H2) hold. Moreover, the solution $S$ satisfies:
\setlist[enumerate]{leftmargin=.5in}
\setlist[itemize]{leftmargin=.5in}
\begin{itemize}
\item[(S1) ] $S\in\Gamma_0\left(\Rn\times [0,+\infty)^N\right)\cap C^1(\Rn\times (0,+\infty)^N)$;
\item[(S2) ] $S$ solves the multi-time Hamilton-Jacobi equation \cref{eqt:H-J}.
\end{itemize}
\setlist[enumerate]{leftmargin=.25in}
\setlist[itemize]{leftmargin=.25in}

As it is discussed in \cref{sec:prop}, $S_H$ defined in the Hopf formula \cref{eqt:hopf} is a solution satisfying the assumptions (S1) and (S2). Hence, we just need to prove $S=S_H$ for any $S$ satisfying (S1)-(S2). 
First, we consider the single-time case when the time dimension $N=1$, and formulate its Legendre transform $S^*(p,E^-)$ for $p\in\Rn$ and $E^-\in\R$ in the following lemma.

\begin{lem}\label{lem:1}
Assume (H1)-(H2) hold and $S$ satisfies (S1)-(S2). Let $N=1$. Then there exists a convex function $\tilde{H}:\ \Rn\to \R\cup\{+ \infty\}$, 
such that 
$S^*(p,E^-) = J^*(p) + I_{V}(p,E^-)$, where $V := \{(p,E^-):\ E^-\leq -\tilde{H}(p)\}$.
\end{lem}

\begin{proof}
\begin{figure}[htbp]
\centering
\subfloat[]{\label{fig:illustrationpf1}\includegraphics[width = 0.3\textwidth]{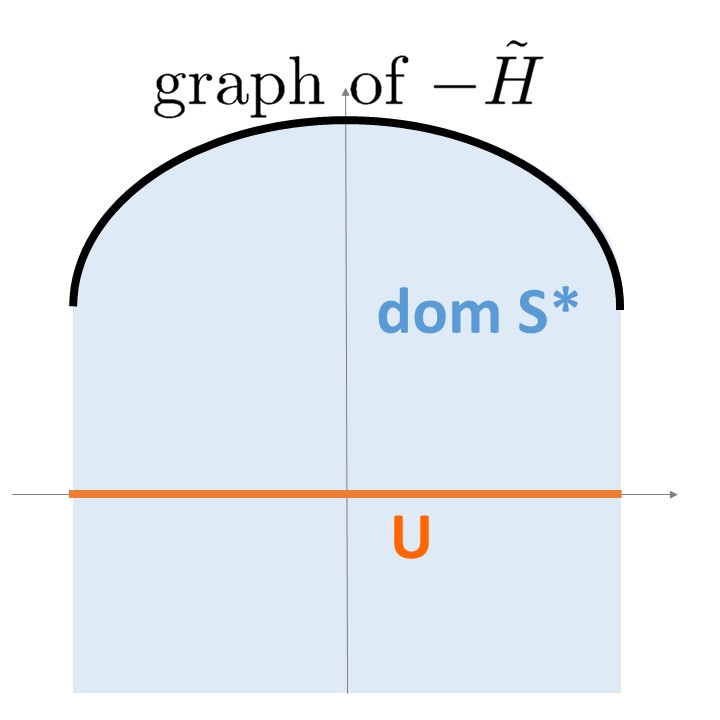}}
\subfloat[]{\label{fig:illustrationpf2}\includegraphics[width = 0.3\textwidth]{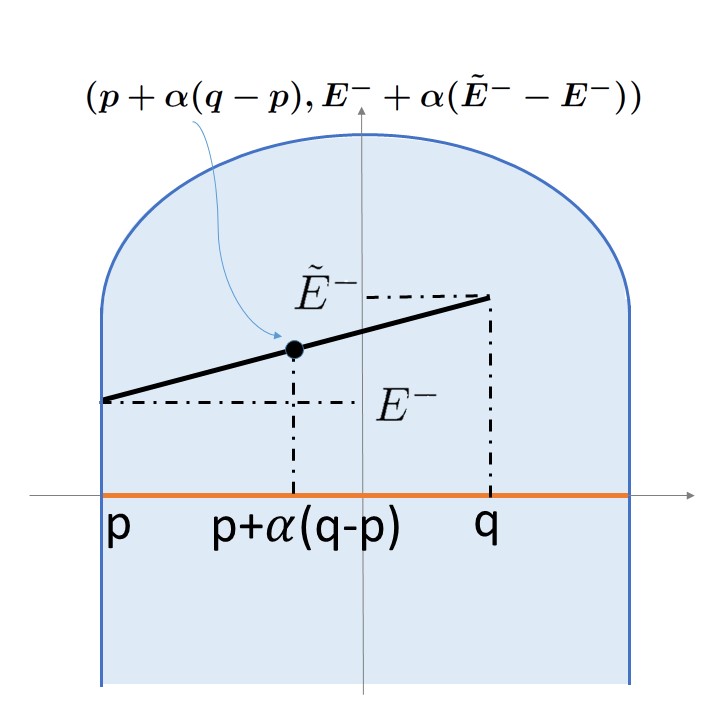}}
\subfloat[]{\label{fig:illustrationpf3}\includegraphics[width = 0.3\textwidth]{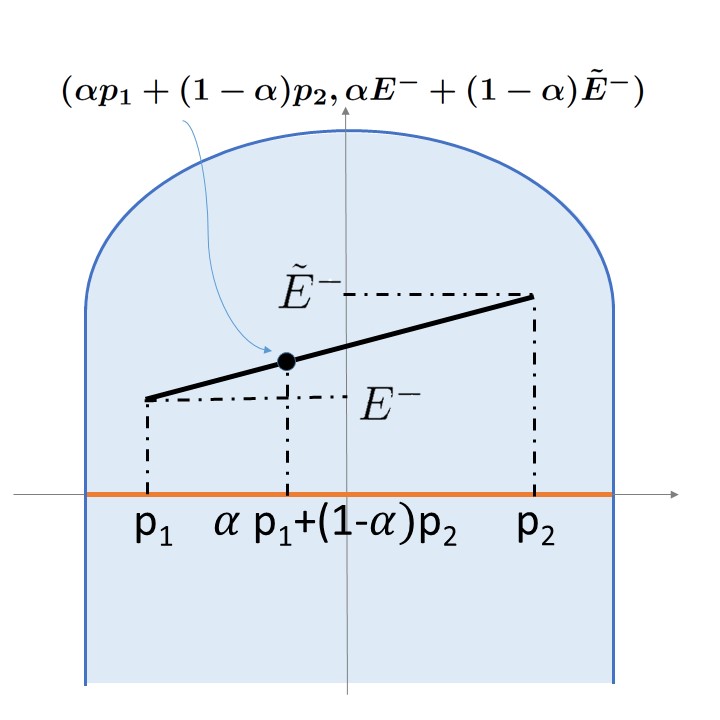}}
\caption{Illustrations for different steps in the proof of \cref{lem:1}.}
\label{fig:illustrationpf}
\end{figure}
In this proof, we only consider the single-time HJ equation. For the single-time case, $H$ is used to denote the Hamiltonian, instead of $H_1$, for simplicity.
First, consider the domain of $S^*$. For each $p\in\Rn$, define 
\begin{equation}\label{eqt:lem41_defHtilde}
\tilde{H}(p) := \inf\{-E^-:\ (p,E^-)\in \dom S^*\}\in \bar{\R} :=\R\cup\{\pm\infty\}.
\end{equation}
For the illustration of this definition, see \cref{fig:illustrationpf1}.
The function $\tilde{H}$ defined here is an extended-valued function taking values in $\bar{\R}$. In the last step of this proof, we will show the convexity and specify the range of this function.
From this definition, it is obvious that $\dom S^*\subseteq V$, where $V = \{(p,E^-):\ E^-\leq -\tilde{H}(p)\}$, as defined in the statement of this lemma. Moreover, denote $V_1 = \{(p,E^-):\ E^- < -\tilde{H}(p)\}$, then we prove $V_1\subseteq \dom S^*$ by using the monotonicity of $S^*(p,\cdot)$. To be specific, let $p\in \Rn$ and $-\infty<\tilde{E}^-\leq E^-<+\infty$, then, we have
\begin{equation}\label{eqt:monotone}
S^*(p,\tilde{E}^-) = \sup_{x\in\Rn, t\geq 0} \langle p,x\rangle + t\tilde{E}^- -S(x,t)
\leq \sup_{x\in\Rn, t\geq 0} \langle p,x\rangle + tE^- -S(x,t) = S^*(p,E^-).
\end{equation}
Hence, $S^*(p,E^-)$ is non-decreasing with respect to $E^-$. As a result, $(p,E^-)\in \dom S^*$ implies $\{p\}\times (-\infty, E^-]\subseteq \dom S^*$. Therefore we obtain $V_1\subseteq \dom S^*\subseteq V$. 

\bigbreak
In the next step, we prove $\dom S^* = V$.

Denote $U := \{p\in \Rn:\ \tilde{H}(p) < +\infty\}$ (see \cref{fig:illustrationpf1}). \update{Here and after in  this section, we use the bold character $\mathbf{0}$ to denote the zero vector in $\Rn$.} Since $U$ is the projection of $\dom S^*$ along the direction $(\mathbf{0},1)$, $U$ is a convex set. Let $p\in \ri U$. Take $E^-<-\tilde{H}(p)$, then $\partial S^*(p,E^-) \neq \emptyset$ because $(p,E^-)\in\ri\dom S^*$. Let $(x,t)\in \partial S^*(p,E^-)$, which implies $(p,E^-)\in \partial S(x,t)$. If $t>0$, then $E^-=\frac{\partial S}{\partial t}(x,t)$ and $p=\nabla_x S(x,t)$. Since $S$ satisfies the HJ equation \cref{eqt:H-J}, $E^-+H(p)=0$. In other words, if $(x,t)\in \partial S^*(p,E^-)$ with $E^-\neq -H(p)$, then we can conclude that $t=0$. Therefore, for any $E^-<-\tilde{H}(p)$ and $E^-\neq -H(p)$, by \cref{prop5}, the directional derivative of $S^*$ in the direction $(\mathbf{0},1)$ is:
\begin{equation*}
(S^*)'((p,E^-),(\mathbf{0},1)) = \sup_{(x,t)\in \partial S^*(p,E^-)}\langle (x,t),(\mathbf{0},1)\rangle = \sup_{(x,t)\in \partial S^*(p,E^-)}\langle (x,0),(\mathbf{0},1)\rangle = 0.
\end{equation*}
As a result, $S^*(p,\cdot)$ is a constant function in its domain. Denote this value as $f(p)$. By the continuity of $S^*$ when restricting to the straight line $\{p\}\times \R$, the value $S^*(p,-\tilde{H}(p))$ is also $f(p)$ if $\tilde{H}(p)$ is finite. Hence, $S^*(p,E^-) = f(p)$ for any $p\in\ri U$ and $E^-\leq -\tilde{H}(p)$.

Now, we consider the case when $p\in U \setminus \ri U$. For the illustration, see \cref{fig:illustrationpf2}. Let $E^- < -\tilde{H}(p)$. Take $q\in \ri U$ and $\tilde{E}^- < -\tilde{H}(q)$, then by \cref{prop1},
\begin{equation}\label{eqt:bdry}
S^*(p, E^-) = \lim_{\alpha \to 0^+} S^*(p + \alpha (q-p), E^- + \alpha (\tilde{E}^- - E^-)) = \lim_{\alpha \to 0^+} f(p + \alpha (q-p)).
\end{equation}
Hence, the value of $S^*(p, E^-)$ does not depend on $E^-$ if $E^- < -\tilde{H}(p)$. Denote this value as $f(p)$. By continuity, $S^*(p,-\tilde{H}(p))=f(p)$ if $\tilde{H}(p)$ is finite. Therefore, we have proved that the domain of $S^*$ coincides with the set $V$ and $S^*(p, E^-) = f(p)$ in the domain of $S^*$.

\bigbreak

Then, we prove $f=J^*$ when restricting to $\dom f$. By setting $f(p)=+\infty$ if $p\not\in U$, we can regard $f$ as a function from $\Rn$ to $\R\cup\{+\infty\}$. It is not hard to check the convexity of $f$.
To be specific, for any $p_1,p_2\in \dom f$ and $\alpha\in (0,1)$, choose $E^- < -\tilde{H}(p_1)$ and $\tilde{E}^- < -\tilde{H}(p_2)$ (see \cref{fig:illustrationpf3}), then we have 
\begin{equation*}
\begin{split}
f(\alpha p_1 + (1-\alpha) p_2) &= S^*(\alpha p_1 + (1-\alpha) p_2, \alpha E^- + (1-\alpha) \tilde{E}^-) \\
&\leq \alpha S^*(p_1, E^-) + (1-\alpha) S^*(p_2, \tilde{E}^-) \\
&= \alpha f(p_1) + (1-\alpha) f(p_2).
\end{split}
\end{equation*}
Hence $f$ is a convex function taking values in $\R\cup \{+\infty\}$. 
Also, for each $x\in\Rn$, we have
\begin{equation*}
\begin{split}
    J(x) &= S(x,0) = \sup_{\update{(p,E^-)\in V}} \langle x, p\rangle - S^*(p,E^-) \\
    &= \sup_{\update{\substack{p\in\Rn \\ \exists E^- \,|\, (p,E^-)\in V}}} \langle x, p\rangle - f(p) 
=\sup_{p\in \dom f} \langle x, p\rangle - f(p)= f^*(x).
\end{split}
\end{equation*}
Therefore, $f^{**}=J^*$, which implies $\ri \dom f = \ri\dom J^*$ and $f(p)=J^*(p)$ if $p\in \ri\dom f$. Moreover,
according to \cref{prop1} and \cref{eqt:bdry},
we deduce that
\begin{equation*}
f(p) = \lim_{\alpha \to 0^+} f(p + \alpha (q-p))
= \lim_{\alpha \to 0^+} J^*(p + \alpha (q-p)) = J^*(p),
\end{equation*}
for any $p\in \dom f \setminus \ri\dom f$ and $q\in \ri\dom f$. As a result we have $f=J^*$ in the domain of definition.
In conclusion, we get the following formula for $S^*$
\begin{equation}\label{eqt:pflem41_formSstar}
    S^*(p,E^-) = J^*(p) + I_{V}(p,E^-).
\end{equation}

\bigbreak
The final part is to prove that $\tilde{H}$ is a convex function taking values in $\R\cup\{+\infty\}$.

\update{
First, we prove that $\tilde{H}$ cannot take the value $-\infty$ by contradiction. Suppose there exists $p\in\Rn$ such that $\tilde{H}(p)$ equals $-\infty$. Then, by definition of $\tilde{H}$ we have $\{p\}\times \R\subseteq \dom S^*$. Together with the formula of $S^*$ in \cref{eqt:pflem41_formSstar}, we derive
\begin{equation*}
\{p\}\times \R \times \{J^*(p)\} \subseteq \epi S^*.    
\end{equation*}
Therefore, $(\mathbf{0},1,0)$ and $(\mathbf{0},-1,0)$ are in the asymptotic cone of $\epi S^*$ by definition \cref{eqt:defasym}. Then, by \cref{prop:asym}, for any $q\in U$, we obtain
\begin{equation*}
\{q\}\times \R \times \{J^*(q)\} \subseteq \epi S^*,    
\end{equation*}
which implies $\{q\} \times \R \subseteq \dom S^*$. Since $q$ is an arbitrary vector in $U$, we deduce that $\dom S^* = U \times \R$. Moreover, according to \cref{eqt:pflem41_formSstar}, the function $S^*$ is a constant on the line $\{q\}\times \R$ for any $q\in U$, which implies that the directional derivative of $S^*$ in the direction $(\mathbf{0},1)$ is zero. In other words, we have}
\begin{equation}\label{eqt:dirderivativeS}
    (S^*)'((p,E^-),(\mathbf{0},1)) = 0 \text{ for any }p\in U\text{ and }E^-\in \R.
\end{equation}
On the other hand, consider any $y\in \Rn$ and $s>0$ such that $\partial S(y,s)$ is nonempty. Let $(p,E^-)\in \partial S(y,s)$. This implies $(y,s)\in \partial S^*(p,E^-)$. Hence, according to \cref{prop5}, we get
\begin{equation*}
(S^*)'((p,E^-),(\mathbf{0},1)) \geq \sup_{(x,t)\in \partial S^*(p,E^-)} \langle (x,t), (\mathbf{0},1)\rangle\geq \langle (y,s), (\mathbf{0},1)\rangle = s > 0,
\end{equation*}
which contradicts \cref{eqt:dirderivativeS}. Therefore, $\tilde{H}$ cannot take the value $-\infty$. 

At last, the convexity of $\tilde{H}$ follows from the convexity of $\dom S^*$.
In fact, $\epi \tilde{H} = \{(p, -E^-):\ (p,E^-)\in \dom S^*\}$, which is a reflection of the convex set $\dom S^*$, hence it is also convex. Therefore, $\tilde{H}$ is a convex function from $\Rn$ to $\R\cup\{+\infty\}$.
\end{proof}

Based on this lemma, the following proposition states the uniqueness result. 
It can be easily seen in the above lemma that the Legendre transform of $S$ has a similar form as $S_H^*$. Actually, the following proposition is proved by equating the two functions $S^*$ and $S_H^*$.

\begin{prop}\label{prop:1}
The solution to the multi-time Hamilton-Jacobi equation is unique. Specifically, under the assumptions (H1) and (H2), if $S$ satisfies (S1)-(S2), then $S=S_H$.
\end{prop}

\begin{proof}
In the proof of this proposition, we first consider the case of single-time. Let $N=1$, and $H$ be the Hamiltonian.

From \cref{lem:1}, it is proved that $S^*(p,E^-) = J^*(p) + I_{V}(p,E^-)$, where $V = \{(p,E^-):\ E^-\leq -\tilde{H}(p)\}$ and $\tilde{H}$ is a convex function whose domain is the projection of $\dom S^*$ along $(\mathbf{0},1)$. Moreover, $\ri\dom J^* = \ri\dom \tilde{H}$ (note that the domains of $\tilde{H}$ and $f$ are the same).

First, we prove that $\tilde{H}(p)=H(p)$ for any $p\in \ri\dom \tilde{H}$ by contradiction. 
Assume there exists $p\in \ri\dom \tilde{H}$ such that $\tilde{H}(p)\neq H(p)$. Let $E^- = -\tilde{H}(p)$. Then, by \cref{prop7} and \cref{eqt:subgrad_indicator}, we deduce that
\begin{equation}\label{eqt:subdiff_Sstar}
\partial S^*(p,E^-) = \partial J^*(p)\times \{0\} + N_{V}(p,E^-)
= \partial J^*(p)\times \{0\} + \{t(v,1):\ v\in \partial \tilde{H}(p), t\geq 0\},
\end{equation}
where the last equality holds because $V$ is the reflection of $\epi \tilde{H}$. Here $N_{V}(p,E^-)$ denotes the normal cone of the set $V$ at $(p,E^-)$.
Let $x_0 \in \partial J^*(p)$, $t>0$ and $v\in \partial \tilde{H}(p)$. Denote $x=x_0+tv$. Then, by \cref{eqt:subdiff_Sstar} we have  $(x, t)\in \partial S^*(p,E^-)$, which implies $(p,E^-)\in \partial S(x,t)$. However, $E^- + H(p)= -\tilde{H}(p) + H(p)\neq 0$, hence the HJ equation \cref{eqt:H-J} does not hold at $(x,t)$, which is a contradiction. Therefore, $\tilde{H}=H$ when restricting to the relative interior of the domain of $\tilde{H}$, which implies 
\begin{equation*}
S^*(p,E^-) = J^*(p) + I\{E^-\leq -\tilde{H}(p)\} = J^*(p) + I\{E^-\leq -H(p)\} = S_H^*(p,E^-),
\end{equation*}
for any $p\in \ri\dom \tilde{H}$. 

Actually, the values of any convex lower semi-continuous function on the relative boundary of its domain is fully determined by the values in the relative interior. It is not hard to check that
\begin{equation*}
    \ri\dom S^* = \ri\dom S_H^* = \{(p,E^-):\ p\in \ri\dom J^*, E^- < -H(p)\}.
\end{equation*}
Hence, we have proved that $S^*$ and $S_H^*$ agree in the relative interior of the domain. Therefore, $S^*=S_H^*$ in the whole domain, which implies $S=S_H$ and gives the uniqueness of the convex solution to the single-time HJ equation.

\bigbreak
Then, we can consider the case of multi-time. Now, we assume $N>1$. It suffices to prove $S$ and $S_H$ coincide for any $x\in \Rn$ and any $t_1,\cdots, t_N>0$. Let $\alpha_1, \cdots, \alpha_N$ be arbitrary positive real numbers and denote $\alpha := (\alpha_1, \cdots, \alpha_N)$. Define $T(x,s) := S(x, s\alpha_1, \cdots, s\alpha_N)$ for any $x\in\Rn$ and $s\geq 0$. Then $T\in \Gamma_0(\R^{n+1})$.
We can compute the gradient of $T$ with respect to $s$ for any $x\in\Rn$ and $s>0$ using chain rule and the assumption that $S$ satisfies the multi-time HJ equation \cref{eqt:H-J} to obtain
\begin{equation*}
\frac{\partial T(x,s)}{\partial s} = \sum_{j=1}^N \alpha_j \frac{\partial S(x,s\alpha)}{\partial t_j} = -\sum_{j=1}^N \alpha_j H_j(\nabla_x S (x,s\alpha))
= -\sum_{j=1}^N \alpha_j H_j(\nabla_x T (x,s)).
\end{equation*}
It is easy to check that $T$ satisfies the initial condition given by $J$, i.e. $T(x,0) = J(x)$ for any $x\in\Rn$.
Hence, $T$ is a solution to the single-time HJ equation with Hamiltonian $H = \sum_{j=1}^N \alpha_jH_j$, which is finite-valued, 1-coercive and strictly convex. Therefore, for the single-time HJ equation, the conditions (H1)-(H2) and (S1)-(S2) are satisfied.
Then, the solution $T$ is unique and equal to the Hopf formula with respect to the Hamiltonian $H$. Hence, for any $x\in\Rn, s>0$ and any $\alpha_1, \cdots, \alpha_N >0$, we have
\begin{equation*}
S(x, s\alpha_1, \cdots, s\alpha_N) = (J^* + sH)^*(x) = \left(J^* + \sum_{j=1}^N s\alpha_jH_j\right)^*(x) = S_H(x, s\alpha_1, \cdots, s\alpha_N). 
\end{equation*}
Therefore, $S=S_H$ in the relative interior of the domain, which implies $S=S_H$ in the whole space, because of the lower semi-continuity of $S$ and $S_H$. The uniqueness of the solution to the multi-time HJ equation follows.
\end{proof}

One can actually apply the above arguments to weaker assumptions and obtain a generalized result, which is stated in the following corollary. 
In this generalized result, it is possible that the solution $S$ is not a classical solution, hence the subgradients of $S$, instead of the gradients, are assumed to satisfy the HJ equation, which is a natural generalization of the classical solution when we want to consider the solution which is convex and lower semi-continuous.

\begin{cor} \label{cor:uniq}
Let $J\in\Gamma_0(\Rn)$, and $H_1, H_2, \cdots, H_N$ be arbitrary extended-valued functions defined on $\Rn$. Assume there exists a function $S \in \Gamma_0(\Rn\times [0,+\infty)^N)$ satisfying:
\begin{itemize}
\item[(i)]  If $p\in\Rn$ and $E_1^-,\cdots, E_N^-\in \R$ satisfy $(p,E_1^-,\cdots, E_N^-)\in \partial S(x,t_1,\cdots, t_N)$ for some $x\in\Rn$ and $t_1,\cdots, t_N>0$, then $E_j^- + H_j(p)=0$ for any $j=1,\cdots, N$.
\item[(ii)] $S(x,0,\cdots, 0) = J(x)$ for any $x\in\Rn$. 
\end{itemize}
Then, the following statements hold:
\begin{itemize}
\item[1.] For the case of single time, i.e. $N=1$, denote $H=H_1$ to be the Hamiltonian. If there exists $x\in\Rn$, $t>0$ such that $S(x,t)\neq +\infty$, then $S$ is unique and $S= F^*$, where $F$ is defined by
\begin{equation}\label{eqt:cor43_eqt1}
F(p,E^-) := J^*(p) + I\{E^-\leq -H(p)\} + I\{p\in \ri\dom J^*\},
\end{equation}
for any $p\in\Rn$ and $E^-\in \R$. Moreover, the restriction of $H$ on $\ri\dom J^*$ is finite-valued and convex.
\item[2.] For the multi-time case, i.e. $N>1$, if $\tilde{S}$ is another function satisfying the assumptions (i)-(ii) with $\ri\dom \tilde{S} = \ri\dom S$, then $\tilde{S} = S$. In other words, the solution is unique when the relative interior of the domain is given.
\end{itemize}
\end{cor}
\update{
\begin{proof}
The proof of this corollary is similar to the proof of \cref{prop:1}, so we just give a brief sketch here. First, we adjust the proof of \cref{lem:1} by changing the gradients of $S$ to the subgradients of $S$. The argument still holds because we assume in (i) that the subgradients of $S$ satisfy the HJ equation. Then, we draw the same conclusion as in \cref{lem:1}. In other words, with the function $\tilde{H}$ defined in \cref{eqt:lem41_defHtilde}, we have
\begin{equation}\label{eqt:cor43_eqt2}
S^*(p,E^-) = J^*(p) + I\{E^-\leq -\tilde{H}(p)\}.
\end{equation}
Also, the part of $N=1$ in the proof of \cref{prop:1} still holds. So we derive that the two functions $\tilde{H}$ and $H$ coincide in the relative interior of $\dom J^*$. Together with \cref{eqt:cor43_eqt2}, we derive \cref{eqt:cor43_eqt1}, and hence the first statement in this corollary follows.

For the case when $N>1$, it suffices to prove that $S$ and $\tilde{S}$ coincide in the relative interior of the domain. Let $(y,t_1,\cdots,t_N)$ be an arbitrary point in $\ri\dom S$. It remains to prove that $S$ and $\tilde{S}$ are equal at the point $(y,t_1,\cdots,t_N)$.
Notice that we have $t_i > 0$ for any $i=1,\cdots,N$, then we can choose the positive number $\alpha_i$ in the proof of \cref{prop:1} to be $t_i$ for any $i$. As in the proof of \cref{prop:1}, we define the functions $T$ and $\tilde{T}$ by
\begin{equation*}
    \begin{split}
        T(x,s):= S(x,s\alpha_1,\cdots, s\alpha_N), \quad \text{ and }\quad 
        \tilde{T}(x,s):= \tilde{S}(x,s\alpha_1,\cdots, s\alpha_N),
    \end{split}
\end{equation*}
for any $x\in\Rn$ and $s\geq 0$. Since there exists a point $(y,t_1,\cdots,t_N)$ in the relative interior of $\dom S$, one can easily check that the assumptions in \cite[Thm.XI.3.2.1]{convex_analy_book2} hold. Then, by \cite[Thm.XI.3.2.1]{convex_analy_book2}, the chain rule for the subgradients of $S$ holds. Similarly, the chain rule also holds for the subgradients of $\tilde{S}$. Therefore, the argument in the proof of \cref{prop:1} in the multi-time case remains valid by changing the gradients to the subgradients. As a result, we conclude that both $T$ and $\tilde{T}$ solve the single-time HJ equation with the Hamiltonian $\sum_{j=1}^N\alpha_jH_j$. Then, by the first statement in this corollary, we have $T\equiv \tilde{T}$, which implies that $S$ and $\tilde{S}$ coincide at the point $(y,t_1,\cdots,t_N)$, and the proof is complete.
\end{proof}
}

\section{A Regularization Method for the Degenerate Cases} \label{sec:degenerate}

In the previous two sections, we discussed the relation between some optimization problems and the multi-time HJ equations under the assumptions (H1) and (H2). In general, if those assumptions are not satisfied, some results may collapse. For example, if there is no strictly convex Hamiltonian, then the solution may be non-differentiable, which leads to the non-uniqueness of the maximizer $p$ (called momentum) in the Hopf formula \cref{eqt:hopf}. Also, the minimizer $u$ in the Lax formula \cref{eqt:lax} may be non-unique if the Hamiltonians are not differentiable. However, these are two common situations for optimization problems such as the decomposition models. In fact, any norm or indicator function is neither strictly convex nor differentiable. As a result, it is an important problem to select a meaningful momentum $p$ or minimizer $u$ in the solution set when it contains more than one element. 

In this section, we propose a regularization method to select a unique momentum $p$ and a unique minimizer $u$ simultaneously, and provide the representation formulas for both selected quantities by using the results stated in the previous sections. Intuitively, to select a minimizer $u$, we modify the degenerate term by adding $\lambda H$ to it where $\lambda$ is a positive parameter and $H$ is a differentiable function satisfying (H1). When $\lambda$ approaches zero, the minimizer of the modified problem will converge to the unique minimizer $\bar{u}$ in the solution set of the original problem which minimizes the function $H$. The procedure to select $p$ is the same except performing the inf-convolution with $\lambda H^*(\cdot/\lambda)$ to the degenerate term instead of the addition of $\lambda H$. 

In the literature, the special case selecting the momentum $p$ using inf-convolution with $\|\cdot\|^2/(2\lambda)$ is well-known as Moreau-Yosida approximation, which is introduced, for instance, in \cite[Thm.2, p.144]{Aubin:1984:DIS:577434} and \cite[Thm.3.1, p.54]{brezis1973operateurs}.
\update{Generally, a Moreau-Yosida based regularization method usually selects a unique minimizer $u$ only or a momentum $p$ only, but not both.}
Our contribution here is that we consider the primal problem and the dual problem simultaneously. In other words, one can select the momentum $p$ and the minimizer $u$ at the same time using our method. This analysis can be adapted easily to other decomposition models with more degenerate terms. Moreover, one can also use the same procedure with other function $H$ \update{or even use two different functions in the two added terms. One alternative choice is $\|\cdot\|_\alpha^\alpha / \alpha$ for any $\alpha >1$, for example. In fact, if $H$ is chosen to be any non-negative, finite-valued, 1-coercive, differentiable and strictly convex function, the statements in this section still hold.
To be specific, the proofs of \cref{lem:uniq_bdd}, \cref{lem:cluster_grad} and \cref{prop:conv_general} hold after subtle adjustment, and one can use subdifferential calculus to prove \cref{lem:cluster_pt}. In this paper, for simplicity, we mainly focus on the quadratic regularization terms, which are usually preferred in practice because of the simplicity and efficiency of numerical implementation.}

\bigbreak
Now, we focus on a specific decomposition model, and the regularization function $H$ is chosen to be $\|\cdot\|_2^2/2$. Some other models can be analyzed using similar arguments. Let $\|\cdot\|$ and $\normiii{\cdot}$ be two arbitrary norms whose dual norms are denoted as $\|\cdot\|_*$ and $\normiii{\cdot}_*$. \update{In fact, all the results remain valid if $\|\cdot\|$ and $\normiii{\cdot}$ are two semi-norms, in which case the corresponding dual norms $\|\cdot\|_*$ and $\normiii{\cdot}_*$ are finite in some subspaces and equal to $+\infty$ otherwise.} The set of minimizers is defined as follows
\begin{equation*}
U(x,t) := \argmin_{u\in \Rn} \|u\| + I\{\normiii{x-u}_*\leq t\}.
\end{equation*}
We can regard the minimal value as a solution to the HJ equation given by the Lax formula with spatial variable $x\in\Rn$ and time variable $t >0$ and define
\begin{equation*} 
S(x,t) := \min_{u\in \Rn} \|u\| + I\{\normiii{x-u}_*\leq t\}.
\end{equation*}
Note that in the corresponding HJ equation, the initial function is $\|\cdot\|$ and the Hamiltonian is $\normiii{\cdot}$, hence the assumption (H1) is not satisfied. As a result, we need to apply the regularization method in this example. 
For simplicity we also use $F_1$, $F_2$ to denote these two norms, then $F_2^*(y)=I\{\normiii{y}_*\leq \update{t}\}$. We assume $t=1$ and drop the variable $t$ in the remainder of this section because the variation of $t$ is not considered in this problem. Then, we can rewrite the problem as the following
\begin{equation}\label{eqt:original_S}
\begin{split}
U(x) &= \argmin_{u\in \Rn} F_1(u) + F_2^*(x-u), \\
S(x) &= \min_{u\in \Rn} F_1(u) + F_2^*(x-u).
\end{split}
\end{equation}
\update{In fact, there are in practice some useful models in the literature which can fit in this form. Now, we give two examples. In what follows, we use $\|\cdot\|_{TV}$, $\|\cdot\|_E$ and $\|\cdot\|_G$ to denote the discrete total variation semi-norm, the discrete $E-$norm and the discrete $G-$norm, respectively.
First, in \cite{aujol2003image,aujol2005image}, it is shown that the Meyer's model in the following form
\begin{equation*}
    \argmin_{u\in \Rn} \|u\|_{TV} + \alpha \|x-u\|_G
\end{equation*}
is equivalent to 
\begin{equation*}
    \argmin_{u\in \Rn} \|u\|_{TV} + I\{\|x-u\|_G\leq \beta\},
\end{equation*}
for some suitable positive parameter $\beta$. In this example, both $F_1$ and $F_2$ are the discrete total variation because the discrete $G-$norm is the dual norm of $\|\cdot\|_{TV}$. 
Similarly, another Meyer's model stated as follows
\begin{equation*}
    \argmin_{u\in \Rn} \|u\|_{TV} + \alpha \|x-u\|_E
\end{equation*}
is equivalent to
\begin{equation*}
    \argmin_{u\in \Rn} \|u\|_{TV} + I\{\|x-u\|_{E}\leq \beta\},
\end{equation*}
for some suitable positive parameter $\beta$ \cite{aujol2005dual}. In this example, the functions $F_1$ and $F_2$ are the discrete total variation and the dual norm of the discrete $E-$norm, respectively.}

As mentioned above, we apply two operators to the function $F_1$ and obtain its approximation
\begin{equation}\label{eqt:def_relaxedF}
F_{1,\lambda,\mu} := \left(F_1 + \frac{\lambda}{2}\|\cdot\|_2^2\right)\conv \frac{1}{2\mu}\|\cdot\|_2^2,    
\end{equation}
where $\lambda, \mu>0$ are small regularization parameters. Here, we choose to modify the function $F_1$, but one may instead apply the operators to the function $F_2$ and the analysis is similar.
Then, the problem reads
\begin{equation} \label{eqt:relaxed}
\begin{split}
u_{\lambda,\mu}(x) &= \argmin_{u\in \Rn} F_{1,\lambda,\mu}(u) + F_2^*(x-u),\\
S_{\lambda,\mu}(x) &= \min_{u\in \Rn} F_{1,\lambda,\mu}(u) + F_2^*(x-u).
\end{split}
\end{equation}
We expand the inf-convolution to get
\begin{equation}\label{eqt:P2}
\begin{split} 
&u_{\lambda,\mu} = x - w_{\lambda,\mu},\\
&(v_{\lambda,\mu}, w_{\lambda,\mu}) := \argmin_{v,w\in \Rn} F_1(v) + \frac{\lambda}{2}\|v\|_2^2 + F_2^*(w) + \frac{1}{2\mu}\|x-v-w\|_2^2,\\
&S_{\lambda,\mu} := \min_{v,w\in \Rn} F_1(v) + \frac{\lambda}{2}\|v\|_2^2 + F_2^*(w) + \frac{1}{2\mu}\|x-v-w\|_2^2.
\end{split}
\end{equation}
Here and later in this section, we omit the variable $x$ when there is no ambiguity.

By introducing the quadratic terms, the uniqueness of $(v_{\lambda,\mu}, w_{\lambda,\mu})$ and the differentiability of $S_{\lambda, \mu}$ are guaranteed. When the parameters $\lambda$ and $\mu$ converge to zero in a comparable rate, the reasonable minimizer $u$ and momentum $p$ are selected. In fact, they are the elements with the minimal $l^2$ norms in the target sets $U(x)$ and $\partial S(x)$. The detailed statements are listed as follows.

\begin{lem} \label{lem:uniq_bdd}
\update{
For any $\lambda, \mu >0$, there is a unique minimizer $(v_{\lambda,\mu}, w_{\lambda,\mu})$ to the problem \cref{eqt:P2}. Moreover, for any positive constant $K$,
the sets $\{v_{\lambda,\mu}: \lambda,\mu\in(0,K)\}$ and $\{w_{\lambda,\mu}: \lambda,\mu\in(0,K)\}$ are bounded.
}
\end{lem}
\begin{proof}
It is easy to check that the objective function in \cref{eqt:P2} is 1-coercive and strictly convex, because of the 1-coercivity and strict convexity of the quadratic terms. Therefore, there exists a unique minimizer $(v_{\lambda,\mu}, w_{\lambda,\mu})$. 

Setting $w=x-v$ and $v\in U(x)$ in \cref{eqt:P2} and comparing it with \cref{eqt:original_S}, we obtain
\begin{equation*}
S_{\lambda, \mu}(x) \leq \min_{v\in U(x)} F_1(v) + \frac{\lambda}{2}\|v\|_2^2 + F_2^*(x-v) = S(x) + \lambda \min_{v\in U(x)} \frac{1}{2}\|v\|_2^2.
\end{equation*}
Denote $C := S(x) + \min_{v\in U(x)} \frac{\update{K}}{2}\|v\|_2^2$, \update{where $K$ is an arbitrary positive number as defined in the statement.} Then $C$ is independent of $\lambda$ and $\mu$, and $S_{\lambda, \mu}(x) \leq C$ when \update{$0<\lambda<K$}.
From this inequality and the definition of $S_{\lambda, \mu}(x)$ in \cref{eqt:P2}, we can derive a bound for $x-v_{\lambda,\mu}- w_{\lambda,\mu}$ that reads
\begin{equation}\label{eqt:bdduv}
\|x-v_{\lambda,\mu}- w_{\lambda,\mu}\|_2^2 \leq 2\mu S_{\lambda, \mu}(x) \leq 2C\mu \update{ \,\leq 2CK \text{ whenever $\mu<K$}}.
\end{equation}
Therefore, $v_{\lambda,\mu}+ w_{\lambda,\mu}$ is bounded \update{by the constant $\|x\|_2 + \sqrt{2CK}$ when we assume $\lambda, \mu \in (0,K)$}. 

Then, from the constraint given by the indicator function $F_2^*$ in the minimization problem \cref{eqt:P2}, we have $\normiii{w_{\lambda,\mu}}_*\leq 1$, which implies the boundedness of $w_{\lambda,\mu}$ because all the norms are equivalent in the finite-dimensional space $\Rn$.
As a result, $v_{\lambda,\mu}$ is also bounded \update{whenever $\lambda,\mu\in (0,K)$}. Then the conclusion follows.
\end{proof}

\begin{lem} \label{lem:cluster_pt}
Let $v_{\lambda,\mu}$ and $w_{\lambda,\mu}$ be defined by \cref{eqt:P2}. Then, we have $\lim_{\lambda,\mu \to 0^+}
v_{\lambda,\mu} + w_{\lambda,\mu} = x.$
Any cluster point of $v_{\lambda,\mu}$ is also a cluster point of $u_{\lambda,\mu}$ and vice versa. Moreover, any cluster point of $u_{\lambda,\mu}$ and $v_{\lambda,\mu}$ is in $U(x)$.
\end{lem}

\begin{proof}
The convergence of $v_{\lambda,\mu}+ w_{\lambda,\mu}$ to $x$ follows from \cref{eqt:bdduv}. Since $u_{\lambda,\mu} = x - w_{\lambda,\mu}$, any cluster point of $u_{\lambda,\mu}$ is also a cluster point of $v_{\lambda,\mu}$ and vice versa.
It remains to show that any cluster point of $v_{\lambda,\mu}$ is in $U(x)$.

By the definition of $(v_{\lambda,\mu}, w_{\lambda,\mu})$, we have
\begin{equation}\label{eqt:hopf_vw}
\begin{split}
(v_{\lambda,\mu}, w_{\lambda,\mu}) &= \argmin_{v,w\in \Rn} F_1(v) + \frac{\lambda}{2}\|v\|_2^2 + I\{\normiii{w}_*\leq 1\} + \frac{1}{2\mu}\|x-v-w\|_2^2\\
&= \argmin_{v,w\in \Rn}\  \mu F_1(v) + \frac{\lambda\mu}{2}\|v\|_2^2 + I\{\normiii{w}_*\leq 1\} + \frac{1}{2}\|x-v-w\|_2^2\\
& = \argmax_{v,w\in \Rn}\  \langle x, v\rangle + \langle x, w\rangle - \left(\frac{1}{2}\|v+w\|_2^2 + F_2^*(w) + \mu F_1(v) + \frac{\lambda\mu}{2}\|v\|_2^2 \right),
\end{split}
\end{equation}
where we first multiply the objective function by $\mu$ and then expand the quadratic term. \update{Recall that any indicator function is invariant under multiplication with a positive constant, hence we obtain $I\{\normiii{w}_*\leq 1\} = \mu I\{\normiii{w}_*\leq 1\}$ and the second equality in \cref{eqt:hopf_vw} follows.}
The last maximization problem in \cref{eqt:hopf_vw} is in the form of Hopf formula. The corresponding multi-time HJ equation with time variables $\mu$ and $\nu=\lambda \mu$ is given by
\begin{equation} \label{eqt:HJeqt}
\begin{cases}
\frac{\partial}{\partial\mu}\tilde{S}(y,z,\mu,\nu) + F_1(\nabla_y \tilde{S}(y,z,\mu,\nu)) = 0, &y,z\in \Rn; \mu, \nu > 0;\\
\frac{\partial}{\partial\nu}\tilde{S}(y,z,\mu,\nu) +\frac{1}{2} \|\nabla_y \tilde{S}(y,z,\mu,\nu)\|_2^2 = 0, &y,z\in \Rn; \mu, \nu > 0;\\
\tilde{S}(y,z,0,0) = J(y,z), &y,z\in \Rn.
\end{cases}
\end{equation}
Here, $J$ is the l.s.c. convex function such that $J^*(v,w) = \frac{1}{2}\|v+w\|_2^2 + F_2^*(w)$.
Although the assumption (H1) is not satisfied, 
\update{by \cref{eqt:hopf_vw} and \cref{lem:uniq_bdd}, we know that the Hopf formula is well-defined in $\mathbb{R}^{n} \times \mathbb{R}^{n} \times [0,+\infty)\times [0,+\infty)$}. Moreover, the solution $\tilde{S}$ is the classical solution to the multi-time HJ equation \cref{eqt:HJeqt} and its spatial gradient equals $(v_{\lambda,\mu}, w_{\lambda,\mu})$. To be specific, we have
\begin{equation}\label{eqt:spatialgS}
(v_{\lambda,\mu}, w_{\lambda,\mu}) = \nabla_{y,z} \tilde{S}(x,x,\mu, \lambda\mu).
\end{equation}

Then, we want to apply the results in \cref{thm:conv_grad} (i) to prove that any cluster point of $v_{\lambda,\mu}$ is in $U(x)$. In fact, under the basic assumptions that $H_j, J\in\gmRn$ and the Hopf formula is well-defined, 
the proof of \cref{thm:conv_grad} (i) only requires the following statements: 
\setlist[enumerate]{leftmargin=.5in}
\setlist[itemize]{leftmargin=.5in}
\begin{itemize}
    \item[(a)] $\partial J(x,x)$ is non-empty;
    \item[(b)] the Hamiltonians are finite-valued;
    \item[(c)] $\tilde{S}$ is differentiable;
    \item[(d)] the spatial gradient $\nabla_{y,z} \tilde{S}(x,x,\mu, \lambda\mu)$ is bounded with all limit points in $\partial J(x,x)$.
\end{itemize}
\setlist[enumerate]{leftmargin=.25in}
\setlist[itemize]{leftmargin=.25in}
The statements (b) and (c) are obvious satisfied. 
It is straightforward to check $\partial J(x,x)\neq \emptyset$. Specifically, $(v,w)\in \partial J(x,x)$ iff $(x,x)\in \partial J^*(v,w)$.
By simple computation, $\partial J^*(v,w) = (v+w, v+w + \partial F_2^*(w))$. Then we obtain
\begin{equation}\label{eqt:subgJ}
(v,w)\in \partial J(x,x) \text{ iff } v+w = x \text{ and } \normiii{w}_*\leq 1.
\end{equation}
Such $v$ and $w$ always exist, hence $\partial J(x,x) \neq \emptyset$. 
As for the statement (d), the boundedness of $\nabla_{y,z} \tilde{S}(x,x,\mu, \lambda\mu)$ follows from \cref{eqt:spatialgS} and \cref{lem:uniq_bdd}. By \cref{eqt:bdduv}, $v_{\lambda,\mu}+ w_{\lambda,\mu}$ converges to $x$. Also, $\normiii{w_{\lambda,\mu}}_*\leq 1$ is given by the constraint imposed by $F_2^*$ in the minimization problem \cref{eqt:P2}. Together with \cref{eqt:spatialgS}, we can conclude that any limit point of $\nabla_{y,z} \tilde{S}(x,x,\mu, \lambda\mu)$, denoted as $(v,w)$, satisfies $v+w = x$ and $\normiii{w}_*\leq 1$. Hence, $(v,w)\in\partial J(x,x)$ by \cref{eqt:subgJ} and the statement (d) is proved.

Therefore, the conclusion of \cref{thm:conv_grad} (i) still holds although the assumption (H1) is not satisfied. As a result, for any cluster point $(\bar{v},\bar{w})$ of $(v_{\lambda,\mu}, w_{\lambda,\mu})$,
\begin{equation*}
\begin{split}
(\bar{v},\bar{w}) &\in \argmax_{(v,w) \in \partial J(x,x)} -F_1(v) 
= \argmin_{v+w=x,\ \normiii{w}_*\leq 1} F_1(v)
= \{(v,w):\ v\in U(x), w = x-v\},
\end{split}
\end{equation*}
where the last two equalities follow from \cref{eqt:subgJ} and the definition of $U(x)$ in \cref{eqt:original_S}.
In conclusion, any cluster point $\bar{v}$ of $v_{\lambda,\mu}$ is in $U(x)$.
\end{proof}

\begin{lem} \label{lem:cluster_grad}
\update{
For any $\lambda, \mu>0$, the function $S_{\lambda,\mu}$ defined in \cref{eqt:relaxed} is differentiable. Let $x\in\Rn$ and define $p_{\lambda,\mu} := \nabla S_{\lambda,\mu}(x)$. Then for any positive constant $K$, the set of gradients $\{p_{\lambda,\mu}: \lambda,\mu\in(0,K)\}$ is bounded.}
Moreover, as $\lambda$ and $\mu$ approach zero, any cluster point of $p_{\lambda,\mu}$ is in $\partial S(x)$.
\end{lem}

\begin{proof}
Rewriting the formula of $S_{\lambda,\mu}$ in \cref{eqt:P2}, we get
\begin{equation*}
    S_{\lambda,\mu} = \left(F_1 + \frac{\lambda}{2}\|\cdot\|_2^2\right) \conv F_2^* \conv \left(\frac{1}{2\mu}\|\cdot\|_2^2\right).
\end{equation*}
From straightforward computation, by \cref{prop4} and the definition of $(v_{\lambda,\mu},w_{\lambda,\mu})$ in \cref{eqt:P2}, we obtain
\begin{equation} \label{eqt:partial_S}
\begin{split}
\partial S_{\lambda,\mu}(x) &= \partial \left(F_1 + \frac{\lambda}{2}\|\cdot \|_2^2\right)(v_{\lambda,\mu}) \bigcap \partial F_2^*(w_{\lambda,\mu}) \bigcap \left\{\frac{1}{\mu}(x-v_{\lambda,\mu} - w_{\lambda,\mu}) \right\} \\
&= \left(\partial F_1(v_{\lambda,\mu}) + \lambda v_{\lambda,\mu}\right) \bigcap \partial F_2^*(w_{\lambda,\mu}) \bigcap \left\{\frac{1}{\mu}(x-v_{\lambda,\mu} - w_{\lambda,\mu}) \right\}.
\end{split}
\end{equation}
As a result, $\partial S_{\lambda,\mu}(x)$ contains at most one element. On the other hand, $S_{\lambda,\mu}$ is convex and finite-valued, which implies the subdifferential of $S_{\lambda,\mu}$ is non-empty.
Hence, $S_{\lambda,\mu}$ is differentiable and its gradient is given by
\begin{equation} \label{eqt:formula_p}
p_{\lambda,\mu} := \nabla S_{\lambda,\mu}(x)=
\frac{1}{\mu}(x-v_{\lambda,\mu} - w_{\lambda,\mu}).
\end{equation}

\update{Let $K$ be an arbitrary positive number. Now, we prove that there exists a constant $C$ such that $\|p_{\lambda,\mu}\|_2\leq C$ whenever $\lambda, \mu \in (0,K)$.
By \cref{eqt:partial_S,eqt:formula_p}, $p_{\lambda,\mu}$ is in the set $\partial F_1(v_{\lambda,\mu}) + \lambda v_{\lambda,\mu}$.
On the one hand, the subdifferential of the norm $F_1$ is always bounded. In other words, there exists a constant $C_1$ such that $\|s\|_2\leq C_1$ whenever $s\in \partial F_1(z)$ for some $z\in\R^n$. Then, we deduce that the set $\partial F_1(v_{\lambda,\mu})$ is bounded by $C_1$. On the other hand, according to \cref{lem:uniq_bdd}, there exists a constant $C_2$ such that $\|v_{\lambda,\mu}\|_2\leq C_2$ whenever $\lambda,\mu\in(0,K)$. Therefore,  $\{p_{\lambda,\mu}: \lambda, \mu\in(0,K)\}$ is bounded by $C_1+C_2 K$.
}

Let $p$ be a cluster point of $\{p_{\lambda,\mu}\}$. By taking a subsequence we can assume $\lambda_k$ and $\mu_k$ converge to zero and $p_k := p_{\lambda_k,\mu_k}$ converges to $p$. By \cref{lem:uniq_bdd}, $v_k := v_{\lambda_k,\mu_k}$ is bounded, hence we can assume $v_k$ converges to a point $u$ by taking a subsequence. Then, $w_k := w_{\lambda_k,\mu_k}$ converges to $x-u$ by \cref{lem:cluster_pt}. From \cref{eqt:partial_S}, we have
\begin{equation*}
    p_k\in \left(\partial F_1(v_k) + \lambda_k v_k\right) \cap \partial F_2^*(w_k).
\end{equation*}
Since the subdifferential operators $\partial F_1$ and $\partial F_2^*$ are continuous \cite[Prop.XI.4.1.1]{convex_analy_book2}, when $k$ goes to infinity, the above inclusion becomes
\begin{equation}\label{eqt:egp1}
    p\in \left(\partial F_1(u) + 0\cdot u\right) \cap \partial F_2^*(x-u) = \partial F_1(u) \cap \partial F_2^*(x-u).
\end{equation}
On the other hand, by \cref{prop4} and the definition of $S(x)$ and $U(x)$ in \cref{eqt:original_S}, we have
\begin{equation}\label{eqt:egp2}
    \partial S(x) = \partial F_1(\tilde{u}) \cap \partial F_2^*(x-\tilde{u}),
\end{equation}
for any $\tilde{u} \in U(x)$. Moreover, by \cref{lem:cluster_pt}, since $u$ is a cluster point of $v_k$, we can conclude that $u\in U(x)$. As a result, we can choose $\tilde{u}=u$ in \cref{eqt:egp2} and compare it with \cref{eqt:egp1} to conclude that $p\in \partial S(x)$.
\end{proof}

\begin{prop} \label{prop:conv_general}
Assume $\{\lambda_k\} \subset (0,+\infty)$ and $\{\mu_k\} \subset (0,+\infty)$  converge to zero and $\lim_{k\to +\infty} \frac{\lambda_k}{\mu_k} = c\in (0,+\infty)$. 
Then, the minimizer $u_k:=u_{\lambda_k,\mu_k}$ and the gradient $p_k:=\nabla S_{\lambda_k,\mu_k}(x)$ converge to the $l^2$ projections of zero onto the sets $U(x)$ and $\partial S(x)$, respectively. To be specific, 
\begin{equation*}
\lim_{k\to+\infty} u_k = \argmin_{u\in U(x)} \|u\|_2, \ \ \text{and }\ 
\lim_{k\to+\infty} p_k = \argmin_{p\in \partial S(x)} \|p\|_2.
\end{equation*}
\end{prop}

\begin{proof}
\update{
Define $H(\cdot) := \|\cdot\|_2^2/2$. We will use the general symbol $H$ to replace the quadratic function because this proof holds for a general finite-valued, 1-coercive, differentiable and strictly convex function $H$.
}
Note that the limit of $u_k$ is the same as the limit of $v_k$, hence we just need to prove the result for $v_k$ and $p_k$. Denote 
\begin{equation}\label{eqt:pfprop54_defubarpbar}
 \bar{u} := \argmin_{u\in U(x)} \update{H(u)}, \quad \text{ and }\quad \bar{p} := \argmin_{p\in \partial S(x)} \update{H(p)}.
\end{equation}
Since $v_k$ and $p_k$ are bounded, we can assume that $v_k$ converges to $u$ and $p_k$ converges to $p$ by taking a subsequence. Then it suffices to prove $u=\bar{u}$, $p=\bar{p}$.

\update{
By \cref{eqt:partial_S} and \cref{eqt:formula_p}, we have
\begin{equation}\label{eqt:pfprop54_eqtpk}
    p_k \in \left(\partial F_1(v_k) + \lambda_k \nabla H(v_k)\right) \bigcap \partial F_2^*(w_k) \bigcap \left\{\nabla H^*\left(\frac{x-v_k - w_k}{\mu_k} \right)\right\}.
\end{equation}
By \cref{prop3}, we deduce that $w_k\in \partial F_2(p_k)$ and $x-v_k-w_k = \mu_k \nabla H(p_k)$.
Together with \cref{eqt:pfprop54_eqtpk}, we obtain}
\begin{equation}\label{eqt:subgradinclusion}
\begin{split}
    &\update{p_k - \lambda_k \nabla H(v_k) \in \partial F_1(v_k);}\\
    &\update{x - \mu_k \nabla H(p_k) - v_k = w_k \in \partial F_2(p_k).}
\end{split}
\end{equation}
On the other hand, since $\bar{u}$ and $\bar{p}$ are the minimizer and momentum of the original problem \cref{eqt:original_S}, we have
\begin{equation} \label{eqt:orginal}
\bar{p} \in \partial F_1(\bar{u}) \cap \partial F_2^*(x-\bar{u}).
\end{equation}
Combining \cref{eqt:subgradinclusion} and \cref{eqt:orginal}, we obtain
\begin{equation*}
\begin{cases}
p_k - \lambda_k \update{\nabla H(v_k)} \in \partial F_1(v_k);\ \ \text{and }\ \ \bar{p} \in \partial F_1(\bar{u}).\\
x-\mu_k \update{\nabla H(p_k)} - v_k \in \partial F_2(p_k);\ \ \text{and}\ \ x-\bar{u}\in \partial F_2(\bar{p}).
\end{cases}
\end{equation*}
Since the subdifferential operators $\partial F_1$ and $\partial F_2$ are monotone, by \cref{eqt:subdiff_monotone}, we obtain
\begin{equation*}
\begin{cases}
\langle p_k - \lambda_k \update{\nabla H(v_k)} - \bar{p}, v_k - \bar{u}\rangle \geq 0;\\
\langle x-\mu_k \update{\nabla H(p_k)} - v_k - (x-\bar{u}), p_k - \bar{p}\rangle \geq 0.
\end{cases}
\end{equation*}
We sum up the two inequalities to get
\begin{equation*}
\begin{split}
0\ &\geq -\langle p_k - \lambda_k \update{\nabla H(v_k)} - \bar{p}, v_k - \bar{u}\rangle -
\langle x-\mu_k \update{\nabla H(p_k)} - v_k - (x-\bar{u}), p_k - \bar{p}\rangle\\
&=\lambda_k\langle \update{\nabla H(v_k)}, v_k - \bar{u}\rangle + \mu_k \langle \update{\nabla H(p_k)}, p_k - \bar{p}\rangle.
\end{split}
\end{equation*}

\update{
We divide the above inequality by $\mu_k$ and take the limit $k\to+\infty$ to obtain 
\begin{equation} \label{eqt:pfprop54_ineqtup}
0\geq c\langle \nabla H(u), u - \bar{u}\rangle + \langle \nabla H(p), p - \bar{p}\rangle,
\end{equation}
where the positive constant $c$ is defined in the statement of this proposition to be $c:= \lim_{k\to+\infty}\lambda_k / \mu_k$.}
From \cref{lem:cluster_pt} and \cref{lem:cluster_grad}, we know that $u\in U(x)$ and $p\in \partial S(x)$, hence \update{we have $H(u)\geq H(\bar{u})$ and $H(p)\geq H(\bar{p})$ by \cref{eqt:pfprop54_defubarpbar}. Taken together with \cref{eqt:pfprop54_ineqtup},
we obtain
\begin{equation}\label{eqt:ineqt}
0\geq c(H(\bar{u}) - H(u)) + H(\bar{p}) - H(p)
\geq c\langle \nabla H(u), \bar{u}-u\rangle + \langle \nabla H(p), \bar{p} - p\rangle \geq 0.
\end{equation}

As a result, the inequalities in \cref{eqt:ineqt} become equalities, which implies $H(u)=H(\bar{u})$ and $H(p)=H(\bar{p})$ because $c$ is positive by assumption. Therefore, we conclude that $u=\bar{u}$ and $p=\bar{p}$, since the minimizers in \cref{eqt:pfprop54_ineqtup} are unique.}
\end{proof}

In practice, if a model has non-unique minimizers, then some existing optimization algorithms may fail to converge, in which case one may consider this modification procedure and perform the optimization algorithm to the modified problem to obtain a sequence converging to the selected minimizer. 
Here, for simplicity, we only demonstrate the method on a specific optimization problem whose objective function contains two parts \update{including one norm and one constraint}. In fact, this method works for more general cases, such as some other decomposition models with more degenerate parts. \update{Now, we give a numerical illustration for this proposed regularization method on the celebrated TVL1 model \cite{Alliney1997, Alliney1994Alg, aujol2006structure, Chan2005Aspects, Chan2006Alg, Duval2009TVL1, Nikolova2002Minimizers, Nikolova2004}.

\begin{table}[!ht]
     \begin{center}
     \begin{tabular}{  c | c | c | c | c}
     \hline
      &Example 1 & Example 2 & Example 3 & Example 4 \\ \hline
      \makecell{Original \\ Image}
      &
      \raisebox{-.5\height}{\includegraphics[width=0.18\textwidth]{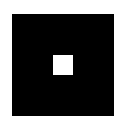}}
      & 
      \raisebox{-.5\height}{\includegraphics[width=0.18\textwidth]{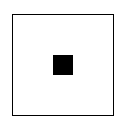}}
      &
      \raisebox{-.5\height}{\includegraphics[width=0.18\textwidth]{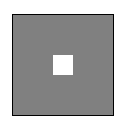}}
      & 
      \raisebox{-.5\height}{\includegraphics[width=0.18\textwidth]{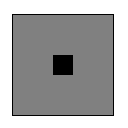}}
      \\ \hline
      \makecell{$v$ \\Component}
      &
      \raisebox{-.5\height}{\includegraphics[width=0.18\textwidth]{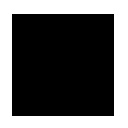}}
      & 
      \raisebox{-.5\height}{\includegraphics[width=0.18\textwidth]{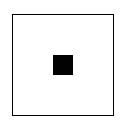}}
      &
      \raisebox{-.5\height}{\includegraphics[width=0.18\textwidth]{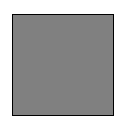}}
      & 
      \raisebox{-.5\height}{\includegraphics[width=0.18\textwidth]{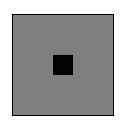}}
      \\ \hline
      \end{tabular}
      \caption{Numerical results of the TVL1 model with the proposed regularization method.}
      \label{tab:sec5_exp}
      \end{center}
      \end{table}
      
To be specific, the TVL1 model solves the following optimization problem
\begin{equation}\label{eqt:sec5_TVL1}
    U^{TVL1}(x) := \argmin_{u\in\Rn} \alpha \|u\|_{TV} + \|x-u\|_{1},
\end{equation}
where $\|\cdot\|_{TV}$ denotes the discrete total variation semi-norm defined in \cref{eqt:discreteTV}.
However, it is well-known that this minimization problem may have non-unique minimizers \cite{Chan2005Aspects, darbon2006image}. For instance, let $\Omega$ be the domain of an image and $\Omega_1$ be any small rectangle in $\Omega$ such that $2|\Omega_1| < |\Omega|$. Let $I$ be the set of indices whose corresponding pixels are in $\Omega_1$. Let $m_1, m_2$ be the numbers of pixels on the two adjacent sides of the small rectangle $\Omega_1$. In other words, there are $m_1m_2$ pixels in $\Omega_1$ and $2(m_1+m_2)$ pixels on the boundary of $\Omega_1$. Let $a$ and $b$ be two different real numbers in $[0,1]$ and set the discretized image $x$ as follows
\begin{equation*}
    x_{i,j} := \begin{cases}
    a, &\text{if }(i,j)\in I;\\
    b, &\text{if }(i,j)\not\in I.
    \end{cases}
\end{equation*}
Then, the minimizers of the TVL1 model \cref{eqt:sec5_TVL1} with $\alpha = (m_1m_2) / (2m_1+2m_2)$ are not unique.
Moreover, we have 
\begin{equation*}
U^{TVL1}(x) = \{\beta u_1 + (1-\beta) u_2:\ \beta\in[0,1]\},
\end{equation*}
where $u_1$ and $u_2$ are defined by
\begin{equation*}
    (u_1)_{i,j} := \begin{cases}
    \min\{a,b\}, &\text{if }(i,j)\in I;\\
    b, &\text{if }(i,j)\not\in I,
    \end{cases}
    \quad\text{ and }\quad 
    (u_2)_{i,j} := \begin{cases}
    \max\{a,b\}, &\text{if }(i,j)\in I;\\
    b, &\text{if }(i,j)\not\in I.
    \end{cases}
\end{equation*}
By applying the proposed regularization method, a unique minimizer is selected in this set of minimizers. To be specific, we solve the following problem
\begin{equation}\label{eqt:sec5_TVL1relaxed}
    \left(v^{TVL1}_{\lambda, \mu}(x), w^{TVL1}_{\lambda, \mu}(x)\right) := \argmin_{v,w\in\Rn} \alpha \|v\|_{TV} + \|w\|_{1} + \frac{\lambda}{2}\|v\|_2^2 + \frac{1}{2\mu}\|x-v-w\|_2^2.
\end{equation}
\updatenew{Note that the above model is related to models incorporating infinal convolution of $L^1$ and $L^2$ fidelity terms, which are used for mixed Gaussian and Salt \& Pepper noise image restoration, as proposed in \cite{Calatroni2017Infimal,Calatroni2019Analysis} for instance.}
Although this model is different from the example we give in \cref{eqt:original_S}, one can adjust the arguments to prove the same statements for this model. In other words, when the two parameters $\lambda$ and $\mu$ converge to zero in a comparable rate, the $v$-component $v^{TVL1}_{\lambda, \mu}(x)$ converges to the element $\bar{u}^{TVL1}(x)$ defined by
\begin{equation*}
    \bar{u}^{TVL1}(x) := \argmin_{u\in U^{TVL1}(x)} \|u\|_2 = u_1,
\end{equation*}
and the $w$-component converges to the residual $x-u_1$.
Numerically, we use a splitting method and the algorithm in \cite{chambolle.09.ijcv,darbon2006image,hochbaum.01.jacm} to solve the minimizer in \cref{eqt:sec5_TVL1relaxed} when $\lambda = \mu = 0.01$. We test the regularization method on the four images shown in the first row in \cref{tab:sec5_exp}, and the corresponding $v$-components are shown in the second row.
}

\section{Conclusion}
\label{sec:conclusion}
In this paper, we provide connections between multi-time Hamilton-Jacobi equations and some optimization problems such as the decomposition models in image processing. To be specific, we show a representation formula for the minimizers $u_j$ and clarify the connection between the minimizers $u_j$ and the spatial gradient $p$ of the minimal values. Moreover, we also study the variational behaviors of the momentum $p$ and the velocities $\frac{u_j}{t_j}$. It turns out that their limits solve two optimization problems which are dual to each other. In addition, we provide a new perspective from convex analysis to prove the uniqueness of the convex solution to the multi-time Hamilton-Jacobi equation, taking advantage of the convexity assumptions to overcome the difficulty that the functions can take the value $+\infty$. At last, we demonstrate a regularization method to modify the decomposition models which have non-unique minimizers.

In this work, we consider the optimization problems which can be written in the form of Lax formula \cref{eqt:lax}. Hence, we assume the observed data $x$ is the summation of different components $\{u_j\}$. We do not consider non-additive perturbation models such as \cite{Auslender:2006:IGP:1113209.1113255, Eckstein:1993:NPP:2781820.2781832, Teboulle:1997:CPA:588889.589048}. However, our analysis actually covers a wide range of decomposition models with additive noise and the results can be easily extended to vector-valued images such as color images.

\bibliographystyle{siam}
\bibliography{references}

\begin{thebibliography}{10}

\bibitem{Acar1994Analysis}
{\sc R.~Acar and C.~R. Vogel}, {\em Analysis of bounded variation penalty
  methods for ill-posed problems}, Inverse Problems, 10 (1994), pp.~1217--1229.

\bibitem{Allard2008TotalI}
{\sc W.~Allard}, {\em Total variation regularization for image denoising, {I}.
  geometric theory}, SIAM Journal on Mathematical Analysis, 39 (2008),
  pp.~1150--1190.

\bibitem{Allard2008TotalII}
\leavevmode\vrule height 2pt depth -1.6pt width 23pt, {\em Total variation
  regularization for image denoising, {II}. examples}, SIAM Journal on Imaging
  Sciences, 1 (2008), pp.~400--417.

\bibitem{Allard2009TotalIII}
\leavevmode\vrule height 2pt depth -1.6pt width 23pt, {\em Total variation
  regularization for image denoising, {III}. examples.}, SIAM Journal on
  Imaging Sciences, 2 (2009), pp.~532--568.

\bibitem{Alliney1997}
{\sc S.~Alliney}, {\em A property of the minimum vectors of a regularizing
  functional defined by means of the absolute norm}, IEEE Transactions on
  Signal Processing, 45 (1997), pp.~913--917.

\bibitem{Alliney1994Alg}
{\sc S.~{Alliney} and S.~A. {Ruzinsky}}, {\em An algorithm for the minimization
  of mixed $l_1$ and $l_2$ norms with application to bayesian estimation}, IEEE
  Transactions on Signal Processing, 42 (1994), pp.~618--627.

\bibitem{aubert.02.book}
{\sc G.~Aubert and P.~Kornprobst}, {\em Mathematical Problems in Image
  Processing}, Springer-Verlag, 2002.

\bibitem{Aubin:1984:DIS:577434}
{\sc J.~P. Aubin and A.~Cellina}, {\em Differential Inclusions: Set-Valued Maps
  and Viability Theory}, Springer-Verlag, Berlin, Heidelberg, 1984.

\bibitem{aujol2003image}
{\sc J.-F. Aujol, G.~Aubert, L.~Blanc-F{\'e}raud, and A.~Chambolle}, {\em Image
  decomposition application to {SAR} images}, in Scale Space Methods in
  Computer Vision, L.~D. Griffin and M.~Lillholm, eds., Berlin, Heidelberg,
  2003, Springer Berlin Heidelberg, pp.~297--312.

\bibitem{aujol2005image}
\leavevmode\vrule height 2pt depth -1.6pt width 23pt, {\em Image decomposition
  into a bounded variation component and an oscillating component}, Journal of
  Mathematical Imaging and Vision, 22 (2005), pp.~71--88.

\bibitem{aujol2005dual}
{\sc J.-F. Aujol and A.~Chambolle}, {\em Dual norms and image decomposition
  models}, International Journal of Computer Vision, 63 (2005), pp.~85--104.

\bibitem{AUJOL20061004}
{\sc J.-F. Aujol and T.~F. Chan}, {\em Combining geometrical and textured
  information to perform image classification}, Journal of Visual Communication
  and Image Representation, 17 (2006), pp.~1004 -- 1023.

\bibitem{Aujol2006Constrained}
{\sc J.-F. Aujol and G.~Gilboa}, {\em Constrained and {SNR}-based solutions for
  {TV}-{H}ilbert space image denoising}, Journal of Mathematical Imaging and
  Vision, 26 (2006), pp.~217--237.

\bibitem{aujol2006structure}
{\sc J.-F. Aujol, G.~Gilboa, T.~Chan, and S.~Osher}, {\em Structure-texture
  image decomposition---modeling, algorithms, and parameter selection},
  International Journal of Computer Vision, 67 (2006), pp.~111--136.

\bibitem{AUJOL2006916}
{\sc J.-F. Aujol and S.~H. Kang}, {\em Color image decomposition and
  restoration}, Journal of Visual Communication and Image Representation, 17
  (2006), pp.~916 -- 928.

\bibitem{Auslender:2006:IGP:1113209.1113255}
{\sc A.~Auslender and M.~Teboulle}, {\em Interior gradient and proximal methods
  for convex and conic optimization}, SIAM Journal on Optimization, 16 (2006),
  pp.~697--725.

\bibitem{Bardi1997}
{\sc M.~Bardi and I.~Capuzzo-Dolcetta}, {\em Optimal Control and Viscosity
  Solutions of {H}amilton-{J}acobi-{B}ellman Equations}, Birkh{\"a}user Basel,
  1997.

\bibitem{bardi1984hopf}
{\sc M.~Bardi and L.~Evans}, {\em On {H}opf's formulas for solutions of
  {H}amilton-{J}acobi equations}, Nonlinear Analysis: Theory, Methods \&
  Applications, 8 (1984), pp.~1373 -- 1381.

\bibitem{barles1994solutions}
{\sc G.~Barles}, {\em Solutions de viscosit{\'e} des {\'e}quations de
  Hamilton-Jacobi}, Math{\'e}matiques et Applications, Springer-Verlag Berlin
  Heidelberg, 1994.

\bibitem{barles2001commutation}
{\sc G.~Barles and A.~Tourin}, {\em Commutation properties of semigroups for
  first-order {H}amilton-{J}acobi equations and application to multi-time
  equations}, Indiana University Mathematics Journal, 50 (2001),
  pp.~1523--1544.

\bibitem{Basdevant.07.book}
{\sc J.-L. Basdevant}, {\em Variational principles in physics}, Springer, New
  York, 2007.

\bibitem{Bauschke:2011:CAM:2028633}
{\sc H.~H. Bauschke and P.~L. Combettes}, {\em Convex Analysis and Monotone
  Operator Theory in Hilbert Spaces}, Springer Publishing Company,
  Incorporated, 2011.

\bibitem{Bertalmio2003}
{\sc M.~Bertalmio, L.~Vese, G.~Sapiro, and S.~Osher}, {\em Simultaneous
  structure and texture image inpainting}, IEEE Transactions on Image
  Processing, 12 (2003), pp.~882--889.

\bibitem{Borwein}
{\sc J.~Borwein and A.~S. Lewis}, {\em Convex Analysis and Nonlinear
  Optimization: Theory and Examples}, Springer-Verlag New York, 2006.

\bibitem{Bredies2010Total}
{\sc K.~Bredies, K.~Kunisch, and T.~Pock}, {\em Total generalized variation},
  SIAM Journal on Imaging Sciences, 3 (2010), pp.~492--526.

\bibitem{bressan2007introduction}
{\sc A.~Bressan and B.~Piccoli}, {\em Introduction to the mathematical theory
  of control}, vol.~2, American Institute of Mathematical Sciences (AIMS),
  Springfield, MO, 2007.

\bibitem{brezis1973operateurs}
{\sc H.~Brezis}, {\em Op{\'e}rateurs maximaux monotones et semi-groupes de
  contractions dans les espaces de Hilbert}, North-Holland Mathematics Studies,
  Notas de Matem{\'a}tica, North-Holland Publishing Company, 1973.

\bibitem{Brezis2019Remarks}
{\sc H.~Brezis}, {\em Remarks on some minimization problems associated with bv
  norms}, Discrete \& Continuous Dynamical Systems - A, 0 (2019), p.~1.

\bibitem{Burger2004Convergence}
{\sc M.~Burger and S.~Osher}, {\em Convergence rates of convex variational
  regularization}, Inverse Problems, 20 (2004), pp.~1411--1421.

\bibitem{Calatroni2017Infimal}
{\sc L.~Calatroni, J.~C. De~Los~Reyes, and C.-B. Schönlieb}, {\em Infimal
  convolution of data discrepancies for mixed noise removal}, SIAM Journal on
  Imaging Sciences, 10 (2017), pp.~1196--1233.

\bibitem{Calatroni2019Analysis}
{\sc L.~Calatroni and K.~Papafitsoros}, {\em Analysis and automatic parameter
  selection of a variational model for mixed {G}aussian and salt-and-pepper
  noise removal}, Inverse Problems, 35 (2019), p.~114001.

\bibitem{cardin2008commuting}
{\sc F.~Cardin and C.~Viterbo}, {\em Commuting {H}amiltonians and
  {H}amilton-{J}acobi multi-time equations}, Duke Math. J., 144 (2008),
  pp.~235--284.

\bibitem{Casas1999Regularization}
{\sc E.~Casas, K.~Kunisch, and C.~Pola}, {\em Regularization by functions of
  bounded variation and applications to image enhancement}, Applied Mathematics
  and Optimization, 40 (1999), pp.~229--257.

\bibitem{Caselles2007Discont}
{\sc V.~Caselles, A.~Chambolle, and M.~Novaga}, {\em The discontinuity set of
  solutions of the {TV} denoising problem and some extensions}, Multiscale
  Modeling \& Simulation, 6 (2007), pp.~879--894.

\bibitem{Caselles2011Regularity}
{\sc V.~Caselles, A.~Chambolle, and M.~Novaga}, {\em Regularity for solutions
  of the total variation denoising problem}, Rev. Mat. Iberoam., 27 (2011),
  pp.~233--252.

\bibitem{Caselles2015Total}
{\sc V.~Caselles, A.~Chambolle, and M.~Novaga}, {\em Total variation in
  imaging}, in Handbook of mathematical methods in imaging. {V}ol. 1, 2, 3,
  Springer, New York, 2015, pp.~1455--1499.

\bibitem{Chambolle2004}
{\sc A.~Chambolle}, {\em An algorithm for total variation minimization and
  applications}, Journal of Mathematical Imaging and Vision, 20 (2004),
  pp.~89--97.

\bibitem{Chambolle2010Introduction}
{\sc A.~Chambolle, V.~Caselles, D.~Cremers, M.~Novaga, and T.~Pock}, {\em An
  introduction to total variation for image analysis}, in Theoretical
  foundations and numerical methods for sparse recovery, vol.~9 of Radon Ser.
  Comput. Appl. Math., Walter de Gruyter, Berlin, 2010, pp.~263--340.

\bibitem{chambolle.09.ijcv}
{\sc A.~Chambolle and J.~Darbon}, {\em On total variation minimization and
  surface evolution using parametric maximum flows}, International Journal of
  Computer Vision, 84 (2009), pp.~288--307.

\bibitem{Chambolle2016Total}
{\sc A.~Chambolle, V.~Duval, G.~Peyr{\'{e}}, and C.~Poon}, {\em Total variation
  denoising and support localization of the gradient}, Journal of Physics:
  Conference Series, 756 (2016), p.~012007.

\bibitem{Chambolle1997}
{\sc A.~Chambolle and P.-L. Lions}, {\em Image recovery via total variation
  minimization and related problems}, Numerische Mathematik, 76 (1997),
  pp.~167--188.

\bibitem{Chan2005Aspects}
{\sc T.~Chan and S.~Esedoglu}, {\em Aspects of total variation regularized {L}1
  function approximation}, SIAM Journal on Applied Mathematics, 65 (2005),
  pp.~1817--1837.

\bibitem{Chan2006Alg}
{\sc T.~Chan, S.~Esedoglu, and M.~Nikolova}, {\em Algorithms for finding global
  minimizers of image segmentation and denoising models}, SIAM Journal on
  Applied Mathematics, 66 (2006), pp.~1632--1648.

\bibitem{chan2005image}
{\sc T.~Chan and J.~Shen}, {\em Image Processing and Analysis: Variational,
  PDE, Wavelet, and Stochastic Methods}, Other Titles in Applied Mathematics,
  Society for Industrial and Applied Mathematics (SIAM, 3600 Market Street,
  Floor 6, Philadelphia, PA 19104), 2005.

\bibitem{CHAN2007464}
{\sc T.~F. Chan, S.~Esedoglu, and F.~E. Park}, {\em Image decomposition
  combining staircase reduction and texture extraction}, Journal of Visual
  Communication and Image Representation, 18 (2007), pp.~464 -- 486.

\bibitem{Chavent1997Regularization}
{\sc G.~Chavent and K.~Kunisch}, {\em Regularization of linear least squares
  problems by total bounded variation}, ESAIM: Control, Optimisation and
  Calculus of Variations, 2 (1997), pp.~359--376.

\bibitem{Choksi2014Remarks}
{\sc R.~Choksi, I.~Fonseca, and B.~Zwicknagl}, {\em A few remarks on
  variational models for denoising}, Communications in Mathematical Sciences,
  12 (2014), pp.~843--857.

\bibitem{crandall1992user}
{\sc M.~G. Crandall, H.~Ishii, and P.-L. Lions}, {\em User's guide to viscosity
  solutions of second order partial differential equations}, Bulletin of the
  American mathematical society, 27 (1992), pp.~1--67.

\bibitem{Darbon1}
{\sc J.~Darbon}, {\em On convex finite-dimensional variational methods in
  imaging sciences and {H}amilton--{J}acobi equations}, SIAM Journal on Imaging
  Sciences, 8 (2015), pp.~2268--2293.

\bibitem{darbon2006image}
{\sc J.~Darbon and M.~Sigelle}, {\em Image restoration with discrete
  constrained total variation part {I}: Fast and exact optimization}, Journal
  of Mathematical Imaging and Vision, 26 (2006), pp.~261--276.

\bibitem{Dobson1996Analysis}
{\sc D.~Dobson and O.~Scherzer}, {\em Analysis of regularized total variation
  penalty methods for denoising}, Inverse Problems, 12 (1996), pp.~601--617.

\bibitem{Duval2009TVL1}
{\sc V.~Duval, J.~Aujol, and Y.~Gousseau}, {\em The {TVL1} model: A geometric
  point of view}, Multiscale Modeling \& Simulation, 8 (2009), pp.~154--189.

\bibitem{Duval2009}
{\sc V.~Duval, J.-F. Aujol, and L.~A. Vese}, {\em Projected gradient based
  color image decomposition}, in Scale Space and Variational Methods in
  Computer Vision, Berlin, Heidelberg, 2009, Springer Berlin Heidelberg,
  pp.~295--306.

\bibitem{Duval2010}
\leavevmode\vrule height 2pt depth -1.6pt width 23pt, {\em Mathematical
  modeling of textures: Application to color image decomposition with a
  projected gradient algorithm}, Journal of Mathematical Imaging and Vision, 37
  (2010), pp.~232--248.

\bibitem{Eckstein:1993:NPP:2781820.2781832}
{\sc J.~Eckstein}, {\em Nonlinear proximal point algorithms using {B}regman
  functions, with applications to convex programming}, Mathematics of
  Operations Research, 18 (1993), pp.~202--226.

\bibitem{ELAD2005340}
{\sc M.~Elad, J.-L. Starck, P.~Querre, and D.~Donoho}, {\em Simultaneous
  cartoon and texture image inpainting using morphological component analysis
  ({MCA})}, Applied and Computational Harmonic Analysis, 19 (2005), pp.~340 --
  358.
\newblock Computational Harmonic Analysis - Part 1.

\bibitem{evans1983differential}
{\sc L.~C. Evans and P.~E. Souganidis}, {\em Differential games and
  representation formulas for solutions of {H}amilton-{J}acobi-{I}saacs
  equations}, tech. rep., Mathematics Research Center, University of
  Wisconsin-Madison, 1983.

\bibitem{Fadili:2010:ieee}
{\sc M.~J. {Fadili}, J.~{Starck}, J.~{Bobin}, and Y.~{Moudden}}, {\em Image
  decomposition and separation using sparse representations: An overview},
  Proceedings of the IEEE, 98 (2010), pp.~983--994.

\bibitem{GARNETT200725}
{\sc J.~B. Garnett, T.~M. Le, Y.~Meyer, and L.~A. Vese}, {\em Image
  decompositions using bounded variation and generalized homogeneous {B}esov
  spaces}, Applied and Computational Harmonic Analysis, 23 (2007), pp.~25 --
  56.
\newblock Special Issue on Mathematical Imaging.

\bibitem{Gilles2007}
{\sc J.~Gilles}, {\em Noisy image decomposition: A new structure, texture and
  noise model based on local adaptivity}, Journal of Mathematical Imaging and
  Vision, 28 (2007), pp.~285--295.

\bibitem{gilles2009image}
\leavevmode\vrule height 2pt depth -1.6pt width 23pt, {\em Image decomposition:
  theory, numerical schemes, and performance evaluation}, Advances in Imaging
  and Electron Physics, 158 (2009), pp.~89--137.

\bibitem{Gilles2010}
{\sc J.~Gilles and Y.~Meyer}, {\em Properties of {BV}-{G} structures$+$textures
  decomposition models. application to road detection in satellite images},
  IEEE Transactions on Image Processing, 19 (2010), pp.~2793--2800.

\bibitem{Hintermuller2018Function}
{\sc M.~Hinterm\"{u}ller, M.~Holler, and K.~Papafitsoros}, {\em A function
  space framework for structural total variation regularization with
  applications in inverse problems}, Inverse Problems, 34 (2018), p.~064002.

\bibitem{HINTERMULLER2017Analytical}
{\sc M.~Hinterm\"{u}ller, K.~Papafitsoros, and C.~N. Rautenberg}, {\em
  Analytical aspects of spatially adapted total variation regularisation},
  Journal of Mathematical Analysis and Applications, 454 (2017), pp.~891 --
  935.

\bibitem{convex_analy_book1}
{\sc J.-B. Hiriart-Urruty and C.~Lemarechal}, {\em Convex Analysis and
  Minimization Algorithms I: Fundamentals}, vol.~305, Springer-Verlag Berlin
  Heidelberg, 1993.

\bibitem{convex_analy_book2}
\leavevmode\vrule height 2pt depth -1.6pt width 23pt, {\em Convex Analysis and
  Minimization Algorithms II: Advanced Theory and Bundle Methods}, vol.~306,
  Springer-Verlag Berlin Heidelberg, 1993.

\bibitem{hochbaum.01.jacm}
{\sc D.~S. Hochbaum}, {\em An efficient algorithm for image segmentation,
  {M}arkov random fields and related problems}, Journal of the {A}{C}{M}, 48
  (2001), pp.~686--701.

\bibitem{hopf1965generalized}
{\sc E.~Hopf}, {\em Generalized solutions of non-linear equations of first
  order}, Journal of Mathematics and Mechanics, 14 (1965), pp.~951--973.

\bibitem{imbert:2001:jnca}
{\sc C.~Imbert}, {\em Convex analysis techniques for {H}opf-{L}ax formulae in
  {H}amilton-{J}acobi equations}, Journal of Nonlinear and Convex Analysis. An
  International Journal, 2 (2001), pp.~333--343.

\bibitem{Le2005BMO}
{\sc T.~Le and L.~Vese}, {\em Image decomposition using total variation and
  div({BMO})}, Multiscale Modeling \& Simulation, 4 (2005), pp.~390--423.

\bibitem{lions1986hopf}
{\sc P.~L. Lions and J.-C. Rochet}, {\em Hopf formula and multitime
  {H}amilton-{J}acobi equations}, Proceedings of the American Mathematical
  Society, 96 (1986), pp.~79--84.

\bibitem{meyer2001oscillating}
{\sc Y.~Meyer}, {\em Oscillating Patterns in Image Processing and Nonlinear
  Evolution Equations: The Fifteenth Dean Jacqueline B. Lewis Memorial
  Lectures}, American Mathematical Society, Boston, MA, USA, 2001.

\bibitem{motta2006nonsmooth}
{\sc M.~Motta and F.~Rampazzo}, {\em Nonsmooth multi-time {H}amilton-{J}acobi
  systems}, Indiana University Mathematics Journal, 55 (2006), pp.~1573--1614.

\bibitem{Nikolova2002Minimizers}
{\sc M.~Nikolova}, {\em Minimizers of cost-functions involving nonsmooth
  data-fidelity terms. application to the processing of outliers}, SIAM Journal
  on Numerical Analysis, 40 (2002), pp.~965--994.

\bibitem{Nikolova2004}
{\sc M.~Nikolova}, {\em A variational approach to remove outliers and impulse
  noise}, Journal of Mathematical Imaging and Vision, 20 (2004), pp.~99--120.

\bibitem{Nikolova2004Weakly}
\leavevmode\vrule height 2pt depth -1.6pt width 23pt, {\em Weakly constrained
  minimization: Application to the estimation of images and signals involving
  constant regions}, Journal of Mathematical Imaging and Vision, 21 (2004),
  pp.~155--175.

\bibitem{OSV}
{\sc S.~Osher, A.~Solé, and L.~Vese}, {\em Image decomposition and restoration
  using total variation minimization and the {$H^{-1}$} norm}, Multiscale
  Modeling \& Simulation, 1 (2003), pp.~349--370.

\bibitem{plaskacz2002oleinik}
{\sc S.~Plaskacz and M.~Quincampoix}, {\em {O}leinik-{L}ax formulas and
  multitime {H}amilton-{J}acobi systems}, Nonlinear Analysis: Theory, Methods
  \& Applications, 51 (2002), pp.~957--967.

\bibitem{Ring2000Structural}
{\sc W.~Ring}, {\em Structural properties of solutions to total variation
  regularization problems}, ESAIM: Mathematical Modelling and Numerical
  Analysis, 34 (2000), pp.~799--810.

\bibitem{rochet1985taxation}
{\sc J.~Rochet}, {\em The taxation principle and multi-time {H}amilton-{J}acobi
  equations}, Journal of Mathematical Economics, 14 (1985), pp.~113 -- 128.

\bibitem{RockafellarConvexAnalysis}
{\sc R.~T. Rockafellar}, {\em Convex Analysis}, Princeton University Press,
  1970.

\bibitem{RockWets98}
{\sc R.~T. Rockafellar and R.~J.-B. Wets}, {\em Variational Analysis}, vol.~317
  of Grundlehren der Mathematischen Wissenschaften [Fundamental Principles of
  Mathematical Sciences], Springer-Verlag, Berlin, 1998.

\bibitem{rudin1992nonlinear}
{\sc L.~I. Rudin, S.~Osher, and E.~Fatemi}, {\em Nonlinear total variation
  based noise removal algorithms}, Phys. D, 60 (1992), pp.~259--268.

\bibitem{souganidis1985max}
{\sc P.~E. Souganidis}, {\em Max-min representations and product formulas for
  the viscosity solutions of {H}amilton-{J}acobi equations with applications to
  differential games}, Nonlinear Analysis: Theory, Methods \& Applications, 9
  (1985), pp.~217 -- 257.

\bibitem{Starck2004}
{\sc J.-L. Starck, M.~Elad, and D.~Donoho}, {\em Redundant multiscale
  transforms and their application for morphological component separation},
  Advances in Imaging and Electron Physics - ADV IMAG ELECTRON PHYS, 132
  (2004), pp.~287--348.

\bibitem{Starck2005}
{\sc J.~L. Starck, M.~Elad, and D.~L. Donoho}, {\em Image decomposition via the
  combination of sparse representations and a variational approach}, IEEE
  Transactions on Image Processing, 14 (2005), pp.~1570--1582.

\bibitem{Teboulle:1997:CPA:588889.589048}
{\sc M.~Teboulle}, {\em Convergence of proximal-like algorithms}, SIAM Journal
  on Optimization, 7 (1997), pp.~1069--1083.

\bibitem{tho2005hopf}
{\sc N.~Tho}, {\em {H}opf-{L}ax-{O}leinik type formula for multi-time
  {H}amilton-{J}acobi equations}, Acta Math. Vietnamica, 30 (2005),
  pp.~275--287.

\bibitem{Valkonen2015Jump}
{\sc T.~Valkonen}, {\em The jump set under geometric regularization. part 1:
  Basic technique and first-order denoising}, SIAM Journal on Mathematical
  Analysis, 47 (2015), pp.~2587--2629.

\bibitem{Vese2004}
{\sc L.~A. Vese and S.~J. Osher}, {\em Image denoising and decomposition with
  total variation minimization and oscillatory functions}, Journal of
  Mathematical Imaging and Vision, 20 (2004), pp.~7--18.

\bibitem{Vogel2002Computational}
{\sc C.~Vogel}, {\em Computational Methods for Inverse Problems}, Society for
  Industrial and Applied Mathematics, 2002.

\end{thebibliography}

\end{document}